\documentclass[11pt,a4paper]{article}
\usepackage[affil-it]{authblk}
\usepackage[numbers,sort&compress,square]{natbib}

\usepackage[UKenglish]{babel}

\usepackage{appendix}

\usepackage{comment}

\usepackage{csquotes} 

\usepackage{amsmath}
\usepackage{amsfonts}
\usepackage{amssymb}

\usepackage{graphicx}

\usepackage[font=small, format=hang, labelfont=bf]{caption}
\usepackage[font=footnotesize, format=hang, margin={0.15\textwidth,0.1\textwidth}]{subcaption}

\usepackage{microtype}

\usepackage{bookmark}

\usepackage{tikz}
\usepackage{pgfplots}
\usetikzlibrary{arrows}
\pgfplotsset{compat = 1.9}

\usepackage{amsthm}

\usepackage{ wasysym, pifont }
\usepackage[shortlabels]{enumitem}

\usepackage{anyfontsize}
\usepackage{pbox}
\providecommand{\keywords}[1]{{\footnotesize\textbf{\text{Keywords.}} \pbox[t]{.9\textwidth}{#1}}}

\DeclareMathOperator{\supp}{supp}

\newcommand{\sref}[2]{\hyperref[#2]{#1~\ref{#2}}}

\newtheorem{theorem}{Theorem}[section]
\newtheorem{lemma}[theorem]{Lemma}
\newtheorem{proposition}[theorem]{Proposition}

\newtheorem{corollary}[theorem]{Corollary}

\theoremstyle{definition}
\newtheorem{definition}[theorem]{Definition}
\newtheorem{notations}[theorem]{Notations}

\newtheorem{fact}{Fact}

\theoremstyle{remark}
\newtheorem{remark}[theorem]{Remark}
\newtheorem*{remark*}{Remark}
\newtheorem*{remarks*}{Remarks}

\numberwithin{equation}{section}

\newtheorem*{claim*}{Claim} 
\newtheorem*{rem}{Remark}

\newcommand{\nref}[1]{{\normalfont{\ref{#1}}}}

\renewcommand{\d}{\mathrm{d}}
\newcommand{\dd}{\mathrm{d}}
\newcommand{\Om}{\Omega}
\newcommand{\om}{\omega}

\newcommand{\ve}{\varepsilon}

\usepackage{mathtools}
\mathtoolsset{showonlyrefs} 

\hypersetup{
  bookmarksopen,
  bookmarksopenlevel=1,
  pdfborder={0 0 0},
  pdfdisplaydoctitle,
  pdfpagelayout={SinglePage},
  pdfstartview={FitH}
}
\usepackage[margin=2.75cm]{geometry}

\author[ ]{Jos\'e A.\;Carrillo\,\thanks{\tt carrillo@imperial.ac.uk}}
\author[ ]{Katharina Hopf\,\thanks{\texttt{k.hopf@warwick.ac.uk, hopf@wias-berlin.de} (corresponding author)}}
\author[ ]{Jos\'e L.\;Rodrigo\,\thanks{\tt j.rodrigo@warwick.ac.uk}}

\affil[$*$]{\footnotesize Department of Mathematics, Imperial College London}
\affil[$\dag$,$\ddag$]{\footnotesize Mathematics Institute, University of Warwick}

\date{}

\usepackage{pbox}
\usepackage{calc}
\providecommand{\MSC}[1]{{\footnotesize\textbf{\text{2010 Mathematics Subject Classification.}}  \pbox[t]{.9\textwidth}{#1}}}
\providecommand{\keywords}[1]{{\footnotesize\textbf{\text{Keywords.}} \pbox[t]{.8\textwidth}{#1}}}

\title{\texorpdfstring{\vspace{-1cm}}{}
  On the singularity formation and relaxation to equilibrium in 1D Fokker--Planck model with superlinear drift}

\begin{document} 
  
\newpage
\pagenumbering{arabic}
\maketitle

\vspace{-1cm}

\begin{abstract}  
  We consider a class of Fokker--Planck equations with linear diffusion and superlinear drift enjoying a formal Wasserstein-like gradient flow structure with convex mobility function. 
In the drift-dominant regime, the equations have a finite critical mass above which the
measure minimising the associated entropy functional displays a singular component. 
Our approach, which addresses the one-dimensional case, is based on a reformulation of the problem in terms of the pseudo-inverse distribution function.
    Motivated by the structure of the equation in the new variables, we establish a general framework for global-in-time existence, uniqueness and regularity of monotonic viscosity solutions to a
   class of nonlinear degenerate (resp.\ singular) parabolic equations, using as a key tool comparison principles and maximum arguments. 
   We then focus on a specific equation and
   study in more detail the regularity and dynamics of solutions.
   In particular, blow-up behaviour, formation of condensates (i.e.\ Dirac measures at zero) and long-time asymptotics are investigated. 
As a consequence, in the mass-supercritical case,  solutions will blow up in $L^\infty$ in finite time and---understood in a generalised, measure sense---they will eventually have condensate. 
We further show that the singular part of the measure solution does in general interact with the density and that condensates can be transient.
The equations considered are motivated by a model for bosons introduced by Kaniadakis and Quarati (1994), which has a similar entropy structure and a critical mass if $d\ge3$.   
\end{abstract}


{\scriptsize
\tableofcontents
}

\section{Introduction}\label{sec:intro}

In~\cite{kaniadakis_classical_1994}, Kaniadakis and Quarati introduced a nonlinear Fokker--Planck equation with quadratic drift as a model for the dynamics of the velocity distribution of a spatially homogeneous system of bosons.
The model is based on a direct modification of the transition probability rates governing the particle kinetics in order to account for the quantum effect.
The resulting equation, the so-called \textit{Kaniadakis--Quarati model for bosons} (KQ), states
\begin{align}\label{eq:KQ}
 \partial_tf&=\Delta_v f+\nabla_v\cdot(vf(1+ f)), \;\;\;t>0,\;v\in\mathbb{R}^d.
\end{align}
Here the variable $v$ represents velocity while $f(t,\cdot)\ge0$ denotes the particle density at time~$t$, whose integral over $\mathbb{R}^d$---its \textit{mass}---is formally preserved under the evolution.  
The evolution~\eqref{eq:KQ} is driven by the entropy-type functional 
\begin{align*}
  \mathcal{H}_1(\mu) =  \int_{\mathbb{R}^d}\left(\frac{|v|^2}{2}f+f\log(f)-(1+f)\log(1+f)\right)\dd v.
\end{align*}
Indeed, eq.~\eqref{eq:KQ} can formally be rewritten as a nonlinear continuity equation
\begin{align}\label{eq:1Cont}
  \partial_tf&=\nabla\cdot\left(h_1(f)\nabla\frac{\delta \mathcal{H}_1}{\delta f}[f]\right), \;\;\;t>0,\;v\in\mathbb{R}^d
\end{align}
with mobility function $h_1(s):=s(1+s)$. 
Thus, formally along solutions of~\eqref{eq:KQ} 
\begin{align}\label{eq:1ede}
  \frac{\dd}{\dd t}\mathcal{H}_1(f)=-\int_{\mathbb{R}^d}h_1(f)\left|\nabla\frac{\delta \mathcal{H}_1}{\delta f}[f]\right|^2\le 0.
\end{align}
A peculiarity of the above problem lies in the fact that boundedness in entropy
does not imply local equi-integrability and, in fact, in space dimension $d>2$ 
there exists a critical mass $m_c<\infty$ above which the mass constraint minimisation of $\mathcal{H}_1$ leads to Dirac measure concentrations (see~Theorem~\ref{thm:entrMin}).

Here, we are interested in generalisations of the above physically motivated problem allowing for a critical mass also in lower space dimensions.
For $\gamma>0$ we consider the functional 
$\mathcal{\widetilde H}_\gamma:\mathcal{M}_b^+(\mathbb{R}^d)\to\mathbb{R}\cup\{\infty\}$ defined by
\begin{align}
\mathcal{\widetilde H}_\gamma(\mu) = \int_{\mathbb{R}^d}\left(\frac{|v|^2}{2}\mu(\dd v)+\Phi_\gamma(f)\dd v\right),\qquad \mu=f\cdot\mathcal{L}^d+\mu_s,\; \mu_s\perp\mathcal{L}^d,
\end{align}
where  $\Phi_\gamma(f):=\frac{1}{\gamma}\int_0^f\log\left(\frac{s^\gamma}{1+s^\gamma}\right)\d s$ and hence $\Phi_\gamma''(f)=(h_\gamma(s))^{-1}:=\left(s(1+s^\gamma)\right)^{-1}$.
The minimisers of $\mathcal{\widetilde H}_\gamma$ at a fixed level of mass have been explicitly determined in~\cite{abdallah_minimisation_2011}:
\begin{theorem}[See~\cite{abdallah_minimisation_2011}, Thm~3.1]\label{thm:entrMin}
 Let $m\in(0,\infty)$. The functional $\mathcal{\widetilde H}_\gamma$ restricted to the set
 $$ \{\mu\in\mathcal{M}^+_b(\mathbb{R}^d):\int \mu=m\}$$
 has a unique minimiser $\mu^{(m)}_\infty$. 
 Letting
 \begin{align}\label{eq:ss1}
   f_{\infty,\theta}(v)=\left(\mathrm{e}^{\gamma\left(\frac{|v|^2}{2}+\theta\right)}-1\right)^{-\frac{1}{\gamma}},\quad\theta\ge0,
\end{align}
 $f_c:=f_{\infty,0}$ and $m_c:=\int_{\mathbb{R}^d}f_c(v)\,\dd v\in(0,\infty]$, the minimiser is given by 
 \begin{align}\mu^{(m)}_\infty=
  \begin{cases}
    f_{\infty,\theta}\cdot\mathcal{L}^d&\text{ if }m\le m_c,\text{ where }\theta\ge0\text{ is s.t. }\int f_{\infty,\theta}=m
    \\f_c\cdot\mathcal{L}^d+(m-m_c)\delta_0&\text{ if }m>m_c.
  \end{cases}
  \label{eq:minmeasL}
 \end{align}
\end{theorem}
\noindent In the $L^1$ supercritical regime $d>2$, where $m_c<\infty$, the problem of understanding the long-time dynamics of KQ~\eqref{eq:KQ} has remained largely open. 
Toscani~\cite{toscani_finite_2012} demonstrated that, for highly concentrated initial data or data with very large mass (above a threshold $\underline{m}\gg m_c$), solutions must blow up after finite time (in the sense that they cannot be extended to a global-in-time classical solution).
The proof is indirect---based on a virial type argument---and does not provide any insights in the nature of blow-up. Formal matched asymptotic expansions on the blow-up dynamics of isotropic solutions in the case $d=3$ can be found in an earlier paper~\cite{sopik_dynamics_2006}.
The $L^1$ critical case $d=2$ has recently been investigated in the ref.~\cite{canizo_fokkerplanck_2016}. Exploiting the fact that in this case the nonlinear equation~\eqref{eq:KQ} in isotropic form can be transformed into a linear Fokker--Planck equation  (by means of the Cole--Hopf transformation), the authors are able to prove global existence of classical solutions and relaxation to equilibrium for a large class of initial data. Global regularity in the non-radial case is obtained upon comparison with isotropic solutions.
In the $L^1$ subcritical case $d=1$ a formal study of the relaxation to equilibrium was performed in the ref.~\cite{carrillo_1d_2008}. The global existence of regular solutions, which in this work was only guaranteed for initial data lying below one of the steady states in the pointwise sense, can in fact be obtained for any sufficiently regular initial datum 
by means of a comparison argument for the distribution function. 
In summary, while in the $L^1$ subcritical and critical case solutions are globally regular and converge to the steady state of the same mass, in the supercritical regime there do exist solutions which become unbounded after finite time, but beyond this little is known in that case.

In this work we aim to study in the $L^1$ supercritical case in 1D the singularity formation and long-time dynamics of the following generalisation of eq.~\eqref{eq:KQ},~\eqref{eq:1Cont}:
\begin{align}\label{eq:befpCont} 
 \partial_tf&=\nabla\cdot\left(h_\gamma(f)\nabla\frac{\delta \mathcal{H}_\gamma}{\delta f}[f]\right), \;\;\;t>0,\;v\in\mathbb{R}^d,
\end{align}
where $\mathcal{H}_\gamma(f)=\mathcal{\widetilde H}_\gamma(f\cdot\mathcal{L}^d)$ and $h_\gamma(s)=s(1+s^\gamma)$ for $\gamma>0$. Eq.~\eqref{eq:befpCont} is formally equivalent to 
\begin{align}\label{eq:befpOrig} 
 \partial_tf&=\Delta f+\nabla\cdot(v\,h_\gamma(f)), \;\;\;t>0,\;v\in\mathbb{R}^d.
\end{align} 
The $L^1$ subcritical, critical resp.\;supercritical 
regimes for problem~\eqref{eq:befpCont},~\eqref{eq:befpOrig} are given by $\gamma<\frac{2}{d}$, $\gamma=\frac{2}{d}$ resp.\;$\gamma>\frac{2}{d}$. The rationale of this division is as follows: at high values of the density $f$ the linear part of the drift in eq.~\eqref{eq:befpOrig} becomes negligible; taking into account conservation of mass, the resulting approximate scale invariance formally leads to a critical threshold $\gamma=\frac{2}{d}$, above which the conservation law should be unable to preclude local explosions of $f$.
For $\gamma\le\frac{2}{d}$ no (finite) critical mass exists so that condensation cannot be expected.
 In this case methods similar to those employed for the analysis of eq.~\eqref{eq:KQ} apply, and our focus will thus be on the $L^1$ supercritical case. Still, the theory we develop remains valid in the critical case, where global regularity of solutions to the Cauchy problem will be a simple corollary of our results.

 In view of the common entropy structure in Thm~\ref{thm:entrMin}, we will refer to the generalisations~\eqref{eq:befpOrig} as \textit{bosonic Fokker--Planck equations} and interpret singular measures concentrated at $v=0$ as condensates---although the physical description involving bosons is meaningful only if $\gamma=1$ and $d=3$. With this comparison with the physical case in mind, we will talk about condensation referring to the formation of a Dirac Delta (condensate) in finite time.
 
\paragraph{Strategy of the proof and outline of the results.} Our approach to eq.~\eqref{eq:befpOrig} in 1D is motivated by its formal gradient flow structure~\eqref{eq:befpCont} and builds upon the hypothesis of mass conservation. It is based on a reformulation of the problem 
 in terms of the pseudo-inverse distribution function 
 $$ u( x)=\inf\left\{v:\int_{\{v'\le v\}}f(v')\,\dd v'\ge  x\right\}, \quad x\in(0,\|f\|_{L^1}),$$
 of $f$.
 Assuming that $f$ is a strictly positive classical solution of problem~\eqref{eq:befpOrig} with sufficient decay as $|v|\to\infty$ so that its mass $m:=\|f(t,\cdot)\|_{L^1}$ is constant in time,
the (pseudo-)inverse distribution function $u(t.\cdot)$ of $f(t,\cdot)$ satisfies
\begin{align*}
  \partial_tu - (\partial_ x u)^{-2}\partial_ x ^2u + u((\partial_ x u)^{-\gamma}+1)=0\quad\text{ in }\Om:=(0,T)\times(0,m).
\end{align*}
Following an idea in~\cite[Section~4]{carrillo_condensation_2016}, we multiply the last equation by~$(\partial_x u)^{\gamma}$ to obtain
\begin{align}\label{eq:invBefp}
  (\partial_xu)^\gamma\partial_tu - (\partial_xu)^{\gamma-2}\partial_x^2u + u(1+(\partial_xu)^\gamma)=0\quad\text{ in }\Om.
\end{align}
In Section~\ref{chp:framework} we will establish a framework for the existence, uniqueness and regularity of ($x$-monotonic) viscosity solutions $u$ of Cauchy--Dirichlet type problems associated with generalisations of equation~\eqref{eq:invBefp}.
In our framework the minimisers appearing in Theorem~\ref{thm:entrMin} will be admissible  solutions while for $\theta>0$ and $m>\|f_{\infty,\theta}\|_{L^1}$ measures of the form 
$$f_{\infty,\theta}\cdot\mathcal{L}^d + (m-\|f_{\infty,\theta}\|_{L^1})\delta_0$$
will be neither sub- nor supersolutions. In this way, the latter family is naturally ruled out as potential equilibria (which would not be the case when, e.g., considering distributional solutions of the original formulation~\eqref{eq:befpOrig} using test functions vanishing near the origin).

As a consequence of the above framework, see Corollary~\ref{cor:bosonicFP},
 we obtain global existence, uniqueness and Lipschitz continuity of viscosity solutions to eq.~\eqref{eq:invBefp} with $\gamma\ge2$ complemented with the 
following conditions on the parabolic boundary:
\begin{alignat}{2}
  u(0,x) &= u_0(x),	&&x\in(0,m),\label{eq:invIc}
  \\u(t,0)&=-R,\; u(t,m)=R, \qquad&&t>0,\label{eq:invLbc}
\end{alignat}
provided the non-decreasing function $u_0$ satisfies certain mild regularity conditions. 
In the original variables this formally corresponds to eq.~\eqref{eq:befpOrig} with $d=1$ and $\gamma\ge2$, posed on a centred interval of radius $R$ subject to no-flux b.c.:
\begin{alignat}{2}\label{eq:befpBdd} 
 \partial_tf&=\partial_r^2f+\partial_r(r\,h_\gamma(f)), \qquad &&t>0, \;r\in(-R,R), 
 \\ f(0,r)&=f_0(r),  &&r\in(-R,R), \label{eq:initBdd}
 \\ 0&=\partial_rf+rh_\gamma(f),    &&t>0, \;r\in\{-R,R\}. \label{eq:bc}
\end{alignat}
Eq.~\eqref{eq:befpBdd},~\eqref{eq:bc} has an analogous entropy structure and a statement completely analogous to Theorem~\ref{thm:entrMin} holds true. See also Remark~\ref{rem:emin}. 
Let us emphasize that the choice of the domain~$(-R,R)$ has been made for simplicity only and we do \textit{not} assume symmetry of initial data resp.\;solutions. 

In Section~\ref{chp:1dftc} we derive refined regularity properties for the viscosity solutions $u$ of eq.~\eqref{eq:invBefp}, which allows us
to deduce that in the original variables $u(t,\cdot)$ corresponds to a finite measure $\mu(t)$ of mass $m$ of the form
\begin{align*}
 \mu(t) = f(t,\cdot)\mathcal{L}^1 + x_p(t) \delta_0,
\end{align*}
where away from $r=0$ the density $f$ is a classical solution of problem~\eqref{eq:befpBdd}--\eqref{eq:bc} and satisfies $f(t,r)-f_c(r)=O(|r|^{1-2/\gamma}),\;|r|\ll1,$ whenever $f(t,\cdot)$ is unbounded at $r=0$. 
We obtain in particular continuity of $\partial_xu(t,\cdot)$, which yields the implication $[x_p(t)>0\Rightarrow \lim_{r\to0} f(t,r)=+\infty]$.
We then establish an entropy method along the measure $\mu(t)$ allowing us to determine the long-time asymptotic behaviour, viz.\;convergence to the entropy minimiser of mass $m$. 

In Section~\ref{sec:wholeLine} we show that the framework previously established for the problem posed on $(-R,R)$ can be extended to the equation on the whole line $\mathbb{R}$, formally corresponding to $R=\infty$ in eq.~\eqref{eq:befpBdd},~\eqref{eq:initBdd}.

As corollaries of our theory we obtain
the following: 
\begin{enumerate}
  \item[a)] {\bf Short-time regularity} of solutions for sufficiently regular data (Remark~\ref{rem:shorttimereg}).
  
  \item[b)] {\bf Finite-time blow-up, blow-up profile and condensation} 
  (formation of a Dirac measure at zero) \textit{whenever} $m>m_c$ (Cor.~\ref{cor:cond}). Finite-time singularities can occur for any value of the mass provided the initial density is sufficiently concentrated near the origin (Prop.~\ref{prop:tc}). Solutions exist globally after the first singularity occurs in the form of a condensation part (Dirac Delta at zero) and a smooth part with an integrable singularity at zero, whose profile is determined (Prop.~\ref{prop:uMmu}, Prop.~\ref{prop:profile}).
  
\item[c)] {\bf Long-time asymptotics} given by the minimisers of the entropy for all values of the mass via a variant of the entropy method in the new variables (Theorem~\ref{thm:FTCond}, Remark~\ref{rem:convC1b}). Eventual regularity \textit{whenever} $m<m_c$ (Cor.~\ref{cor:cond}) and global regularity for $m<m_c$ provided the initial density $f_0$ is sufficiently spread out (Prop.~\ref{prop:globalReg}).

\item[d)] {\bf Transient condensation and interaction of the singular with the regular part} of the solutions. In general, the size of the condensate component $t\mapsto x_p(t)$ is not monotonic and condensates may be transient (Prop.~\ref{prop:tc}, Cor.~\ref{cor:tc}).
\end{enumerate}

\paragraph{Conventions.}

Here, we provide a collection of definitions and notations used throughout this manuscript.
Let $R>0$. For a finite Borel measure $\nu$ on $[-R,R]$ we define the  \textit{cumulative distribution function (cdf)} $M$ associated with $\nu$ via
\begin{align}\label{eq:defCdfOfMeas}
 M(r)=\nu([-R,r]),\quad r\in[-R,R].
\end{align}
The cumulative distribution function of a function $f\in L^1(-R,R)$ is defined as the cdf associated with the measure $f\cdot\mathcal{L}^1$, where here $\mathcal{L}^1$ denotes the one-dimensional Lebesgue measure restricted to the interval $[-R,R]$.

Let $R,m>0$. Given a strictly increasing, right continuous function $M:[-R,R]\to[0,m]$ with $M(R)=m$, 
 its \textit{pseudo-inverse} $u:[0,m]\to[-R,R]$ is defined via 
\begin{align}\label{def:PsI}
  u( x)=\min\{r\in[-R,R]:M(r)\ge  x\},\quad x\in [0,m].
\end{align}
The function $u$ is well-defined and continuous, and 
satisfies $u(0)=-R$, $u(m)=R$ as well as 
$u( x)=r$ whenever $ x\in[M(r-),M(r)]$, $r\in[-R,R]$.

Additional notations are listed below:
\small
\begin{itemize}[label=\raisebox{0.25ex}{\tiny$\bullet$}]\itemsep.1em
 \item We let $\Omega:=I\times J:=(0,T)\times (0,m)$, where $0<T\le\infty$ and $0<m<\infty$. The parabolic boundary of $\Om$, denoted by $\partial_p\Om$, is defined as the set
 \begin{align*}
   \partial \Om\setminus\left(\{T\}\times[0,m]\right),
 \end{align*}
 where $\partial\Om$ denotes the topological boundary of $\Om$.
 This notation will be also be used for more general axis-aligned rectangles $\subset\mathbb{R}\times\mathbb{R}$. We refer to the subset $(0,T)\times\{0,m\}\subset\partial_p\Om$ as the \textit{lateral boundary} of $\Om$.
  \item For an interval $V\subset\mathbb{R}$, any \textit{measure} on $V$
  is understood to be a non-negative Borel measure, and we denote by $\mathcal{M}^+_b(V)$ the set of finite measures on $V$. 
 \item \textit{Test functions} are $C^1$ in time and $C^2$ in space (meaning that the first time derivative and the second spatial derivative exist and are in $C(\Omega)$).
 \item In general, for a function $u=u(x_1,\dots,x_N)$ the expressions $\partial_{x_i}u$ and $u_{x_i}$ for some $i\in\{1,\dots,N\}$ both denote the weak derivative (in the distributional sense) of the function $u$ in the $i^\mathrm{\,th}$ direction.
 The pointwise derivative of $u$ with respect to $x_i$ will be denoted by $^{(p)}\partial_{x_i}u$. 
 \item For a function $u:(a,b)\subset \mathbb{R}\to\mathbb{R}$ we denote by $u'$ its (weak) derivative.
\item For $d\in\mathbb{N}$ the expression $\mathrm{Sym}(d)$ denotes the space of symmetric $d\times d$ matrices with real components. 
 \item For $\alpha\in (0,1]$ and $U\subset \mathbb{R}^d$ 
we abbreviate $[u]_{C^{0,\alpha}(U)}:=\sup_{\overset{x,y\in U}{x\neq y}}\frac{|u(x)-u(y)|}{|x-y|^\alpha}$.
\item The $d$-dimensional Lebesgue measure on $\mathbb{R}^d$ is denoted by $\mathcal{L}^d$. We use the same symbol for its restriction to any Lebesgue measurable subset $U\subset\mathbb{R}^d$.
  \item For $V\subset\mathbb{R}^2$ open,
  we abbreviate $C^{1,2}_{x_1,x_2}(V)=\{u\in C^1(V):\partial_{x_2}^2u\in C(V)\}$.
 In this notation, $x_1$ will always represent the time variable.
 \item For $V\subset\mathbb{R}^2$ open and $\alpha\in(0,1]$, we let $H_{2+\alpha}(\bar V)$ denote the set of functions $u\in C^{1,2}_{t,x}(V)$ for which the quantities $\|u\|_{C^1(V)}$, $\|\partial_x^2u\|_{C(V)}$, $[\partial_x^2u]_\alpha$ and $[\partial_tu]_\alpha$ are finite, 
 where $$[v]_{\alpha}:=\sup_{\overset{(t,x),(s,y)\in V}{(t,x)\neq (s,y)}}\frac{|v(t,x)-v(s,y)|}{d_p((t,x),(s,y))^\alpha}$$
 and $d_p((t,x),(s,y)):=\max\{|t-s|^\frac{1}{2},|x-y|\}$.
  \item Unless stated otherwise, $L^p$ spaces are to be understood with respect to the Lebesgue measure, i.e.\;$L^p(U)=L^p(U,\mathcal{L}^d)$ if $U\subset\mathbb{R}^d$ is Lebesgue measurable.
  \item $L^1_+(U)=\{f\in L^1(U):f\ge0\;\text{almost everywhere}\}$.
 \item  $\mathrm{USC}(U)$ (resp.\;$\mathrm{LSC}(U)$) denotes the set of upper semicontinuous (resp.\;lower semicontinuous) real-valued functions on $U$.
\end{itemize}
\normalsize

\noindent Further notations will be introduced in the course of the text. See in particular Sec.~\ref{ssec:applications} (p.~\pageref{ssec:applications}~f.), Not.~\ref{not:Rmtheta}, Prop.~\ref{prop:uMmu}, Def.~\ref{rem:defuR} and Def.~\ref{def:MmuLine}.

\paragraph{Other models.}
There exist many other models in the literature related to Bose--Einstein condensation.
We only comment on a small selection.
Let us first mention a class of problems based on spatially homogeneous Boltzmann type equations 
such as the Boltzmann--Nordheim equation, which has been the subject of various studies (mainly within the framework of isotropic solutions), see e.g.~\cite{escobedo_quantum_2001, escobedo_asymptotic_2004, escobedo_finite_2015,bandyopadhyay_blow-up_2015,lu_boltzmann_2013,lu_long_2018} 
and references therein. 
Formally more similar to our problem is a model due to  Kompaneets~\cite{kompaneets_establishment_nodate} for 
the relaxation to thermal equilibrium of the photon distribution
in a homogeneous plasma. 
A special case of this model yields a nonlinear Fokker--Planck type equation in $(0,\infty)$, versions of which have been studied in~\cite{escobedo_nonlinear_1998} and~\cite{levermore_global_2016}. In this model the break-down of the zero-flux boundary condition at $x=0$ is interpreted as the onset of a condensate. 
The phenomenon of condensation is, however, rather different from the one observed in our bosonic FP equations,
where in general the condensate does interact with the density and near the condensate diffusion and drift are balanced to leading order. 
In the Kompaneets model condensate formation is a purely hyperbolic phenomenon and near the origin the diffusive part becomes negligible. 
Finally, the ref.~\cite{fornaro_measure_2012} considers a modification of eq.~\eqref{eq:befpOrig} in 1D with sublinear diffusion and \textit{linear} drift, which is such that the associated entropy functional is maintained. The resulting equation is the gradient flow of this functional with respect to the $L^2$-Wasserstein distance, allowing to resort to established techniques in optimal transportation to deduce relaxation to the entropy minimiser. The (sub)linearity of the drift in~\cite{fornaro_measure_2012} is crucial in this approach. 
It, however, precludes the possibility of finite-time condensation for bounded initial data.

\section[Wellposedness for monotonic viscosity solutions]{Existence, uniqueness and regularity for monotonic viscosity solutions}\label{chp:framework}

In this section we introduce a weak notion of solution for a class of equations generalising eq.~\eqref{eq:invBefp} and establish an associated wellposedness theory. 
The equations we consider take the form
\begin{align}\label{eq:genInv}
  G(u,\partial_tu,\partial_xu,\partial_x^2u)=0 \quad \text{in }\Om,
\end{align}
with $\Om:=(0,T)\times(0,m)$, where $G:\mathbb{R}^4\to\mathbb{R}$ is a continuous function satisfying:
\setlist[enumerate,1]{start=0}
\begin{enumerate}[label=(A\arabic*)]
  \item\label{hp:q} The function $q\mapsto G(z,\alpha,p,q)$ is non-increasing for all $z,\alpha,p\in\mathbb{R}$. 
\end{enumerate}
\setlist[enumerate,1]{start=1}
Additional structural assumptions on $G$ will be formulated when needed the first time. We will use the \enquote{curly font} to denote the corresponding operator, i.e.\;we let
\begin{align}\label{eq:abbrG}
 \mathcal{G}(u):=G(u,\partial_tu,\partial_xu,\partial_x^2u)
\end{align}
and similarly $\mathcal{F}(u):=F(u,\partial_tu,\partial_xu,\partial_x^2u)$, where the function $F$ is to be specified. While in the original problem the variable $x$ represents the mass variable, we still refer to $x$ as a spatial variable provided no confusion arises with the variable $v$ or $r$ used for the velocity space.

In comparison to the existing literature \cite{ishii_viscosity_1990,crandall_users_1992,silvestre_fully_2013}, our approach has the following two main novelties: the first one consists in the fact that it can deal with parabolic equations which are not strictly monotonic in the time derivative, as long as $G$ satisfies a certain strict monotonicity condition in its first argument, the second one lies in the preservation of monotonicity in $x$, provided the problem admits monotonic barriers. 

\subsection[Notion of solution]{Preliminary definitions and the notion of solution}

Our concept of solution for equation~\eqref{eq:genInv} is the standard notion of a viscosity solution.
In order to formulate it, we first need to introduce some additional notation.

  We say that a test function $\phi$ \textit{touches} the function $u$ \textit{from above} (resp.\;\textit{from below}) at the point $\omega\in \Omega$ if $\phi(\om)=u(\om)$ and if there exists a neighbourhood $N\subseteq\Om$ of $\om$ such that $\phi\ge u$ (resp.\;$\phi\le u$) in $N$.

\begin{definition}[Parabolic super-/subdifferential]
 For a function $u$ defined on $\Om$ and a point $\om\in\Om$ we let 
  \begin{multline*}
   \mathcal{P}^+u(\om)=\{(\alpha,p,q)\in \mathbb{R}^3:
   \;(\alpha,p,q)=(\partial_t\phi,\partial_x\phi,\partial_x^2\phi)_{|\om}\text{ for some test function }\phi\\\text{  which touches $u$ from above at }\om\}.
 \end{multline*}
 Analogously, we define
 \begin{multline*}
   \mathcal{P}^-u(\om)=\{(\alpha,p,q)\in \mathbb{R}^3:
   \;(\alpha,p,q)=(\partial_t\phi,\partial_x\phi,\partial_x^2\phi)_{|\om}\text{ for some test function }\phi\\\text{  which touches $u$ from below at }\om\}.
 \end{multline*}
 We further let $\mathcal{P}u(\om)=\mathcal{P}^+u(\om)\cap \mathcal{P}^-u(\om)$.
\end{definition}
\begin{rem}
  The set $\mathcal{P}u(\om)$ is non-empty if and only if the pointwise derivatives $^{(p)}\partial_tu(\om)$, $^{(p)}\partial_xu(\om)$, $^{(p)}\partial_x^2u(\om)$ exist. In this case,  $\mathcal{P}u(\om)=\{(^{(p)}\partial_tu(\om),{^{(p)}\partial}_xu(\om),{^{(p)}\partial^2_x}u(\om))\}$ is a singleton, which we will then identify with its unique element, i.e.
  \begin{align*}
    \mathcal{P}u(\om)=(^{(p)}\partial_tu(\om), {^{(p)}\partial}_xu(\om), {^{(p)}\partial_x^2}u(\om)). 
  \end{align*}
\end{rem}

\begin{definition}
 We let 
 \begin{multline*}
   \overline{\mathcal{P}}^\pm u(\om)=\Big\{(\alpha,p,q)\in \mathbb{R}^3:
   \;\exists\,\om_n\in\Om \text{ and }\exists(\alpha_n,p_n,q_n)\in\mathcal{P}^\pm u(\om_n) \\\text{ such that }
   (\om_n, u(\om_n),\alpha_n,p_n,q_n)\to(\om,u(\om),\alpha,p,q)
   \Big\}.
 \end{multline*}
\end{definition}
We will also need the elliptic analogues of $\mathcal{P}$ and its versions.
\begin{definition}[Second order sub-/superdifferential] Let $d\in\mathbb{N}^+$ and  $U\subset\mathbb{R}^d$ be open. For a function $v:U\to\mathbb{R}$ and $x\in U$ we define
\begin{multline*}
  \mathcal{J}^{2,+}v(x)=\Big\{(p,q)\in \mathbb{R}^d\times\mathrm{Sym}(d):\exists\;\phi\in C^2(U)\text{  with }v-\phi\le v(x)-\phi(x)\\\text{ such that  }(p,q)=(D\phi(x),D^2\phi(x))\Big\}.
 \end{multline*}
 The sets $\mathcal{J}^{2,-}u(x),\mathcal{J}^2u(x), \overline{\mathcal{J}}^{2,\pm}u(x)$ are then defined analogously as in the parabolic case and, if $\mathcal{J}^2u(x)$ is non-empty, this set will be identified with its unique element $(^{(p)}Du(x),{^{(p)}D^2}u(x))$.
\end{definition}

 We remark that $(\alpha,p,q)\in \mathcal{P}^+u(t,x)$ resp.\;$(\alpha,p,q)\in\mathcal{P}^-u(t,x)$ if and only if there exists a neighbourhood $N$ of $(t,x)$ such that as $N\ni(s,y)\to(t,x):$
\begin{align}\label{eq:taylor+}\
 u(s,y)\le u(t,x) +\alpha (s-t)+p(y-x)+\frac{q}{2}|y-x|^2+o(|s-t|+|y-x|^2)
\end{align}
resp.\;
\begin{align}\label{eq:taylor-}
 u(s,y)\ge u(t,x) +\alpha (s-t)+p(y-x)+\frac{q}{2}|y-x|^2+o(|s-t|+|y-x|^2).
\end{align}

  If $u(t,\cdot)$ is non-decreasing, letting $s=t$ in ineq.~\eqref{eq:taylor+} resp.\;in ineq.~\eqref{eq:taylor-} and $y\to x^+$ resp.\;$y\to x^-$, it follows that $p\ge0$. In particular, for functions $u$ which are non-decreasing in $x$, we have 
  \begin{align*}
   \mathcal{P}^\pm u(\om)\subseteq \mathbb{R}\times \mathbb{R}^+_0\times\mathbb{R}
  \end{align*}
  for all $\om\in\Om$.
\begin{definition}[Semicontinuous envelopes]
  Given $u=u(\om)$ we define the functions
  \begin{align*}
    u^*(\om)&=\lim_{r\searrow0}\sup\{u(\xi):\xi\in\Om, |\xi-\om|\le r\},
  \\u_*(\om)&=\lim_{r\searrow0}\inf\{u(\xi):\xi\in\Om, |\xi-\om|\le r\}.
  \end{align*} 
  The function $u$ is \textit{upper semicontinuous} (usc) if $u=u^*$, and \textit{lower semicontinuous} (lsc) if $u=u_*$.
   We call $u^*$ (resp.\;$u_*$) the \textit{usc} (resp.\;\textit{lsc}) \textit{envelope} of~$u$. 
\end{definition}
Notice that for any $\om\in\Om$ there exists a sequence $\xi_k\overset{k\to\infty}{\to}\om$ such that $u(\xi_k)\overset{k\to\infty}{\to} u^*(\om)$. Also note that the function $u$ is usc if and only if $u(\om)\ge \limsup_{k\to\infty} u(\xi_k)$ for any sequence $\xi_k\overset{k\to\infty}{\to}\om$. Furthermore, $v$ is lsc if and only if $-v$ is usc. 

Now we are in a position to state the notion of solution we propose for eq.~\eqref{eq:genInv}.
\begin{definition}[Viscosity (sub-/super-) solution]\label{def:viscsol}\itemsep.15em
  Suppose that the continuous function $G$ satisfies property~\nref{hp:q}, and let $u$ be a function defined on $\Om$.  We call $u$ a
  \begin{itemize}[label=\raisebox{0.25ex}{\tiny$\bullet$}]
    \item  \textit{(viscosity) subsolution} of equation~\eqref{eq:genInv} in $\Om$ if it is upper semicontinuous and if for any $\om\in\Om$  and any 
    $(\alpha,p,q)\in\mathcal{P}^+u(\om)$
  we have 
  \begin{align*}
   G(u(\om),\alpha,p,q)\le0.
  \end{align*}
  \item \textit{(viscosity) supersolution} of equation~\eqref{eq:genInv} in $\Om$ if it is  lower semicontinuous and if for any $\om\in\Om$ and any $(\alpha,p,q)\in\mathcal{P}^-u(\om)$
  we have 
  \begin{align*}
   G(u(\om),\alpha,p,q)\ge0.
  \end{align*}
  \item \textit{viscosity solution} of equation~\eqref{eq:genInv} in $\Om$ if it is both a subsolution and a supersolution of equation~\eqref{eq:genInv} in $\Om$. (In this case $u$ is necessarily continuous.)
  \end{itemize}
  In places we use the short phrase \enquote{$u$ is a viscosity (sub-/super-) solution of $\mathcal{G}=0$} if it is a viscosity (sub-/super-) solution of eq.~\eqref{eq:genInv}. Since we will only deal with sub- and supersolutions in the viscosity sense, we usually drop the word \enquote{viscosity} in these cases.
\end{definition}
Notice that, by the continuity of $G$, in Definition~\ref{def:viscsol} one can replace $\mathcal{P}^\pm u(\om)$ with $\overline{\mathcal{P}}^\pm u(\om)$.
\begin{rem}
Of course, the mere formulation of Definition~\ref{def:viscsol} does not require the assumption~\nref{hp:q}. However, it is this property which ensures that the definition is  meaningful in the sense that it generalises the notion of a classical solution.
\end{rem}

\subsection{Stability}
One advantage of the notion of viscosity solutions lies in its good stability properties.
In order to demonstrate this, we reformulate~\cite[Proposition~4.3]{crandall_users_1992} (for elliptic problems) in terms of our parabolic problem.
\begin{proposition}\label{prop:approxSuper}
 Let $v\in\mathrm{USC}(\Om)$, let $\om\in \Om$ and assume that $(\alpha,p,q)\in \mathcal{P}^+v(\om)$. Suppose that $u_n\in\mathrm{USC}(\Om)$ is a sequence of functions satisfying  
 \begin{align*}
  \begin{rcases*}
    (i) \text{ there exist $\om_n\in \Om$ such that }(\om_n,u_n(\om_n))\to(\om,v(\om))
    \\(ii) \text{ if $\xi_n\in \Om$ and $\xi_n\to \xi$, then }\limsup_{n\to\infty}u_n(\xi_n)\le v(\xi).
  \end{rcases*}
 \end{align*} 
 Then there exist $\hat \om_n\in \Om$, $(\alpha_n,p_n,q_n)\in \mathcal{P}^+u_n(\hat \om_n)$ such that 
 \begin{align*}
(\hat \om_n, u_n(\hat \om_n),\alpha_n,p_n,q_n)\to (\om,v(\om),\alpha,p,q).
 \end{align*}
\end{proposition}
\begin{proof}
 The proof is similar to the one of~\cite[Proposition~4.3]{crandall_users_1992}. Notice that this result does not involve the equation.
\end{proof}

\begin{remark}[Stability]\label{rem:stab}
 Observe that we have the following corollaries of Proposition~\ref{prop:approxSuper}.
  \begin{enumerate}[label=(\alph*)]
   \item\label{it:stabC}
  The notion of viscosity solutions is stable under locally uniform convergence: 
  let $G_n=G_n(z,\alpha,p,q)$, $n\in \mathbb{N}$, be continuous and such that $G_n\to G$ as $n\to\infty$ locally uniformly. Furthermore assume that, for each $n$, $u_n$ is a viscosity solution of $\mathcal{G}_n=0$ in $\Om$ and that the sequence $(u_n)$ converges locally uniformly in $\Om$ to some function $u$. Then $u$ is a viscosity solution of $\mathcal{G}=0$ in $\Om$.
  \item\label{it:stabsupr} If ${V}$ is a family of subsolutions of equation~\eqref{eq:genInv} and $u:=\sup_{v\in{V}}v$ is such that the usc envelope $u^*$ of $u$ satisfies $u^*(\om)<\infty$ for all $\om\in \Om$, then $u^*$ is a subsolution of equation~\eqref{eq:genInv}.
  \end{enumerate}
\end{remark}

\subsection{Comparison}\label{ssec:visccp}

Given that our notion of solution is a rather weak one, our first concern is the question of uniqueness subject to prescribed data. 
The comparison principle established below is a fundamental and very powerful tool in our theory, and its range of applications goes beyond uniqueness.
\begin{proposition}[Comparison]\label{prop:visccomp}
  Suppose that, in addition to~\nref{hp:q}, the continuous function $G$ has the following property:
{\normalfont
\begin{enumerate}[label=(A\arabic*)]
 \item\label{hp:strict} For all $p, q$ 
 the function $(z,\alpha)\mapsto G(z,\alpha,p,q)$ is \textit{weakly strictly increasing} in the sense that
for all $(z,\alpha),(z',\alpha')\in\mathbb{R}^2$ 
 \begin{align*}
   \begin{cases}
  [ z\le z' \text{ and } \alpha\le \alpha'] \quad \Rightarrow \quad G(z,\alpha,p,q)\le G(z',\alpha',p,q),\\
  [ z< z' \text{ and } \alpha<\alpha'] \quad \Rightarrow \quad G(z,\alpha,p,q)< G(z',\alpha',p,q).
   \end{cases}
 \end{align*}
\end{enumerate}
} 
\noindent Let $0<T\le\infty$ and assume that $u\in \mathrm{USC}(\Om\cup\partial_p\Om)$ is a subsolution and $v\in\mathrm{LSC}(\Om\cup\partial_p\Om)$ a supersolution of eq.~\eqref{eq:genInv} in $\Om$ satisfying $u\le v$ on $\partial_p\Om$. Then $u\le v$ in $\Om$.
\end{proposition}

\begin{proof}[Proof of Proposition~\ref{prop:visccomp}]
  We may assume, without loss of generality, that $T<\infty$ and that the upper semicontinuous $\mathbb{R}$-valued functions $u$ and $-v$ are bounded above. 
  (Otherwise, we apply the argument below with $T$ replaced by $T'<T$.)
  
   Arguing by contradiction, let us suppose that 
 \begin{align*}
   \sup_\Omega (u-v)>0.
 \end{align*}
 This implies that for $\eta>0$ sufficiently small
  \begin{align*}
    K:=\sup_{(t,x)\in\Omega}\left(u(t,x)-v(t,x)-\frac{\eta}{T-t}\right)>0.
 \end{align*}
 Notice that the function 
 \begin{align*}
  \tilde u(t,x):=u(t,x)-\frac{\eta}{T-t}
 \end{align*}
  is a subsolution of eq.~\eqref{eq:genInv} which is bounded above and satisfies $\lim_{t\nearrow T}u(t,\cdot)=-\infty$ where the convergence is uniform in $x\in J$.
 
     Due to the mere semicontinuity of the functions involved we cannot proceed using classical calculus.
     Also notice that we do not know whether the function $\tilde u-v$ is the subsolution of a suitable parabolic equation. 
 To compensate for the lack of regularity, we use a well-known technique consisting in first doubling the independent variables and then penalising the deviation of corresponding variables. 
 Concretely, for $\varepsilon>0$ we consider the function 
 \begin{align*}
  h_\varepsilon(t,x,s,y):= \tilde u(t,x)-v(s,y)-\frac{|t-s|^2}{2\varepsilon}-\frac{|x-y|^2}{2\varepsilon}.
 \end{align*}
 Now let
\begin{align*}
  K_\varepsilon:=\sup_{(t,x),(s,y)\in\Omega}h_\varepsilon(t,x,s,y)
\end{align*}
and notice that $K_\varepsilon\ge K>0$. 
The fact that $h_\varepsilon$ is usc and bounded above combined with the behaviour of $\tilde u(t,\cdot)$ as $t\to T$ 
implies that for sufficiently small $\varepsilon>0$ 
the supremum is attained at some point $\om_\varepsilon:=(\om_{1,\varepsilon},\om_{2,\varepsilon}):=((t_\varepsilon,x_\varepsilon),(s_\varepsilon,y_\varepsilon))\in (\Om\cup\partial_p\Om)\times(\Om\cup\partial_p\Om)$. 
Moreover,
$(\om_{1,\varepsilon}-\om_{2,\varepsilon})\to0$ as $\varepsilon\to0$ and, after passing to a subsequence, $\om_{i,\varepsilon}\to\bar\om$, $i=1,2$, for some $\bar\om\in\Om\cup\partial_p\Om$.
First assume $\bar\om\in\partial_p\Om$. Then we obtain
\begin{align*}
  0<K\le\limsup_{\varepsilon\to0}h_\varepsilon(\om_{1,\varepsilon},\om_{2,\varepsilon})\le \limsup_{\varepsilon\to0}(\tilde u(\om_{1,\varepsilon})-v(\om_{2,\varepsilon}))\le \tilde u(\bar\om)-v(\bar\om)\le0,
\end{align*}
a contradiction. Hence, we must have $\bar\om\in\Om$, so that for small enough $\varepsilon$, we have $\om_{1,\varepsilon},\om_{2,\varepsilon}\in \Omega$. 
Now we can apply \cite[Theorem~3.2]{crandall_users_1992}, with 
$k=2$, $N_i=2$, $\mathcal{O}_i=\Omega$, $u_1=\tilde u, u_2=-v$ (which is usc), $\phi(t,x,s,y)=\frac{|t-s|^2}{2\varepsilon}+\frac{|x-y|^2}{2\varepsilon}$, and the maximiser $\hat x=(\om_{1,\varepsilon},\om_{2,\varepsilon})$. Then, \cite[Theorem~3.2]{crandall_users_1992} guarantees the existence of 
$Q_{i,\varepsilon}\in\mathrm{Sym}(2)$ , $i=1,2$, such that
\begin{align*}
  (D_{\om_i}\phi(\om_\varepsilon),Q_{i,\varepsilon})\in \overline{\mathcal{J}}^{2,+}u_i(\om_{i,\varepsilon})\;\;\;\text{ for }i=1,2
\end{align*}
and 
\begin{align}\label{eq:matineq}
 Q_\varepsilon:=\left(\begin{matrix}
  Q_{1,\varepsilon} & 0 \\ 0 & Q_{2,\varepsilon}
 \end{matrix}\right)\le A+A^2,
\end{align}
where $A=D^2\phi(\om_\varepsilon)$. 
Notice that
\begin{align*}
 D_{\om_1}\phi(\om_\varepsilon)=\frac{1}{\varepsilon}(t_\varepsilon-s_\varepsilon, x_\varepsilon-y_\varepsilon)^t=:(\tau_\varepsilon,p_\varepsilon)^t,
\end{align*}
\begin{align*}
 D_{\om_2}\phi(\om_\varepsilon)=-(\tau_\varepsilon,p_\varepsilon)^t,
\end{align*}
and
\begin{align*}
A=D^2\phi(\om_\varepsilon)=\frac{1}{\varepsilon}\left(\begin{matrix}
  1 & 0 & -1 & 0 \\ 0 & 1 & 0 & -1 \\ -1 & 0 & 1 & 0 \\ 0 & -1 & 0 & 1
 \end{matrix}\right).
\end{align*}
Writing 
\begin{align*}
 Q_{i,\varepsilon}=:\left(\begin{matrix}
a_{i,\varepsilon} & b_{i,\varepsilon} \\ b_{i,\varepsilon} & q_{i,\varepsilon}
 \end{matrix}\right)
\end{align*}
we have for 
 $\xi:=(0,1,0,1)^t$
the identity
 $\xi^t Q_\varepsilon\xi = q_{1,\varepsilon}+q_{2,\varepsilon}.$
Hence, since $\xi\in \ker(A)$, the matrix inequality~\eqref{eq:matineq} implies 
\begin{align*}
 q_{1,\varepsilon}+q_{2,\varepsilon}\le0.
\end{align*}
By definition, the fact that 
\begin{align*}
  (D_{\om_1}\phi(\om_\varepsilon),Q_{1,\varepsilon})\in \overline{\mathcal{J}}^{2,+}u_1(\om_{1,\varepsilon}),\;\;u_1=\tilde u
\end{align*}
means that there exist sequences $\om^{(n)}_{1,\varepsilon}:=(t^{(n)}_\varepsilon,x^{(n)}_\varepsilon)$ and 
\begin{align*}
  (\tau^{(n)}_\varepsilon,p^{(n)}_\varepsilon,Q^{(n)}_{1,\varepsilon})\in \mathcal{J}^{2,+}\tilde u(\om^{(n)}_{1,\varepsilon})
\end{align*}
such that as $n\to\infty$
\begin{align*}
  \om^{(n)}_{1,\varepsilon}\to\om_{1,\varepsilon},\;\;\tilde u(\om^{(n)}_{1,\varepsilon})\to \tilde u(\om_{1,\varepsilon})\;\;\text{and}\;\;
(\tau^{(n)}_\varepsilon,p^{(n)}_\varepsilon,Q^{(n)}_{1,\varepsilon})\to (\tau_\varepsilon,p_\varepsilon,Q_{1,\varepsilon}).
\end{align*}
In particular, we have as $(t,x)\to \om^{(n)}_{1,\varepsilon}:$
{\small
\begin{align*}
  \tilde u(t,x)\le \tilde u(\om^{(n)}_{1,\varepsilon}) &+(t-t_\varepsilon^{(n)})\tau_\varepsilon^{(n)}+(x-x_\varepsilon^{(n)})p_\varepsilon^{(n)}
  + \frac{1}{2}(t-t_\varepsilon^{(n)})^2a_{1,\varepsilon}^{(n)}+\frac{1}{2}(x-x_\varepsilon^{(n)})^2q_{1,\varepsilon}^{(n)}
  \\&+(t-t_\varepsilon^{(n)})(x-x_\varepsilon^{(n)})b_{1,\varepsilon}^{(n)}+o((t-t_\varepsilon^{(n)})^2+(x-x_\varepsilon^{(n)})^2)
  \\\le \tilde u(\om^{(n)}_{1,\varepsilon}) &+(t-t_\varepsilon^{(n)})\tau_\varepsilon^{(n)}+(x-x_\varepsilon^{(n)})p_\varepsilon^{(n)}
  +\frac{1}{2}(x-x_\varepsilon^{(n)})^2q_{1,\varepsilon}^{(n)}+\frac{\sigma}{2}(x-x_\varepsilon^{(n)})^2
  \\&+o(|t-t_\varepsilon^{(n)}|+(x-x_\varepsilon^{(n)})^2),
\end{align*}
}\normalsize
where $\sigma>0$ can be chosen arbitrarily small.
This means that for all $\sigma>0$
\begin{align*}
  (\tau_\varepsilon^{(n)}, p_\varepsilon^{(n)}, q_{1,\varepsilon}^{(n)}+\sigma)\in \mathcal{P}^+\tilde u(\om_{1,\varepsilon}^{(n)}),
\end{align*}
which, upon choosing $\sigma=\frac{1}{n}$ and letting $n\to\infty$, yields 
\begin{align*}
  (\tau_\varepsilon, p_\varepsilon, q_{1,\varepsilon})\in \overline{\mathcal{P}}^+\tilde u(\om_{1,\varepsilon})
\end{align*}
or, equivalently, 
\begin{align}\label{eq:superdiff}
  \left(\tau_\varepsilon+\frac{\eta}{(T-t_\varepsilon)^2}, p_\varepsilon, q_{1,\varepsilon}\right)\in \overline{\mathcal{P}}^+u(\om_{1,\varepsilon}).
\end{align}
Starting from 
\begin{align*}
  (-\tau_\varepsilon,-p_\varepsilon,Q_{2,\varepsilon})\in \overline{\mathcal{J}}^{2,+}u_2(\om_{2,\varepsilon}),
\end{align*}
we can argue analogously for $u_2$ to find 
  $(-\tau_\varepsilon, -p_\varepsilon, q_{2,\varepsilon})\in \overline{\mathcal{P}}^+u_2(\om_{2,\varepsilon}),$
or, equivalently, 
\begin{align}\label{eq:subdiff}
  (\tau_\varepsilon, p_\varepsilon, -q_{2,\varepsilon})\in \overline{\mathcal{P}}^-v(\om_{2,\varepsilon}).
\end{align}
Thanks to the conclusions~\eqref{eq:superdiff} and~\eqref{eq:subdiff}, we can make use of the fact that $u$ (resp.\;$v$) is a subsolution (resp.\;a supersolution) of equation~\eqref{eq:genInv} and obtain the inequalities
\begin{align}\label{eq:inSub}
  G(u(\om_{1,\varepsilon}),\tilde\tau_{\varepsilon},p_{\varepsilon},q_{1,\varepsilon})\le0,
\end{align}
where $\tilde\tau_\varepsilon=\tau_{\varepsilon}+\frac{\eta}{(T-t_\varepsilon)^2}>\tau_\varepsilon$, and
\begin{align}\label{eq:inSuper}
 G(v(\om_{2,\varepsilon}),\tau_{\varepsilon},p_{\varepsilon},-q_{2,\varepsilon})\ge0.
\end{align}
Subtracting ineq.~\eqref{eq:inSuper} from ineq.~\eqref{eq:inSub}, we infer the following contradiction
\begin{align*}
 0 &\ge G(u(\om_{1,\varepsilon}),\tilde\tau_{\varepsilon},p_{\varepsilon},q_{1,\varepsilon})-G(v(\om_{2,\varepsilon}),\tau_{\varepsilon},p_{\varepsilon},-q_{2,\varepsilon})\hfill
 \\ &\ge G(u(\om_{1,\varepsilon}),\tilde\tau_{\varepsilon},p_{\varepsilon},q_{1,\varepsilon})-G(v(\om_{2,\varepsilon}),\tau_{\varepsilon},p_{\varepsilon},q_{1,\varepsilon})>0,
\end{align*}
where we used hypotheses~\nref{hp:q} and~\nref{hp:strict}. 
\end{proof}

As a consequence of the proof of Proposition~\ref{prop:visccomp}, viscosity solutions of $\mathcal{G}=0$ obey an intersection comparison principle. 
For its precise formulation we recall the notion of the number of sign changes of a continuous function defined on an interval (see e.g.~\cite[Appendix~F]{quittner_superlinear_2007}, \cite{galaktionov_sturmian_2004} and references therein). 

\begin{definition}[Number of sign changes]\label{def:Z}
 Let $J\subset \mathbb{R}$ be connected. Given 
 $v\in C(J)$ define the set 
 \begin{multline*}N_v:=\{k\in \mathbb{N}:\exists\,x_j\in J, j=0,1,\dots,k\text{ such that } x_0<x_1<\dots<x_k, \\\text{ and }v(x_{j-1})\cdot v(x_{j})<0\text{ for }j=1,\dots,k\}
  \end{multline*}
  and let $Z[v]:=\sup\left(N_v\cup\{0\}\right).$
    We call $Z[v]\in\mathbb{N}\cup\{0,\infty\}$ the \textit{number of sign changes} of $v$.
\end{definition}
In the literature the number of sign changes is also referred to as the \textit{zero number}. 
We are usually interested in the number of sign changes $Z[u_1-u_2]$ of the difference of two functions $u_i\in C(J_i)$, $i=1,2$, where in general $J_1\not=J_2$. 
In this case, our notation is to be understood as
\begin{align*}
  Z[u_1-u_2]:=Z[\left(u_1-u_2\right)_{|J'}], \quad \text{where }J'=J_1\cap J_2. 
\end{align*}
We now state the intersection comparison principle in a form typically used in applications.

\begin{corollary}[Intersection comparison]\label{cor:intsec}
 Assume that the continuous function $G$ satisfies hypotheses~\nref{hp:q} and~\nref{hp:strict}.
Let $t_1<t_2,\;x_1<x_2$ and define $Q:=(t_1,t_2)\times(x_1,x_2)$. Suppose that 
$u,v\in C(\bar Q)$
are viscosity solutions of $\mathcal{G}=0$ in $Q$ satisfying:
\begin{enumerate}[$(\mathrm{L)}$]
  \item\label{hp:lbc} the number of connected components of $\partial_pQ^\pm:=\{q\in\partial_pQ:\pm(u-v)(q)>0\}$ does not exceed the number of connected components of $\partial_pQ^\pm\cap\left(\{t_1\}\times[x_1,x_2]\right)$.
\end{enumerate}
Then the number of sign changes of the difference $w:=u-v$ is non-increasing in time, i.e.
\begin{align*}
  Z[w(t,\cdot)]\le  Z[w(t_1,\cdot)]\quad\text{ for all }t\in(t_1,t_2).
\end{align*}
\end{corollary}
Loosely speaking, the corollary asserts that the number of intersections of two viscosity solutions of $\mathcal{G}=0$ is non-increasing in time provided that no intersections occur on the lateral boundary.
Corollary~\ref{cor:intsec} is a consequence of the maximum principle as it is applied in the proof of Proposition~\ref{prop:visccomp}. The proof essentially follows the original approach by Sturm~\cite{sturm_partielles_1836}
treating linear parabolic equations (see also~\cite[Chapter~1]{galaktionov_sturmian_2004}), where the application of the classical maximum principle needs to be substituted for the maximum type argument used in the proof of Proposition~\ref{prop:visccomp}.
\begin{proof}[Proof of Corollary~\ref{cor:intsec}]
  Consider the sets $Q^\pm:=\{q\in \bar Q:\pm w(q)>0\}$ and $$A^\pm:=\{q\in\partial_pQ:\pm w(q)>0\}.$$ Notice that, by the continuity of $w$, the number of connected components of $Q^\pm$ equals that of $Q^\pm\setminus  \partial Q$. 
  
 The main ingredient in the proof is the following auxiliary result: 
 \begin{lemma}\label{l:cc}
   Suppose that, except for condition~\ref{hp:lbc}, the hypotheses of Corollary~\ref{cor:intsec} hold true. 
   For each connected component $Q'$ of $Q^\pm$ there exists a connected component $A'$ of $A^\pm$ such that $A'\subset Q'$. 
 \end{lemma}
\begin{proof}[Proof of Lemma~\ref{l:cc}]
  Assume that  $Q'\subset Q^+$. 
  Then the assertion follows if we can show that $\sup_{\partial Q'\cap\partial_pQ}w>0$, where we use the convention $\sup_{\emptyset}w=-\infty$. We argue by contradiction and assume that   $\sup_{\partial Q'\cap\partial_pQ}w\le 0$. 
By the definition of $Q'$ and the continuity of $w$, we have $w=0$ in $\partial Q'\cap Q$.
Since $\sup_{Q'}w>0$, the contradiction is now obtained as in the proof of Proposition~\ref{prop:visccomp}.

If $Q'\subset Q^-$, apply the previous reasoning to $\tilde w=v-u$ instead of~$w$.
\end{proof}
We can now conclude the proof of Corollary~\ref{cor:intsec}. 
Let $t\in(t_1,t_2)$ and suppose that there exist $y_j\in(x_1,x_2),$ $j=0,\dots,k$ such that $y_0<y_1<\dots<y_k$ and $w(t,y_j)\cdot w(t,y_{j-1})<0$ for $j=1,\dots,k$. For each $j$ let $Q_j$ be the connected component of $Q^+\cup Q^-$ containing $(t,y_j)$. 
Using Lemma~\ref{l:cc} with $Q$ replaced by a suitable axis-aligned rectangle $\tilde Q\subset Q$, it is easy to see that $Q_j\neq Q_l$ whenever $j\neq l$. Applied once more, Lemma~\ref{l:cc} combined with hypothesis~\ref{hp:lbc} provides us with $\tilde y_j\in (x_1,x_2)$,  $j=0,\dots,k$, such that $\tilde y_0<\dots<\tilde y_k$ and $w(t_1,\tilde y_j)\cdot w(t_1,\tilde y_{j-1})<0$ for $j=1,\dots,k$.
\end{proof}

\subsection{Perron method} 

As a preparatory step towards existence we establish a Perron method for equation~\eqref{eq:genInv} for monotonic (and non-monotonic) functions, which roughly states that once a subsolution $u^-$ and a supersolution $u^+$ 
satisfying $u^-\le u^+$ are found, there exists an \enquote{almost}  viscosity solution squeezed between these barriers. Since in our applications we are particularly interested in functions which are non-decreasing with respect to $x$, we start with some preliminaries on monotonicity.

\begin{definition}[$x$-monotonicity]\label{def:xm} We say that a function $u=u(t,x)$ is \textit{$x$-monotonic}, in short \textit{$x$-m}, if the function $x\mapsto u(t,x)$ is non-decreasing for any $t$.
 
\end{definition}

\begin{fact}
 If $u=u(t,x)$ is $x$-monotonic, so are the semicontinuous envelopes $u^*$ and $u_*$.
\end{fact}
\begin{fact}
  If $V$ is a set of functions such that all $v\in V$ are $x$-m, then the function $u$ defined via $u(t,x):=\sup_{v\in V}v(t,x)$ is $x$-m.
\end{fact} 

While the idea of the Perron method is well-known in the literature, the assumption of monotonicity requires some non-trivial modifications. The version provided below is an adaptation of \cite[Lemma~2.3.15]{silvestre_fully_2013}.
  
\begin{proposition}\label{prop:perron} Suppose that hypothesis~\nref{hp:q} holds true and let $0<T\le\infty$.  Assume that $u^\pm$ are locally bounded $x$-m functions satisfying $u^-\le u^+$ in $\Om$ and suppose that $u^-$ is a subsolution and $u^+$ a supersolution of eq.~\eqref{eq:genInv} in $\Om$. 
    Then there exists an $x$-m function $u:\Om\to\mathbb{R}$ such that $u^*$ is a subsolution of eq.~\eqref{eq:genInv} in $\Om$, $u_*$ a supersolution and $u^-\le u \le u^+$.  
    
    The statement remains valid when the $x$-m property is dropped everywhere.
\end{proposition}

\begin{proof} We confine ourselves to showing the (more interesting) assertion regarding the $x$-monotonic case. The proof of the second assertion is easier and can be carried out along similar lines (without the need of a distinction of cases).
    Consider the non-empty set
  \begin{align*}
    V = \{v:\Om\to\mathbb{R}\;|\; u^-\le v\le u^+,\; v \text{ is $x$-monotonic, $v^*$ is a subsolution of eq.}~\eqref{eq:genInv}\}
  \end{align*}
and let 
\begin{align*}
  u = \sup_{v\in V}v.
\end{align*}
Then $u$ is $x$-monotonic and, by Remark~\ref{rem:stab}\;\ref{it:stabsupr}, $u^*$ is a subsolution of eq.~\eqref{eq:genInv} in $\Om$.

It remains to show that the $x$-m, lsc function $u_*$ is a supersolution of eq.~\eqref{eq:genInv}. We argue by contradiction and assume that there exists $\om\in\Om$, $(\alpha,p,q)\in \mathcal{P}^-u_*(\om)$ and $\theta>0$ such that 
\begin{align}\label{eq:hpNoSuper}
    G(z,\alpha,p,q)\le-\theta,
\end{align}
where $z:=u_*(\om)$. 
Notice that, since $u_*\le u^+$, if $u_*(\om)=u^+(\om)$, then $(\alpha,p,q)\in \mathcal{P}^-u^+(\om)$, and the fact that $u^+$ is a supersolution would then imply  $G(z,\alpha,p,q)\ge0$, 
which contradicts~\eqref{eq:hpNoSuper}. Therefore
\begin{align*}
 u_*(\om)<u^+(\om),
\end{align*}
and, after possibly decreasing $\theta>0$, we can assume that
\begin{align}\label{eq:strictin}
 u_*(\om)-u^+(\om)\le -\theta<0.
\end{align}
By the translation invariance of the equation with respect to the independent variable $\om$, we can
further assume that $(0,0)\in\Om$ and $\om=(0,0)$.  For small parameters $\delta,\varepsilon>0$ to be determined later, we define
\begin{align*}
 P(s,y)=z+\alpha s+py+\frac{1}{2}qy^2+\delta-\varepsilon\left(|s|+\frac{1}{2}|y|^2\right).
\end{align*}
Note that for any $(s,y)\in\Om$ and  $(\alpha',p',q')\in \mathcal{P}^+P(s,y)$
one has $|\alpha'-\alpha|\le\varepsilon$, $p'=p+qy-\varepsilon y$ and $q'\ge q-\varepsilon$. We further let 
$N_r:=\{(\tilde s,\tilde y):|\tilde s|+|\tilde y|^2/2<r\}$.

We now have to distinguish between the case in which $p>0$ and the one in which $p$ vanishes.

\underline{Case~1:} $p>0$.

In this case,  $P$ is $x$-monotonic in $N_r$ for $r>0$ small enough, and after decreasing $r$ again and choosing $\varepsilon,\delta>0$ sufficiently small, we have 
\begin{align*}
 G(P(s,y),\alpha',p',q') \le -\frac{\theta}{2}
\end{align*}
for any $(s,y)\in N_r$ and $(\alpha',p',q')\in \mathcal{P}^+P(s,y)$. Thus, $P$ is a subsolution of eq.~\eqref{eq:genInv} in~$N_r$.

Since, by inequality~\eqref{eq:strictin}, we have $P(\om)\le u^+(\om)+\delta-\theta$, the fact that $P$ is usc and $u^+$ lsc ensures that, after possibly decreasing $\delta>0$,
\begin{align}\label{eq:belowu+}
 P(s,y)<u^+(s,y)\;\;\;\text{ for }(s,y)\in N_r.
\end{align}

Since $(\alpha,p,q)\in \mathcal{P}^-u_*(\om)$, by inequality~\eqref{eq:taylor-},
\begin{align*}
 u_*(s,y)&\ge z+\alpha s+py+\frac{1}{2}qy^2+o(|s|+|y|^2)
 \\&\ge P(s,y)-\delta+\varepsilon\left(|s|+\frac{1}{2}|y|^2\right)+o(|s|+|y|^2).
\end{align*}
After possibly decreasing $r$, 
we can choose $\delta=\frac{\varepsilon r}{4}$. Then for $(s,y)\in N_r\setminus N_{r/2}$
\begin{align*}
 u_*(s,y)\ge P(s,y)-\frac{\varepsilon r}{4}+\frac{\varepsilon r}{2}+o(r)=P(s,y)+\frac{\varepsilon r}{4}+o(r)
\end{align*}
and hence, for $r$ sufficiently small,
\begin{align*}
  u(s,y)-P(s,y)\ge\frac{\varepsilon r}{8}>0  \;\;\;\text{ for }(s,y)\in N_r\setminus N_{r/2}.
\end{align*}
Let us now define 
\begin{align}\label{eq:defU}
 U(s,y)=\begin{cases}
   \max\{u(s,y),P(s,y)\} \;\;&\text{ if }(s,y)\in N_r,
   \\u(s,y)\;\;&\text{ otherwise}.
        \end{cases}
\end{align}
Then $U$ is non-decreasing, $U^*$ is a subsolution of~\eqref{eq:genInv} in $\Om$ and $u^-\le  U\le u^+$, where the last bound follows from ineq.~\eqref{eq:belowu+}. Hence $U\in V$ and thus $U\le u$. However, by definition 
there exists a sequence $\xi_n\to\om$ such that $u(\xi_n)\to u_*(\om)=z$ and therefore
\begin{align*}
 \liminf_{n\to\infty}(U(\xi_n)-u(\xi_n))\ge \lim_{n\to\infty}(P(\xi_n)-u(\xi_n))=\delta>0.
\end{align*}
 This contradicts $U\le u$.

\medskip 
 \underline{Case~2:} $p=0$.
 
 In this case the $x$-monotonicity of $u_*$ implies that $q\le 0$. Hence,  hypothesis~\nref{hp:q} and inequality~\eqref{eq:hpNoSuper} imply that  
 $G(z,\alpha,0,0)\le G(z,\alpha,0,q)\le-\theta$. 

 The competitor $P=P(s,y)$ needs to be adapted since it is strictly decreasing in $y$ for $y>0$. We define 
 \begin{align*}
  \tilde P(s,y)=\begin{cases}
    P(s,y)\;\;\;&\text{ if }y\le0,
    \\P(s,0)=z+\delta+s\alpha-\varepsilon|s|\;\;&\text{ if }y>0. 
                \end{cases}
 \end{align*}
Notice that we can choose $r,\delta,\varepsilon$ sufficiently small such that for all $\sigma\in[-1,1]$
{\small
\begin{align*}
  G(\tilde P,\alpha+\sigma\varepsilon, \partial_y\tilde P,\partial_y^2\tilde P)_{|(s,y)} =G(P,\alpha+\sigma\varepsilon, \partial_yP,\partial_y^2P)_{|(s,y)} \le -\frac{\theta}{2}\;\;\forall (s,y)\in N_r:\;y<0
\end{align*}
 and 
\begin{align*}
 G(\tilde P,\alpha+\sigma\varepsilon, \partial_y\tilde P,\partial_y^2\tilde P)_{|(s,y)} 
 =G(P(s,0),\alpha+\sigma\varepsilon, 0,0) 
 \le -\frac{\theta}{2},\;\;\;\forall |s|< r,\;y>0.
\end{align*}
}\normalsize
Moreover, since $\partial_y\tilde P\in C^0$ with $\partial_y\tilde P(s,0)=0$, whenever $(\tilde\alpha,\tilde p,\tilde q)\in \mathcal{P}^{+}\tilde P(s,0)$, we must have $\tilde p=0$, $\tilde q\ge0$, $\tilde\alpha=\alpha+\sigma\varepsilon$ for some $\sigma\in[-1,1]$ and therefore 
$G(\tilde P(s,0),\tilde\alpha,\tilde p,\tilde q)\le-\frac{\theta}{2}$ whenever $|s|<r$.
Hence, $\tilde P$ is a subsolution of $G=0$ in the domain $\tilde N_r$ defined via
\begin{align*}
  \tilde N_r:=N_r\cup\{(s,y)\in\Om:|s|< r, y\ge0\}.
\end{align*}
As in Case~1 we have $P(\om)<u^+(\om)$ for $\delta$ sufficiently small, so that after possibly decreasing $r$ once more, we obtain
\begin{align*}
  \tilde P<u^+\;\;\;\text{ in }\tilde N_r.
\end{align*}
For this conclusion we have used in particular the $x$-monotonicity of $u^+$.

Arguing as in Case~1 and letting in particular $\delta=\frac{\varepsilon r}{4}$, for $r,\varepsilon$ sufficiently small, we can guarantee that 
\begin{align}\label{eq:lnr}
  u> \tilde P \;\;\;\text{ in }\left(N_r\setminus N_{\frac{r}{2}}\right)\cap\{(s,y)\in\Om:y\le0\}.
\end{align}
The inequality~\eqref{eq:lnr} implies that $u(s,0)>\tilde P(s,0)$ for $\frac{r}{2}\le|s|< r$, and thanks to the $x$-monotonicity of $u$ therefore 
\begin{align*}
  u(s,y)>\tilde P(s,y)\;\;\;\text{ for all }\frac{r}{2}\le|s|< r,\;y\ge0.
\end{align*}
We now define $U$ as in formula~\eqref{eq:defU} with $P$ replaced by $\tilde P$ and $N_r$ replaced by $\tilde N_r$. Then $U$ is $x$-monotonic,  $U^*$ is a subsolution of $G=0$ in $\Om$, $u^-\le u\le U\le u^+$ but $U\not\equiv u$, which contradicts the maximality of $u$.
\end{proof}

\subsection{Existence, uniqueness and Lipschitz regularity}
We are now in a position to show existence and uniqueness for the Cauchy--Dirichlet problem associated with equation~\eqref{eq:genInv} conditional on the existence of appropriate barriers.

\begin{theorem}[Existence and uniqueness]\label{thm:exuniq}
  Suppose that the continuous function $G$ satisfies the conditions~\nref{hp:q} and~\nref{hp:strict}.
  Given $0<T\le\infty$ and locally bounded
  $x$-monotonic functions $u^\pm:\Om\cup\partial_p\Om\to\mathbb{R}$ 
  such that $u^-$ is a subsolution and $u^+$ a supersolution of eq.~\eqref{eq:genInv} in $\Om$ satisfying 
  {\normalfont
     \begin{enumerate}[label=(B\arabic*)]
    \item $u^-\le u^+$ in $\Om\cup\partial_p\Om$ 
   \item $(u^-)_*=(u^+)^*$ on $\partial_p\Om$,
  \end{enumerate}
}
\noindent there exists a unique $x$-monotonic viscosity solution $u\in C(\Om\cup \partial_p\Om)$ of eq.~\eqref{eq:genInv} in $\Om$ with the property that $u=u^-(=u^+)$ on $\partial_p\Om$. This solution satisfies 
 $u^-\le u \le u^+$.  
 
 The assertion remains valid when dropping the $x$-monotonicity everywhere.
\end{theorem}
\begin{rem}
  By replacing $u^\pm$ with $-u^\mp$ one obtains the same result for functions which are non-\textit{increasing} in $x$.
\end{rem}

\begin{proof} We only consider the $x$-m case since the reasoning in the non-monotonic case is completely similar. From the assumptions we infer that
  \begin{align*}
    \lim_{\textrm{\scriptsize{\parbox{2cm}{\vspace{3pt}\centering $\om\in\Om,$\\ $\om\to\bar\om\in\partial_p\Om$}}}}\hspace{-.5cm}u^\pm(\om)=u^-(\bar\om)=u^+(\bar\om)\in\mathbb{R}.
  \end{align*}
 Thus, Proposition~\ref{prop:perron} guarantees the existence of an $x$-m function $u:\Om\cup\partial_p\Om\to\mathbb{R}$ satisfying $u^-\le u\le u^+$ such that $u^*$ is a subsolution, $u_*$ a supersolution of eq.~\eqref{eq:genInv} and  $u_*=u^*=u^\pm$ on $\partial_p\Om$.
  Hence, Proposition~\ref{prop:visccomp} implies that $u^*\le u_*$, and thus $u=u^*=u_*\in C(\Om\cup\partial_p\Om)$ is a viscosity solution of eq.~\eqref{eq:genInv}.
  Uniqueness subject to prescribed values on $\partial_p\Om$ is a consequence of Proposition~\ref{prop:visccomp}.
\end{proof}

Before providing concrete examples to Theorem~\ref{thm:exuniq}, we show that if the barriers $u^\pm$ are Lipschitz continuous, the viscosity solution obtained in Theorem~\ref{thm:exuniq} inherits this regularity.
The main ingredients in the proof are again versions of the so-called theorem on sums (a maximum type principle for semicontinuous functions), which already was the key to proving the comparison principle (Proposition~\ref{prop:visccomp}).
Related approaches can be found in~\cite{ishii_viscosity_1990} and~\cite{silvestre_fully_2013}. 

\begin{proposition}[Lipschitz continuity in time]\label{prop:lipTime}
Suppose that the conditions~\nref{hp:q}, \nref{hp:strict} hold true and assume that, in addition to the hypotheses in Theorem~\ref{thm:exuniq}, the barriers $u^\pm$ are locally Lipschitz continuous with respect to $t$ in $\Om\cup\partial_p\Om$, i.e.\;for any $T'<T$ there exists $K_{T'}<\infty$ such that for all $s,t\in[0,T']$ and all $x\in\bar J$
\begin{align*}
  |u^\pm(t,x)-u^\pm(s,x)|\le K_{T'}|t-s|.
\end{align*} 
Then for any $T'<T$ and the same constant $K_{T'}$ the associated viscosity solution $u$ satisfies the  estimate
\begin{align}\label{eq:Lipschitz-bound}
  |u(t,x)-u(s,x)|\le K_{T'}|t-s|
\end{align}
for all $s,t\in[0,T']$ and all $x\in\bar J$.
\end{proposition}

\begin{proof}
  Assume that the assertion is false.  
  Then there exists $T'<T$ such that for $K=K_{T'}$
\begin{align*}
  \sup_{t,s\in[0,T'], x\in J}\left(u(t,x)-u(s,x)- K|t-s|\right)>0
\end{align*}
and thus for $\eta>0$ sufficiently small 
\begin{align*}
 M:= \sup_{t,s\in[0,T'), x\in J}\left(u(t,x)-\frac{\eta}{T'-t}-\left(u(s,x)+\frac{\eta}{T'-s}\right) -K|t-s|\right)>0.
\end{align*}
With the abbreviation $u_1(t,x):=u(t,x)-\frac{\eta}{T'-t}$ and $u_2(s,x):=-\left(u(s,x)+\frac{\eta}{T'-s}\right)$ it follows that for any $\varepsilon>0$ 
\begin{align*}
 M_\varepsilon:= \sup_{t,s\in[0,T'), x,y\in J}
 \left(u_1(t,x)+u_2(s,y)
   -(K|t-s|+\frac{1}{2\varepsilon}|x-y|^2)\right)\ge M>0.
\end{align*}
Let now $\varphi(t,x,s,y):=(K|t-s|+\frac{1}{2\varepsilon}|x-y|^2)$ and define 
$w(t,x,s,y):=u_1(t,x)+u_2(s,y)-\varphi(t,x,s,y)$. 
Since $u\in C([0,T']\times\bar J)$, the function $w$ attains its maximum at some point 
$(\bar t, \bar x,\bar s,\bar y)\in [0,T')\times\bar J\times[0,T')\times\bar J$.
Notice that by the properties of $u^\pm$ one has  $u^-(t,x)-u^-(0,x)\le u(t,x)-u(0,x)\le u^+(t,x)-u^+(0,x)$ and thus for all $x\in \bar J$ and $t\in[0,T')$
\begin{align*}
  |u(t,x)-u(0,x)|\le Kt,
\end{align*}
which implies that $\bar t,\bar s>0$, whenever $\varepsilon=\varepsilon(u^\pm(0,\cdot),M)>0$ is sufficiently small.

We next claim that $\bar x,\bar y\not\in \partial J$ for small enough $\varepsilon=\varepsilon(K)>0$.
Indeed, assuming that this is not the case, we find a sequence $\varepsilon_n\to0$ such that $\bar x\in \partial J$ for all $n$ 
or $\bar y\in \partial J$ for all $n$. By the boundedness of $u$, we must have $\bar x-\bar y\to0$
 as $n\to\infty$, and there exist $x_\infty\in\partial J$,  $t_\infty, s_\infty\in[0,T')$ 
 such that 
 after passing to a subsequence $\bar x,\bar y\to x_\infty$, $\bar t\to t_\infty, \bar s\to s_\infty$ as $n\to\infty$.
 But then the continuity of $u$ and the fact that $u=u^\pm$ on $\partial_p\Om$ lead to a
   contradiction to the assumption $M>0$.

Hence $(\bar t,\bar x, \bar s,\bar y)\in (0,T')\times J\times (0,T')\times J$. Notice also that $\bar t\neq\bar s$
for $\varepsilon$ sufficiently small since otherwise 
$M_\varepsilon\to0$ along a subsequence.
This guarantees that for small enough $\varepsilon$, the function $\varphi$ is $C^2$ in a neighbourhood of the maximiser of $w$.

We can now argue as in the proof of Proposition~\ref{prop:visccomp}:
by \cite[Theorem~3.2]{crandall_users_1992} there exist $\tau,p\in\mathbb{R}$, where $p\ge0$, and 
$Q_1,Q_2\in\mathrm{Sym}(2)$ satisfying 
$(\tau,p,Q_1)\in \overline{\mathcal{J}}^{2,+}(u_1)(\bar t,\bar x)$,
$(-\tau,-p,Q_2)\in \overline{\mathcal{J}}^{2,+}(u_2)(\bar s,\bar y)$ such that for $Q=\mathrm{diag}(Q_1,Q_2)$ and $A=D^2\varphi(\bar t,\bar x,\bar s,\bar y)$ the matrix inequality
$Q\le A+A^2$ holds true. Letting $q_i:=(Q_i)_{2,2}$ for $i=1,2,$ it follows that $q_1+q_2\le0$ and, furthermore, $(\tau,p,q_1)\in \overline{P}^{+}u_1(\bar t,\bar x)$,
$(-\tau,-p,q_2)\in \overline{P}^{+}u_2(\bar s,\bar y)$. By the definition of $u_i, i=1,2,$ this means that 
$(\tau+\frac{\eta}{(T'-\bar t)^2},p,q_1)\in \overline{P}^{+}u(\bar t,\bar x)$,
$(\tau-\frac{\eta}{(T'-\bar s)^2},p,-q_2)\in \overline{P}^{-}u(\bar s,\bar y)$.
A contradiction is now inferred in precisely the same way as in the proof of Proposition~\ref{prop:visccomp}.
\end{proof}

The Lipschitz bound~\eqref{eq:Lipschitz-bound} implies 
that for all $\om=(t,x)\in\Om$ with $t\le T'$ we have the implication
\begin{align}\label{eq:tauBd}
  \left[\exists\, p,q\in\mathbb{R}:\;\; (\tau,p,q)\in \overline{\mathcal{P}}^+ u(\om) \;\text{ or }\; (\tau,p,q)\in \overline{\mathcal{P}}^- u(\om)\right]\;\;\;\Rightarrow\;\;\;|\tau|\le K_{T'}.
\end{align}
Thanks to this observation, we easily obtain full Lipschitz regularity of viscosity solutions admitting barriers as in Theorem~\ref{thm:exuniq} which are Lipschitz continuous.

\begin{proposition}[Lipschitz continuity in space]\label{prop:lipSpace}
Suppose that the conditions~\nref{hp:q}, \nref{hp:strict} hold true and assume that the barriers $u^\pm$ in Theorem~\ref{thm:exuniq} are in addition locally Lipschitz continuous in $\Om\cup\partial_p\Om$.
Then for any $T'<T$ the associated viscosity solution $u$ satisfies the estimate
\begin{align*}
  |u(t,x)-u(t,y)|\le \tilde K_{T'}|x-y|
\end{align*}
for all $t\in[0,T']$ and all $x,y\in\bar J$,
where $$\tilde K_{T'}:=\max\{[u^-]_{L^\infty(0,T';C^{0,1}(\bar J))},[u^+]_{L^\infty(0,T';C^{0,1}(\bar J))}\}.$$
\end{proposition}
 \noindent Prop.~\ref{prop:lipSpace} can be proved using arguments similar to the proof of Prop.~\ref{prop:lipTime}.

As an immediate consequence of Prop.~\ref{prop:lipTime} and Prop.~\ref{prop:lipSpace} we obtain
\begin{corollary}[Lipschitz continuity]\label{cor:Lip}
 Under the hypotheses in Prop.~\ref{prop:lipSpace}, the viscosity solution~$u$ of eq.~\eqref{eq:genInv} is locally Lipschitz continuous in $\Om\cup\partial_p\Om$, and for any $T'<T$ it satisfies the estimate
  \begin{align*}
    [u]_{C^{0,1}([0,T']\times\bar J)}\le \sqrt{2}\max\{K_{T'},\tilde K_{T'}\},
  \end{align*}
  where $K_{T'}$ and $\tilde K_{T'}$ denote the constants defined in Prop.~\ref{prop:lipTime} and Prop.~\ref{prop:lipSpace}.
\end{corollary}

\subsection[Applications to GBFP]{Applications to generalised bosonic Fokker--Planck equations (GBFP)}\label{ssec:applications}

Here we demonstrate how Thm~\ref{thm:exuniq} can be used to derive global-in-time wellposedness for the Cauchy--Dirichlet problem associated with a class of equations generalising~\eqref{eq:invBefp}. In the original variables these problems correspond to a class of nonlinear Fokker--Planck equations generalising in 1D the nonlinear FP equations~\eqref{eq:befpCont} on a centred ball (cf.\;eq.~\eqref{eq:entropyInt} below).  We refer to this generalised class, considered in Thm~\ref{thm:BEFPtype} below, as \textit{generalised bosonic Fokker--Planck equations} (GBFP). 
The equations are reminiscent of the setting in the ref.~\cite{abdallah_minimisation_2011} treating the stationary problem, but the precise regularity assumptions are slightly different.

 For a positive function $h\in C((0,\infty))$ satisfying $1/h\in L^1(1,\infty)$ and 
 $\int_s^\infty\frac{1}{h(z)}\,\d z\in L^1_\mathrm{loc}([0,\infty))$
 define ${\Phi^{(h)}}:\mathbb{R}_{\ge0}\to\mathbb{R}_{\le0}$ via ${\Phi^{(h)}}(0)=0$, ${\Phi^{(h)}}'(s)=-\int_s^\infty\frac{1}{h(z)}\,\d z,$ $s>0,$ and consider the functional
  \begin{align}\label{eq:entr}
    \mathcal{H}^{(h,R)}(f)=\int_{-R}^R\left(\frac{|r|^2}{2}f+\Phi^{(h)}(f)\right)\, \d r,\qquad f\in L^1_+(-R,R),
  \end{align}
  where $R\in(0,\infty)$. We are interested in the equation
\begin{align}\label{eq:entropyInt}
  \frac{\partial f}{\partial t}= \frac{\d}{\d r} \left(h(f) \frac{\d}{\d r}
\frac{\delta \mathcal{H}^{(h,R)}}{\delta f}[f]\right),\quad t>0, \,r\in(-R,R)
\end{align}
 subject to zero-flux boundary conditions 
$\frac{\d}{\d r}\frac{\delta \mathcal{H}^{(h,R)}}{\delta f}[f] = 0$ on $\{-R,R\}$.\\
We define the \textit{steady states} of this conservative problem to be the positive, smooth solutions $f$ of
$$
\frac{\d}{\d r}\frac{\delta \mathcal{H}^{(h,R)}}{\delta f}[f] = 0,
$$
i.e.~the solutions $f^{(h)}_{\infty,\theta}$ of
$$
\frac{|r|^2}{2}+{\Phi^{(h)}}'(f^{(h)}_{\infty,\theta})=-\theta,
$$
where $\theta$ is a constant of integration. 

In the following we assume that $1/h(s)$ is not integrable near $s=0$, which implies that $\lim_{s\to0^+}{\Phi^{(h)}}'(s)=-\infty$. 
Since ${\Phi^{(h)}}'$ is strictly increasing with $\lim_{s\to\infty}{\Phi^{(h)}}'(s)=0$, we can then solve the last equation 
for $f^{(h)}_{\infty,\theta}$ to obtain
$$
f^{(h)}_{\infty,\theta}(r)=({\Phi^{(h)}}')^{-1}\left(-(|r|^2/2+\theta)\right),\quad r\in(-R,R),
$$
provided that $\theta\in[0,\infty)$. 
Here, for a reason to become clear later, we were slightly imprecise and admitted the limiting case $\theta=0$, despite the fact that the function $f^{(h)}_{\infty,0}$
satisfies  $f^{(h)}_{\infty,0}(r)\to\infty$ as $r\to0$. 
Furthermore, notice that $f^{(h)}_{\infty,\theta}\to0$ uniformly in $[-R,R]$ as  $\theta\to\infty$ and that
for any $\theta<\infty$ there exists $c_\theta>0$ such that $f^{(h)}_{\infty,\theta}\ge c_{\theta}$ in $[-R,R]$.
Thus, letting 
\begin{align}\label{eq:mRth}
 m^{(R,\theta)}_h&:=\int_{-R}^Rf^{(h)}_{\infty,\theta}(r)\,\d r,
 \\\theta^{(R,m)}_h&:=\min\{\theta\ge0: m_h^{(R,\theta)}\le m\}\label{eq:thRm}
\end{align}
 and denoting for given $m\in(0,\infty)$ and given $\theta\ge\theta_h^{(R,m)}$
 by $u^{(R,m)}_{\theta,-,h}:[0,m]\to[-R,R]$ (resp.\;by $u^{(R,m)}_{\theta,+,h}:[0,m]\to[-R,R]$) the pseudo-inverse of the cdf 
 of $(m-m_h^{(R,\theta)})\delta_{-R}+f^{(h)}_{\infty,\theta}\cdot\mathcal{L}^1$ (resp.\;of $(m-m_h^{(R,\theta)})\delta_R+f^{(h)}_{\infty,\theta}\cdot\mathcal{L}^1$), we infer that $u^{(R,m)}_{\theta,\mp,h}$ are Lipschitz continuous in $[0,m]$ and that for any non-decreasing function $u_0\in C^1([0,m])$ with $u_0(0)=-R$, $u_0(m)=R$ there exists $\theta<\infty$ such that 
\begin{align}\label{eq:barrGen}
  u^{(R,m)}_{\theta,-,h}\le u_0\le u^{(R,m)}_{\theta,+,h}.
\end{align}
See Figure~\ref{fig:barrupm} on page~\pageref{fig:barrupm} for an illustration of the functions $u^{(R,m)}_{\theta,\pm,h}$ and~\eqref{def:PsI} for the definition of the pseudo-inverse of an increasing, right-continuous function $M$.

Formally, the equation for the pseudo-inverse $u(t,\cdot)$ of the cdf associated with $f(t,\cdot)$ states
\begin{align}\label{eq:ibcg}
  u_t-\frac{u_{xx}}{u_x^2}+u_x h(1/u_x)u=0\quad\text{in }\Om:=(0,\infty)\times(0,m),
\end{align}
where $m$ denotes the mass of the initial datum $f_0$, i.e.\;
$m=\int_{-R}^Rf_0(r)\,\d r$.
In view of the no-flux boundary conditions for eq.~\eqref{eq:entropyInt}, we complement eq.~\eqref{eq:ibcg} with the Dirichlet conditions
\begin{align}\label{eq:bcD}
 u(t,0)=-R,\quad u(t,m)=R
\end{align}
for all $t>0$.

We henceforth suppose that $\lim_{s\to\infty}s^3/h(s)$ exists in $[0,\infty)$ and define
\begin{align}\label{eq:defG}
  G(z,\alpha,p,q) = \left(|p|^3h(1/|p|)\right)^{-1}\left(|p|^2\alpha-q\right)+z,
\end{align}
with the understanding that for all $z,\alpha,q\in\mathbb{R}$
\begin{align*}
  G(z,\alpha,0,q):=\lim_{p\to0} G(z,\alpha,p,q),
\end{align*}
which, by assumption, exists in $\mathbb{R}$. Then the function $G$ is continuous on $\mathbb{R}^4$, satisfies the conditions~\ref{hp:q} and~\ref{hp:strict}, and
defining $\mathcal{G}$ by formula~\eqref{eq:abbrG},
 equation~\eqref{eq:ibcg} can be reformulated as
\begin{align}\label{eq:ibcgd}
 \mathcal{G}(u)=0\quad\text{in }\Om.
\end{align}
Notice that equations~\eqref{eq:ibcg} and~\eqref{eq:ibcgd} are equivalent if $0<u_x<\infty$.
 
  \begin{definition}[Initial data for GBFP problem]\label{def:admBefp}
   For a given function $h$ as introduced above (and $G$ defined via~\eqref{eq:defG}) let $\mathcal{S}_{h}$ denote the set of all non-decreasing functions $u_0\in C^{1}([0,m])$ having the following properties:
  \begin{itemize}
   \item $u_0(0)=-R$, $u_0(m)=R$,
   \item $u'_0(x)>0$ for all $x\in[0,m]$ with $|u_0(x)|>0$, 
   \item $u_0\in C^2(\{|u_0|>0\})$ and 
   \begin{align}\label{eq:defC}
     C:=C(u_0):=\sup_{\{|u_0|>0\}}\left|p_0h(p_0^{-1})\mathcal{G}(u_0)\right|<\infty,
\end{align}
with $p_0:=u_0'$ and where we have used the abbreviation~\eqref{eq:abbrG}.
  \end{itemize} 
\end{definition}
\noindent
The choice of $C$ in formula~\eqref{eq:defC} guarantees that $u_0\mp Ct$, $t\ge0$, is a sub- resp.\;supersolution of eq.~\eqref{eq:ibcgd} in $\Om:=(0,\infty)\times(0,m)$. Any $u_0\in C^2([0,m])$ with $\min_{[0,m]}u_0'>0$ and $u_0(0)=-R, u_0(m)=R$ lies in the set $\mathcal{S}_h$, but, in general, Def.~\ref{def:admBefp} also allows for functions which have a flat part at level zero, see Remark~\ref{rem:conddat} for details and the meaning of the bound~\eqref{eq:defC}. 

We are now in a position to show wellposedness for the problems introduced above.

\begin{theorem}[Global existence, uniqueness and Lipschitz continuity for GBFP]
  \label{thm:BEFPtype}
  Suppose that the function $h\in C((0,\infty),\mathbb{R}^+)$ satisfies $1/h\not\in L^1(0,1)$, $\int_s^\infty\frac{1}{h(z)}\,\d z\in L^1(0,1)$ and that the limit $\lim_{s\to\infty}s^3/h(s)$ exists in $[0,\infty)$.
  Given $u_0\in\mathcal{S}_h$ there exists a unique, $x$-monotonic viscosity solution $u\in C(\Om\cup\partial_p\Om)$ of problem~\eqref{eq:bcD}--\eqref{eq:ibcgd} such that $u(0,\cdot)=u_0$.
  This solution is globally Lipschitz continuous with constant bounded above by $K=\sqrt{2}\max\{C(u_0),[u_{\theta,\pm,h}^{(R,m)}]_{C^{0,1}}\}$, where $\theta\ge0$ is any number such that ineq.~\eqref{eq:barrGen} is fulfilled.
\end{theorem}

\begin{proof} 
  Choose $\theta<\infty$ such that ineq.~\eqref{eq:barrGen} holds true.
Then the function
\begin{align*}
  u^-(t,x):=\max\left\{u_0(x)-Ct,u^{(R,m)}_{\theta,-,h}(x)\right\}
\end{align*}
 is a subsolution, while the function 
 \begin{align*}
  u^+(t,x):=\min\left\{u_0(x)+Ct, u^{(R,m)}_{\theta,+,h}(x)\right\}
\end{align*}
is a supersolution satisfying $u^-\le u_0\le u^+$. 
\medskip

\begin{figure}[ht]\centering
   \resizebox{.5\linewidth}{!}{     
    \begin{tikzpicture}[line cap=round,line join=round,>=triangle 45,>=stealth' ]
\begin{axis}[
    anchor=origin,  
    x=1cm,y=1.5cm,
    hide axis,
    ymax = 2, xmax = 6
]
\pgfplotstableread{tikzData/barrDat.txt}
\asymIV
\addplot[smooth,line width=0.45mm] table[x = Fp, y = uplus] from \asymIV 
node[above]{\Large$u_{\theta,+,h}^{(R,m)}(x)\hspace{7cm}$};
\addplot[loosely dotted,line width=0.35mm] table[x = Fp, y = uin] from \asymIV node[below left=1cm ]{\large$u_{0}(x)\qquad$};
\addplot[smooth,gray!40,line width=0.45mm] table[x = Fm, y = uminus] from \asymIV 
node[gray!80!black,below =2cm]{\Large$u_{\theta,-,h}^{(R,m)}(x)$};

\end{axis}
\draw[gray,loosely dotted,line width=0.25mm] (0,-1.9)node[below]{$0$} -- (0,-1.5);
\draw[gray,loosely dotted, line width=0.25mm] (3,-2) node[below]{$m-m_h^{(R,\theta)}$} --(3,-1.5);
\draw[gray,loosely dotted, line width=0.25mm] (4.4,-1.9)node[below,gray]{$m$} -- (4.4,-1.5);
\draw[gray,loosely dotted,line width=0.25mm] (-0.4,-1.5)node[left]{$-R$} -- (0,-1.5);
\draw[gray,loosely dotted,line width=0.25mm] (-0.4,1.5)node[left]{$R$} -- (3,1.5);
\end{tikzpicture}
}
\caption[Barriers for lateral boundary conditions]{
  Given an initial datum $u_0$ for the GBFP equation and $\theta$ satisfying~\eqref{eq:barrGen}, the functions $u_{\theta,\pm,h}^{(R,m)}$ depicted above serve as barriers enforcing the lateral boundary conditions~\eqref{eq:bcD}.}
\label{fig:barrupm}
\end{figure}
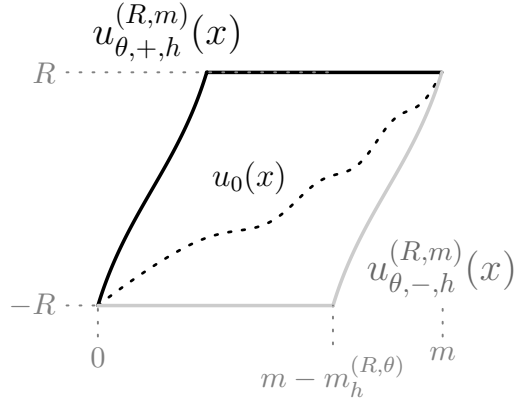

\noindent The functions $u^\pm$ are of class $C^{0,1}(\Om\cup\partial_p\Om)$ and have the desired behaviour on $\partial_p\Om$.
Thus, Thm~\ref{thm:exuniq} yields the first claim. The Lipschitz continuity is a consequence of Cor.~\ref{cor:Lip}.
\end{proof}

\begin{remark}[Critical mass $m_c(R)$]\label{rem:mc}
  In general, the singularity of $f^{(h)}_{\infty,0}$ near the origin may not be integrable.
Following~\cite{abdallah_minimisation_2011}, one finds that $m_c(R):=m_h^{(R,0)}<\infty$
if and only if 
\begin{align}\label{eq:critsupercrit}
 \int_1^\infty \frac{s}{h(s)}\left(\int_s^\infty \frac{1}{h(\sigma)}\,\d\sigma\right)^{-\frac{1}{2}}\,\d s<\infty.
\end{align}
\end{remark} 

\begin{remark}[Entropy minimisers]\label{rem:emin}
  Since $\lim_{s\to\infty}{\Phi^{(h)}}(s)/s=0$, we can proceed as in~\cite{abdallah_minimisation_2011} and extend the functional $\mathcal{H}^{(h,R)}$ to the set of finite measures on $[-R,R]$ by ignoring the singular part (with respect to Lebesgue) in the nonlinear term involving ${\Phi^{(h)}}$. Following the proof of~\cite[Theorem~3.1]{abdallah_minimisation_2011} one can show that the unique minimiser of the extended functional $\mathcal{\widetilde H}^{(h,R)}$ among measures 
  $\mu\in \mathcal{M}^+([-R,R])$ of mass $m>0$ is given by
 \begin{align}\mu^{(m,R,h)}_\infty=
  \begin{cases}
    f_{\infty,\theta}^{(h)}\cdot\mathcal{L}^1&\text{ if }m< m_c(R),\text{ where }\theta=\theta^{(R,m)}_h,
    \\f_c^{(h)}\cdot\mathcal{L}^1+(m-m_c)\delta_0&\text{ if }m\ge m_c(R),
  \end{cases}
  \label{eq:minGBFP}
 \end{align}
 where $f_c^{(h)}:=f_{\infty,0}^{(h)}$.
Notice that for any $m>0$ the pseudo-inverse 
of the cdf of $\mu_\infty^{(m,R,h)}$ is of class $C^1([0,m])$ and is a viscosity solution of eq.~\eqref{eq:ibcgd} while for $\theta>0$ and $m>m^{(R,\theta)}_h$ the pseudo-inverse
cdf of the measure $f^{(h)}_{\infty,\theta}\cdot\mathcal{L}^1 +(m-m_h^{(R,\theta)})\delta_0$ is neither a sub- nor a supersolution of eq.~\eqref{eq:ibcgd}.
\end{remark}

\begin{remark}\label{rem:conddat}
  If $m_c(R)<\infty$, there exist functions $u_0\in\mathcal{S}_h$
  which have a flat part at level zero, so that there exist $0<x_-\le x_+<m$ such that $u_0(x)=0, u_0'(x)=0$ for all $x\in[x_-,x_+]$ and $|u_0(x)|>0$ for $x\not\in[x_-,x_+]$. In this case, condition~\eqref{eq:defC} is non-trivial and enforces that, loosely speaking, the asymptotic behaviour of $u_0(x)$ as $x\to (x_\pm)^\pm$ agrees with the corresponding behaviour of the pseudo-inverse cdf of $f_{\infty,0}^{(h)}$.
  For its meaning at the level of the  density $f_0$ associated with the generalised inverse of $u_0$ (for a specific choice of $h$) see Sec.~\ref{sec:relToOrgBdd}.
\end{remark}

For $\gamma\ge2$ the function $h(s)=h_\gamma(s):=s(1+s^\gamma)$ is admissible in Thm~\ref{thm:BEFPtype}. We thus infer
\begin{corollary}[Global existence, uniqueness and Lipschitz continuity for the 1D bosonic Fokker--Planck model in the $L^1$-supercritical and -critical case]
  \label{cor:bosonicFP}
  Let $m,R\in(0,\infty)$ and abbreviate $\Om := (0,\infty)\times(0,m)$. Suppose that $\gamma\ge2$, let $F$ be defined by 
  \begin{align}\label{eq:defF}
   F(z,\alpha,p,q):=|p|^\gamma\alpha-|p|^{\gamma-2}q+z(1+|p|^\gamma)
  \end{align} 
   and abbreviate $\mathcal{F}(u) := F(u,\partial_tu,\partial_xu,\partial_x^2u)$.
   Given $u_0\in\mathcal{S}_{h_\gamma}$ 
  there exists a unique, $x$-monotonic viscosity solution $u\in C(\Om\cup\partial_p\Om)$ of the problem 
  \begin{align}\label{eq:bosonicFPmassCHP4}
    \begin{cases}
       \mathcal{F}(u)=0, \quad &\text{ in }\Om,
       \\  u(t,0)=-R,\quad u(t,m)=R, \qquad & \text{ for } t>0,
       \\  u(0,x) = u_0(x),\quad & \text{ for } x\in[0,m].
    \end{cases}
\end{align}
This solution is globally Lipschitz continuous with Lipschitz constant bounded above by
\begin{align*}
  K=K\left(C_1(u_0),[u_{\theta,\pm,h_\gamma}^{(R,m)}]_{C^{0,1}}\right) <\infty,
\end{align*}
where $\theta>0$ is any positive number
such that $u_{\theta,-,h_\gamma}^{(R,m)}\le u_0 \le u_{\theta,+,h_\gamma}^{(R,m)}$.
\end{corollary}

\section[Finite-time condensation \texorpdfstring{\&}{} relaxation to equilibrium for BFP]{Finite-time condensation and relaxation to equilibrium in 1D bosonic Fokker--Planck equations (BFP)}
\label{chp:1dftc}

Given $\gamma\ge 2$, a fixed total mass $m\in(0,\infty)$, a radius $R>0$, and an initial datum $u_0\in C^2([0,m])$ such that $\min_{[0,m]} u_0'>0$ and $u_0(0)=-R, u_0(m)=R$ (which implies that $u_0\in\mathcal{S}_{h_\gamma}$), Corollary~\ref{thm:exuniq} ensures the existence, uniqueness and Lipschitz regularity of viscosity solutions $u=u(t,x)$, non-decreasing in $x$, of the Cauchy--Dirichlet problem~\eqref{eq:bosonicFPmassCHP4}. 

\begin{remark}[Original variables]\label{rem:originalVariables}
 Let $t\ge0$ be fixed. Since the continuous function $u(t,\cdot):[0,m]\to[-R,R]$ is non-decreasing from $u(t,0)=-R$ to $u(t,m)=R$, we can define its \textit{generalised inverse} $M(t,\cdot):[-R,R]\to[0,m]$ via
 \begin{align}\label{eq:defGenInv}
  M(t,r):=\sup\{x\in[0,m]:u(t,x)\le r\},\qquad r\in [-R,R]
\end{align} 
or, equivalently, by $M(t,r)=\max\left( u(t,\cdot)^{-1}(\{r\})\right)$.
By definition $M(t,\cdot)$ is non-decreasing and satisfies $M(t,-R)\ge0$, $M(t,R)=m$. Since $u(t,\cdot)$ is continuous,  $M(t,\cdot)$ is actually strictly increasing. Indeed, the 
second representation of $M(t,\cdot)$ implies 
\begin{align*}
 u(t,M(t,r))=r,
\end{align*}
so that the assumption $M(t,r_1)=M(t,r_2)$ yields $r_1=r_2$.
It is also easy to see that $M(t,\cdot)$ is right-continuous.
Hence, there is a unique Borel measure $\mu(t)\in\mathcal{M}([-R,R])$ satisfying 
\begin{align}\label{eq:mu(t)}
  \mu(t)([-R,r]) = M(t,r) \qquad \text{ for all }r\in[-R,R],
\end{align}
see, e.g.,~\cite[Chapter~20.3]{royden2010real}. 
\end{remark}

The main problems to be tackled in this section are as follows:
\begin{enumerate}[label=(Q\arabic{*})]
 \item\label{it:pbReg} Developing a detailed understanding of the regularity of the viscosity solutions~$u$ and analysing its implications for the problem in the original variables (see Remark~\ref{rem:originalVariables}). 
 \item\label{it:pbEntropy} Establishing an entropy technique valid globally in time 
which enables us to identify the long-time behaviour of solutions and allows us to prove that in the mass-supercritical case $m>m_c(R)$ singularities and condensates always emerge in finite-time.
\end{enumerate}

In this section, data $u_0$ for eq.~\eqref{eq:bosonicFPmassCHP4} are assumed to be admissible in the following sense:
\begin{definition}[Admissible initial datum for problem~\eqref{eq:bosonicFPmassCHP4}]\label{def:admissBosonicFP}
  A function $u_0$ on $[0,m]$ is called an \textit{admissible} initial datum for problem~\eqref{eq:bosonicFPmassCHP4}
  if it satisfies $u_0\in C^2([0,m])$ with $\min_{[0,m]} u_0'>0$ and takes the boundary values $u_0(0)=-R, u_0(m)=R$.
\end{definition}

Let us next briefly outline this section's content: 
we first show that our viscosity solutions are actually weak solutions (in a suitable distributional sense) satisfying a natural a priori estimate associated with the equation.
The regularity derived and the equation's structure will then allow us to prove that our solutions are smooth away from $\{u=0\}$ (Sec.~\ref{sec:impreg}). Subsequently, we translate our results back to the original variables to obtain a finite measure $\mu(t)$, as introduced in~\eqref{eq:mu(t)}, whose singular part with respect to the Lebesgue measure is supported at the origin and whose density (with respect to Lebesgue) is smooth away from the origin. The spatial blow-up profile of the density is proved to be universal to leading order (Sec.~\ref{ssec:spatialProfile}). 
In Sec.~\ref{ssec:entropy} we prove that the entropy dissipation identity (at the level of $\mu(t)$) holds true globally in time, even for solutions with a singular part.
Entropy methods then allow us to deduce the long-time asymptotics and, if $m>m_c$, the formation of a condensate in finite time (Sec.~\ref{ssec:convBdd}). We conclude with several observations and corollaries providing further insights into the nature of singularities (Cor.~\ref{cor:tc} to Prop.~\ref{prop:globalReg}).

Finally, it will be convenient in this section to use the following notations.

\begin{notations}[$\mu^{(R,m)}_\infty$ and $u_\infty^{(R,m)}$]
  \label{not:Rmtheta}
  As above we fix $\gamma\ge2$ and let $h_\gamma(s)=s(1+s^\gamma)$. Then for $R\in(0,\infty)$ and $\theta\ge0$ we abbreviate
  $f_{\infty,\theta}:=f_{\infty,\theta}^{(h_\gamma)}$
  $f_c:=f_{\infty,0}$, $m^{(R,\theta)}:=m^{(R,\theta)}_{h_\gamma}$,
  $\theta^{(R,m)} :=\theta^{(R,m)}_{h_\gamma}$, where $m^{(R,\theta)}_{h_\gamma}$ and $\theta^{(R,m)}_{h_\gamma}$ are defined by~\eqref{eq:mRth} resp.~\eqref{eq:thRm}.
 Next, for given $R,m\in(0,\infty)$ we let
\begin{align}\label{eq:muinfRm}
  \mu^{(R,m)}_\infty:=\mu^{(R,m,h_\gamma)}_\infty,
\end{align}
where $\mu^{(R,m,h_\gamma)}_\infty$ is given by~\eqref{eq:minGBFP}. 
We then denote by~$u_\infty^{(R,m)}$ the pseudo-inverse (in the sense of~\eqref{def:PsI}) of the cdf of $\mu^{(R,m)}_\infty$. 
 Notice that $u_\infty^{(R,m)}\in C^1([0,m])$.
 Finally, given $\theta\ge \theta^{(R,m)}$ we abbreviate 
 $u_{\theta,\pm}^{(R,m)}:=u_{\theta,\pm,{h_\gamma}}^{(R,m)}$, where $u_{\theta,\pm,{h_\gamma}}^{(R,m)}$ has been introduced on p.~\pageref{eq:barrGen} (see also Fig.~\ref{fig:barrupm}).
\end{notations}

\subsection[Refined regularity]{Refined regularity}\label{sec:impreg}
Here we aim to establish the following 
 \begin{theorem}[Refined regularity]\label{thm:refreg}
  Suppose that $\gamma\ge2$. Given $m,R>0$ and an initial datum $u_0$ which is admissible in the sense of Def.~\ref{def:admissBosonicFP},
  let $u\in C(\Om\cup\partial_p\Om)$ denote the unique viscosity solution of the Cauchy--Dirichlet problem~\eqref{eq:bosonicFPmassCHP4} (see Cor.~\ref{cor:bosonicFP}). Recall that $u\in C^{0,1}(\bar\Om)$ and that, for each $t\ge0$ the function $u(t,\cdot)$ is non-decreasing. The following assertions hold true:
  \begin{enumerate}[label=\normalfont{(R\arabic{*})}]
   \item\label{it:reg1} We have the regularity 
   \begin{align}\label{eq:bosonicreg}
u\in L^\infty(0,\infty;C^{1,\frac{1}{\gamma-1}}(\bar J)),
\end{align}
where $J=(0,m)$, 
and $u$ satisfies the estimate
\begin{align*}
 \|\partial_x((\partial_xu)^{\gamma-1})\|_{L^\infty(\Om)}\le C([u]_{C^{0,1}(\bar\Om)},R,\gamma).
\end{align*}
Thus, we have $u\in C_b([0,\infty);C^{1,\beta}(\bar J))$ for $\beta\in(0,\frac{1}{\gamma-1})$ and 
\begin{align}\label{eq:C1bU}
  \sup_{t\ge0}\|u(t,\cdot)\|_{C^{1,\beta}(\bar J)}\le C([u]_{C^{0,1}(\bar\Om)},R,\gamma)
\end{align}
\item\label{it:reg2} Defining the sets 
\begin{align*}
  &\Om^{+}:=\{\om\in\Om:\;|u(\om)|>0\},\\
  &\Om^{++}:=\{\om\in\Om:\;\partial_xu(\om)>0\},
\end{align*} 
which, by~\ref{it:reg1}, are open sets, the solution $u$ is $C^\infty$ in $\Om^{++}$, and we have 
\begin{align*}
  \Om^+ \subseteq \Om^{++}.
\end{align*}
In particular, in $\Om^{++}$ the equation $\mathcal{F}(u)=0$ holds true in the classical sense.
  \end{enumerate}
\end{theorem} 
\begin{remark}\label{rem:shorttimereg}
  Observe that the regularity~\ref{it:reg1} and our hypothesis $\inf_{J}u_0'>0$ imply that there exists $t^*=t^*(u_0)>0$ 
  such that $\{(t,x)\in \Om:t<t^*)\}\subset \Om^{++}$. Thus, thanks to~\ref{it:reg2} we deduce short-time regularity of the viscosity solution~$u$.
\end{remark}
\noindent For a possible extension of the regularity results to solutions of GBFP in Thm~\ref{thm:BEFPtype} see Remark~\ref{rem:generalNLFP}.

 The proof of Thm~\ref{thm:refreg} will be given in the following two subsubsections.

 \subsubsection{Approximate problems}
 \label{ssec:approx}
 \begin{proof}[Proof of Thm~\ref{thm:refreg}~\texorpdfstring{\ref{it:reg1}}{}]
   We consider a regularised version of problem~\eqref{eq:bosonicFPmassCHP4} in $\Om:=(0,\infty)\times J$, $J:=(0,m)$, obtained by replacing the function $F(z,\tau,p,q)$ with
$F_\sigma(z,\tau,p,q):=p^\gamma \tau-(p+\sigma)^{\gamma-2}q+z(1+p^\gamma)$, $0<\sigma\ll1$, the lateral boundary conditions with $u(t,0)=-R_\sigma$ and $u(t,m)=R_\sigma$ for suitable $0<R_\sigma\le R$ with $R_\sigma\to R$ as $\sigma\to 0$ 
and the initial value $u_0$ by suitable approximations $u_{0,\sigma}\in C^2(\bar J)$ with $\min_{\bar J}u_{0,\sigma}'>0$ satisfying $u_{0,\sigma}(0)=-R_\sigma$, $u_{0,\sigma}(m)=R_\sigma$, $u_{0,\sigma}\nearrow u_0$ in $C^{2}(\bar J)$.  It is easy to see that such a sequence $(u_{0,\sigma})$ exists. Under these conditions the constants $C_\sigma(u_{0,\sigma})$, where
\begin{align}\label{eq:Csigma}
  C_\sigma(v):=\sup_{x\in J}\left|-\frac{(p(x)+\sigma)^{\gamma-2}}{p^\gamma(x)}q(x)+v(x)(p(x)^{-\gamma}+1)\right|,\;\;p=v',q=v'',
\end{align}
are uniformly bounded in $0<\sigma\ll1$. 

Existence and uniqueness of $x$-monotonic viscosity solutions are obtained by Thm~\ref{thm:exuniq} provided  appropriate barriers can be found. A possible construction of the barriers is as follows: we fix some $\theta>0$ such that
\begin{align*}
  u_{\theta,-}^{(R,m)}\le u_0\le u_{\theta,+}^{(R,m)}
\end{align*}
 and define
\begin{align*}
  \kappa(\sigma):=\sup_{x\in J:|u_{\theta}(x)|>0}\left|u_\theta(x)-\frac{(p_\theta(x)+\sigma)^{\gamma-2}q_\theta(x)}{1+p_\theta^\gamma(x)}\right|,
\end{align*}
where we abbreviated $p_\theta:=u_\theta'$, $q_\theta:=u_\theta''$ and $u_\theta=u_{\theta,-}^{(R,m)}$. (Choosing instead $u_\theta=u_{\theta,+}^{(R,m)}$ does not change the value of $\kappa(\sigma)$.)
We note that $\kappa\in C([0,1])$ with $\kappa(0)=0$, and let
\begin{align*}
 R_\sigma:=R-\kappa(\sigma).
\end{align*}
By construction the function
\begin{align*}
  u_{\theta,\sigma}^-:=\max\{-R_\sigma, u_{\theta,-}^{(R,m)}-\kappa(\sigma)\}
\end{align*} 
is a subsolution of $\mathcal{F}_\sigma=0$, while the function 
\begin{align*}
  u_{\theta,\sigma}^+:=\min\{R_\sigma, u_{\theta,+}^{(R,m)}+\kappa(\sigma)\}
\end{align*} 
is a supersolution. Both functions are continuous on $\bar J$ and they satisfy $u_{\theta,\sigma}^\pm(0)=-R_\sigma$, $u_{\theta,\sigma}^\pm(m)=R_\sigma$. It is also clear that after possibly slightly modifying the choice of $u_{0,\sigma}$, we can assume that $ u_{\theta,\sigma}^-\le u_{0,\sigma}\le u_{\theta,\sigma}^+$.

We now let 
\begin{align*}
  &v^-_\sigma(t,x):=\max\{u_{0,\sigma}(x)-C_\sigma t,u_{\theta,\sigma}^-(x)\},\\
  &v^+_\sigma(t,x):=\min\{u_{0,\sigma}(x)+C_\sigma t,u_{\theta,\sigma}^+(x)\},
\end{align*}
where $C_\sigma:=C_\sigma(u_{0,\sigma})$ (see formula~\eqref{eq:Csigma}). This defines bounded $x$-m functions $v^\pm_\sigma\in C(\Om\cup\partial_p\Om)$ with the desired behaviour on $\partial_p\Om$ such that $v^-_\sigma$ is a subsolution and $v^+_\sigma$ a supersolution of $\mathcal{F}_\sigma=0$.
Thus, subject to the conditions on $\partial_p\Om$ specified above, there exists a unique viscosity solution $v_\sigma$ of $\mathcal{F}_\sigma=0$ in $(0,\infty)\times J$, which, by Corollary~\ref{cor:Lip}, is such that the Lipschitz norm
$\|v_\sigma\|_{C^{0,1}([0,\infty)\times\bar J)}$ is uniformly bounded in $0<\sigma\ll1$.
The Arzel\`a--Ascoli theorem combined with Remark~\ref{rem:stab}~\ref{it:stabC} and the uniqueness part of Thm~\ref{thm:exuniq} now implies that, upon passing to a subsequence, we have $v_\sigma\to u$ locally uniformly in $\bar\Om$. (Notice that the passage to a subsequence was not necessary.)

The approximate solutions $v_\sigma$ are more regular: for any $\om\in\Om$ and any $(\tau,p,q)\in \mathcal{P}^-v_\sigma(\om)$ we have
\begin{align*}
  p^\gamma\tau-(p+\sigma)^{\gamma-2}q + v_\sigma(\om)(1+p^\gamma)\ge0
\end{align*}
and therefore
\begin{align*}
 q \le C([v_\sigma]_{C^{0,1}(\bar\Om)}) + R\left(\sigma^{2-\gamma}+C([v_\sigma]_{C^{0,1}(\bar\Om)})\right).
\end{align*}
Similarly, for any $\om\in\Om$ and any $(\tau,p,q)\in \mathcal{P}^+v_\sigma(\om)$ we deduce
\begin{align*}
 q \ge -C([v_\sigma]_{C^{0,1}(\bar\Om)}) - R\left(\sigma^{2-\gamma}+C([v_\sigma]_{C^{0,1}(\bar\Om)})\right).
\end{align*}

By Prop.~\ref{prop:semi-conv} (see also Def.~\ref{def:semiconvex}), we conclude that for all $t>0$ (and uniformly in $t$) 
the function $v_\sigma(t,\cdot)$ is semi-concave as well as semi-convex, which implies (see Lemma~\ref{l:c11}) the regularity $v_\sigma(t,\cdot)\in C^{1,1}(\bar J)$.
Then, as demonstrated in Appendix~\ref{app:measurable}, 
the second pointwise derivative $^{(p)}\partial_x^2v_\sigma$ of $v_\sigma$ with respect to~$x$ exists $\mathcal{L}^2$-almost everywhere in $\Om$ and  $\partial_xv_\sigma$ has a weak derivative satisfying $\partial_x^2v_\sigma={^{(p)}\partial^2_x}v_\sigma\in L^\infty(\Om)$. 
Now we can relate the viscosity solution property to a more classical notion of solution.
From the preceding observations and Rademacher's theorem (see e.g.~\cite{evans_fine_2015}), it follows that $\mathcal{P}v_\sigma(\om)$ exists for $\mathcal{L}^2$-almost every $\om\in\Om$ and that
the function $v_\sigma$ is a strong solution in the sense that the weak derivatives $\partial_tv_\sigma,\partial_xv_\sigma,\partial_x^2v_\sigma$ exist in $L^\infty(\Om)$ and satisfy $F_\sigma(v_\sigma,\partial_tv_\sigma,\partial_xv_\sigma,\partial_x^2v_\sigma)=0$ in $L^\infty(\Om)$.
In particular, in view of the inequality 
$\frac{1}{\gamma-1}|\partial_x((\partial_xv_\sigma)^{\gamma-1})|\le|(\partial_xv_\sigma+\sigma)^{\gamma-2}\partial_x^2v_\sigma|$,
the equation $\mathcal{F}_\sigma(v_\sigma)=0$ and the fact that $[v_\sigma]_{C^{0,1}(\bar\Om)}\le C([u]_{C^{0,1}(\bar\Om)})$  yield the bound
\begin{align}\label{eq:boundsForApprox}
  \|\partial_x((\partial_xv_\sigma)^{\gamma-1})\|_{L^\infty(\Om)}\le C([u]_{C^{0,1}(\bar\Om)},R,\gamma).
\end{align}
Hence, switching to the Bochner function perspective via Fubini's theorem, we have for any $T<\infty$
\begin{align*}
v_\sigma\in L^\infty(0,T;C^{1,\frac{1}{\gamma-1}}(\bar J)), \;\;\partial_tv_\sigma\in L^\infty(0,T;L^\infty(J)),
\end{align*}
with norms uniformly bounded in $\sigma$ (and $T$). Thus, thanks to the Aubin--Lions lemma and the locally uniform convergence $v_\sigma\to u$, we can pass to a subsequence satisfying for $\beta\in(0,\frac{1}{\gamma-1})$ and any $T<\infty$
\begin{align*}
  v_\sigma\to u\;\;\;\text{ in }C([0,T];C^{1,\beta}(\bar J)).
\end{align*}
\label{bvcontTEMP}
In particular $\partial_xv_\sigma\to\partial_xu$ in $C_{\mathrm{loc}}(\bar\Om)$, which implies that 
$(\partial_xv_\sigma)^{\gamma-1}\to(\partial_xu)^{\gamma-1}$ in $C_{\mathrm{loc}}(\bar\Om)$. Now, the bound~\eqref{eq:boundsForApprox} yields
\begin{align}\label{eq:2ndDerUi}
 \|\partial_x((\partial_xu)^{\gamma-1})\|_{L^\infty(\Om)}\le C([u]_{C^{0,1}(\bar\Om)},R,\gamma)
\end{align}
and $u\in C_b([0,\infty);C^{1,\beta}(\bar J))$, with 
\begin{align*}
  \sup_{t\ge0}\|u(t,\cdot)\|_{C^{1,\beta}(\bar J)}\le C([u]_{C^{0,1}(\bar\Om)},R,\gamma)
\end{align*}
for $\beta\in(0,\frac{1}{\gamma-1})$.
This completes the proof of Thm~\ref{thm:refreg}~\texorpdfstring{\ref{it:reg1}}{}.
\end{proof}

\begin{remark}
 The specific form of the regularised equation in Section~\ref{ssec:approx} is not essential. For instance, we could have chosen $F_\sigma(z,\alpha,p,q):=F(z,\alpha,p+\sigma,q)$ instead.
\end{remark}

\begin{remark}\label{rem:generalNLFP}
 The arguments in the proof of Theorem~\ref{thm:refreg}~\ref{it:reg1} can be generalised
  to the problem of the GBFP equation $\mathcal{G}(u)=0$ (subject to the same Cauchy--Dirichlet conditions) whenever $h$ satisfies the hypotheses in Thm~\ref{thm:BEFPtype}. Let us sketch how to argue in the general case. The family  $(v_\sigma)$ of approximate solutions is constructed analogously, where one can choose, for instance, as regularised problem $G_\sigma(z,\alpha,p,q):=G(z,\alpha,p+\sigma,q)$.
  Of course, we cannot expect to obtain the uniform bound~\eqref{eq:boundsForApprox} (as $h$ may have rapid growth at infinity), but notice that in order to ensure compactness it is sufficient to deduce equicontinuity in $x$ of the family $(\partial_xv_\sigma)_{\sigma\in(0,1)}$. To see the latter, define the continuous function $\kappa:[0,\infty)\to[0,\infty)$ via 
  \begin{align*}
    \kappa(v)=(v^3h(1/v))^{-1},
  \end{align*}
  observe that $\kappa$ is strictly positive for $v>0$,
and then consider the strictly increasing function
\begin{align*}
 K(v)=\int_0^v\kappa(s)\,\d s, \quad v\ge0,
\end{align*}
which satisfies $K(0)=0$. 
Then the equation $\mathcal{G}_\sigma(v_\sigma)=0$ and the fact that $[v_\sigma]_{C^{0,1}}\le C_1([u]_{C^{0,1}})$ yield the bound
\begin{align*}
  |\kappa(\partial_xv_\sigma)\partial_x^2v_\sigma|
  &\le \sup_{\sigma\in(0,1)}\left(\kappa(\partial_xv_\sigma+\sigma)|\partial_x^2v_\sigma|\right)
  \\& \le C([u]_{C^{0,1}(\Om)}, R)
\end{align*}
and thus 
\begin{align*}
  \left\|\frac{\d}{\d x}K(\partial_xv_\sigma)\right\|_{L^\infty(\Om)}\le C([u]_{C^{0,1}(\Om)}, R)=:C_2,
\end{align*}
so that $K(\partial_xv_\sigma)$ is Lipschitz continuous with respect to $x$ uniformly in $\sigma$ with constant bounded above by $C_2$.
In the following we let $C_1:=C_1([u]_{C^{0,1}})+1$ and denote the inverse of $K_{|[0,C_1]}:[0,C_1]\to[0,K(C_1)]$ by $K^{-1}$. Then $\partial_xv_\sigma=K^{-1}\circ(K\circ \partial_xv_\sigma)$, and denoting for a uniformly continuous function $a$ by $\vartheta_a$ its modulus of continuity, we infer that
\begin{align*}
  \vartheta_{\partial_xv_\sigma(t,\cdot)}(\delta)\le \vartheta_{K^{-1}}(C_2\delta) \quad \text{ for  }\delta>0.
\end{align*}
Now compactness is obtained from the Arzel\`a--Ascoli theorem, so that the Aubin--Lions lemma applies as before and yields the bound
\begin{align*}
  \left\|\frac{\d}{\d x}K(\partial_xu)\right\|_{L^\infty(\Om)}\le C([u]_{C^{0,1}(\Om)}, R)=:C_2
\end{align*}
as well as the regularity $\partial_xu\in C(\bar \Om)$. Here $\frac{\d}{\d x}K(\partial_xu)$ denotes the weak derivative of $K(\partial_xu)$ with respect to $x$.
Let us also mention that the main conclusions in Sec.~\ref{sssec:nondeg} below apply to more general $h$.
For simplicity, we only consider the case of the explicit function $h=h_\gamma$, which is in particular smooth in $(0,\infty)$. 
\end{remark}

\subsubsection{The set \texorpdfstring{$\Om^+\setminus\Om^{++}$}{} is empty}\label{sssec:nondeg}

\begin{proof}[Proof of Thm~\ref{thm:refreg}~\texorpdfstring{\ref{it:reg2}}{}]
Since $u,\partial_xu\in C(\Om)$, the sets 
\begin{align*}
  \Om^{+}=\{\om\in\Om:\;|u(\om)|>0\}
\end{align*}
and 
\begin{align*}
  \Om^{++}=\{\om\in\Om:\;\partial_xu(\om)>0\}
\end{align*} 
are open.
From estimate~\eqref{eq:2ndDerUi} we infer that in any open rectangle $\Om'\subset\subset \Om^{++}$
we have $\partial_x^2u\in L^\infty(\Om')$.
Arguing as for $v_\sigma$ (see Section~\ref{ssec:approx}), 
it follows that $u_{|\Om'}$ is a strong solution of a uniformly parabolic equation in $\Om'$ (where the equality holds in $L^\infty(\Om')$).
This allows us to apply classical regularity theory for quasilinear parabolic equations to deduce that $u$ is smooth in $\Om^{++}$: indeed,
take an axis-aligned rectangle $\Om'\subset\subset\Om^{++}$.
Then, recalling the uniqueness of (viscosity) solutions $v$ to the Cauchy--Dirichlet problem $\mathcal{F}(v)=0$ in $\Om'$, $v=u$ on $\partial\Om'$ 
and the fact that $u(\cdot,x)$ is Lipschitz continuous for any $x$ and $\partial_xu(t,\cdot)$ is $\beta$-H\"older continuous for any $t$, as established in part~\ref{it:reg1} of Thm~\ref{thm:refreg}, the results~\cite[Theorems~8.2 \& 8.3]{lieberman_parabolic} imply local Schauder regularity for $u$ in $\Om'$ and, in particular, the regularity $u\in C^{1,2}_{t,x}(\Om')$.
Then, iterating the argument in the proof of~\cite[Lemma~14.11]{lieberman_parabolic}  (successively applied to the equation satisfied by $\partial_x^ku, k\in\mathbb{N}_0$)
one deduces the regularity $u_{|\Om'}\in C^\infty(\Om')$. 

Now define $\mathcal{N}:=\Om^+\setminus\Om^{++}$. 
Our goal is to show that $\mathcal{N}$ is empty. We proceed indirectly supposing that there exists a point $\om=(t,x)\in\mathcal{N}$, where---by the symmetry of the equation---we may assume without loss of generality that $u(\om)>0$.
From now on, we fix this particular time $t$, define $v(y)=u(t,y)$, $J':=(x_0,x]$, where $x_0:=\max\{y\in J:u(t,y)=0\}$, and the non-empty set 
\begin{align}\label{eq:defA}
 A:=J'\setminus(\Om^{++})_t,
\end{align}
where $(\Om^{++})_t:=\{y\in J:(t,y)\in\Om^{++}\}$ denotes the cross section of $\Om^{++}$ at $t$.
We call a point $y\in A$ a \textit{left-isolated} point (of $A$)
if there exists $\delta>0$ such that $(y-\delta,y)\subset J'\setminus A$. 
Notice that in this case $(y-\delta,y)\subset (\Om^{++})_t$, so that 
$v$ is smooth in $(y-\delta,y)$.

\begin{lemma}\label{l:qnUpperBd}
  Let $A$ be defined by formula~\eqref{eq:defA} and suppose that $y\in A$.   Then, there cannot exist a sequence $x_n\to y$ with the property that for every $n$ there are   $(p_n,q_n)\in \mathcal{J}^{2,+}(u(t,\cdot))(x_n)$, where $p_n:=\partial_xu(t,x_n)$, satisfying $q_n\le 0$. 
\end{lemma}
\begin{proof}
 We argue by contradiction and assume that such a sequence $x_n\to y$ exists.
 Let $z:=u(t,y)>0$ and choose $\sigma>0$ small enough such that
\begin{align}\label{eq:ineq-N-empty}
  -\sigma^\gamma K+{z}/{2}>0,
\end{align}
where $K:=\|\partial_tu\|_{L^\infty(\Om)}$. Next, fix some sufficiently large $n$ such that $u(t,x_n)\ge z/2$, $\partial_xu(t,x_n)\le\sigma$ and choose $(p_n,q_n)\in \mathcal{J}^{2,+}(u(t,\cdot))(x_n)$ such that $q_n\le 0$.
Then there exists a function $\phi\in C^2(J)$ satisfying $u(t,\cdot)-\phi\le u(t,x_n)-\phi(x_n)=0$ and $\phi'(x_n)=p_n, \phi''(x_n)=q_n$.
After possibly replacing $\phi$ with $\tilde\phi(y):=\phi(y)+|x_n-y|^4$, we can assume that the maximum of $u(t,\cdot)-\phi$ at $x_n$ is strict.

Now consider for some small $\delta>0$ the function 
\begin{align*}
  w(s,y):=u(s,y)-\left(\phi(y)+\frac{1}{2\varepsilon}|s-t|^2\right) \;\;\text{ in }\;\;Q_\delta:=[t-\delta,t+\delta]\times[x_n-\delta,x_n+\delta],
\end{align*}
which, by continuity, reaches its (non-negative) maximum at some point $(s_\varepsilon,y_\varepsilon)$. Notice that $s_\varepsilon\to t$ as $\varepsilon\to0$ and, moreover, $y_\varepsilon\to x_n$. 
In particular, $(s_\varepsilon,y_\varepsilon)\in \mathrm{int}(Q_\delta)$ for small enough $\varepsilon>0$, so that
\begin{align*}
  (0,0,0)\in\mathcal{P}^+(w)(s_\varepsilon,y_\varepsilon)
\end{align*}
or, equivalently,
\begin{align*}
  \left(\frac{s_\varepsilon-t}{\varepsilon},\phi'(y_\varepsilon),\phi''(y_\varepsilon)\right)\in\mathcal{P}^+u(s_\varepsilon,y_\varepsilon).
\end{align*}
Since $|\frac{s_\varepsilon-t}{\varepsilon}|\le K$, there exists $\bar\tau\in[-K,K]$ and a sequence $\varepsilon_i\to0$ such that $\frac{s_{\varepsilon_i}-t}{\varepsilon_i}\to\bar\tau$.
Letting $i\to\infty$, we find 
\begin{align*}
  \left(\bar\tau,p_n,q_n\right)\in\overline{\mathcal{P}}^+u(t,x_n).
\end{align*}
The subsolution property of~$u$, the fact that $q_n\le0$ and the choice of $n$ now imply the inequality 
\begin{align*}
 -\sigma^\gamma K + z/2\le0,
\end{align*}
which contradicts~\eqref{eq:ineq-N-empty}.
\end{proof}

Thanks to Lemma~\ref{l:qnUpperBd}, we have 
\begin{lemma}\label{l:leftisol}
 There cannot be any left-isolated point in the set $A$ (defined in formula~\eqref{eq:defA}). 
\end{lemma}
\begin{proof}
 We argue again by contradiction, assuming that there exists a point $y\in A$ and $\delta>0$ such that $(y-\delta,y)\subset J'\setminus A$. Then $v'$ is strictly positive and smooth in $(y-\delta,y)$ and reaches its global minimum at the point $y$. 
  Hence, there exists a strictly increasing sequence $(y-\delta,y)\ni\tilde  x_n\nearrow y$, $n\ge0$, such that $\left(v'(\tilde x_n)\right)_n$ is strictly decreasing. 
  Now for $n\ge1$ let $y_n:=\tilde x_n$ and $\varepsilon_n:=\tilde x_n-\tilde x_{n-1}>0$. We then have 
  \begin{align*}
    v'(y_n)-v'(y_n-\varepsilon_n)=v'(\tilde x_n)-v'(\tilde x_{n-1}) < 0 
  \end{align*}
and thus 
\begin{align*}
 \frac{v'(y_n)-v'(y_n-\varepsilon_n)}{\varepsilon_n} < 0 
\end{align*}
for all $n\ge1$.
Since $v'$ is absolutely continuous in $(y-\delta,y)$, we then have 
\begin{align*}
  \frac{1}{\varepsilon_n}\int_{y_n-\varepsilon_n}^{y_n}v''(z)\,\d z=\frac{v'(y_n)-v'(y_n-\varepsilon_n)}{\varepsilon_n} < 0.
\end{align*}
Hence, there exists $x_n\in(y_n-\varepsilon_n,y_n)$ such $q_n:=v''(x_n)<0$. In particular, letting $p_n:=v'(x_n)$, we have $(p_n,q_n)\in \mathcal{J}^{2}v(x_n)$ and by construction $x_n\to y$ as $n\to\infty$.
This contradicts Lemma~\ref{l:qnUpperBd}.
\end{proof}

Notice that the case $A=J'$ is impossible. 
Therefore, there exists $y\in J'\setminus A$.
Now let
$y_1:=\min\left(A\cap[y,x]\right)$, which exists since $x\in A$ and since, by the continuity of $v'$, $A$ is relatively closed in $J'$. Then $y_1>y$, which implies that $y_1\in A$ is left-isolated, contradicting Lemma~\ref{l:leftisol}.

\medskip

We therefore conclude
\begin{align*}
  \Om^+\setminus \Om^{++}=\emptyset.
\end{align*}
The proof of Thm~\ref{thm:refreg}~\texorpdfstring{\ref{it:reg2}}{} is now complete.
\end{proof}

\subsection[Relation to the original equation]{Relation to the original equation on a bounded interval}\label{sec:relToOrgBdd}

For fixed $\gamma\ge2$, $m,R>0$ and an initial datum $u_0$ admissible for problem~\eqref{eq:bosonicFPmassCHP4} in the sense of Def.~\ref{def:admissBosonicFP}, we denote by $u$ the unique global-in-time viscosity solution of the Cauchy--Dirichlet problem~\eqref{eq:bosonicFPmassCHP4}.
In the previous subsection we have seen that $u$ has the improved regularity properties~\ref{it:reg1} and~\ref{it:reg2} of Thm~\ref{thm:refreg}. In particular, $\partial_xu\in C([0,\infty)\times[0,m])$.
In this section we investigate the conclusions which can be drawn from our theory established at the level of $u$ for the problem in the original variables. 
Let us recall the definition~\eqref{eq:defGenInv} of the generalised inverse $M(t,\cdot)$ of $u(t,\cdot)$ as well as the definition~\eqref{eq:mu(t)} of the finite measure $\mu(t)$ on $[-R,R]$ associated with $M(t,\cdot)$:
\begin{align}\label{eq:defMANDmu}
  \begin{cases}
    M(t,r)=\max u(t,\cdot)^{-1}(r),& \; r\in [-R,R],
\\ \mu(t)([-R,r]) = M(t,r),& \; r\in[-R,R].
             \end{cases}
\end{align} 
As seen in Remark~\ref{rem:originalVariables}, the function $M(t,\cdot)$ is strictly increasing and right continuous on $[-R,R]$ and satisfies $M(t,-R)\ge0, M(t,R)=m$. In particular, the total mass of the measure $\mu(t)$ equals $m$ for all $t\ge0$.
Thanks to Thm~\ref{thm:refreg}, we now have a much more detailed understanding of $M(t,\cdot)$ and $\mu(t)$:
\begin{proposition}\label{prop:uMmu}
  Let $\gamma\ge2$, $m,R>0$, assume that $u_0$ is admissible for problem~\eqref{eq:bosonicFPmassCHP4} in the sense of Def.~\ref{def:admissBosonicFP}, let $u$ denote the unique viscosity solution of problem~\eqref{eq:bosonicFPmassCHP4} and define $M(t,\cdot)$ and $\mu(t)$ as in~\eqref{eq:defMANDmu}. The following holds true:
  \begin{enumerate}[\normalfont{(\roman{*})}]
   \item \label{it:x} For each $t>0$ there exist unique points $x_-(t),x_+(t)\in(0,m)$, $x_-(t)\le x_+(t)$,  such that $$[u(t, x)=0\;\Leftrightarrow \; x_-(t)\le x\le  x_+(t)].$$
   In addition, by~\ref{it:reg2} of Thm~\ref{thm:refreg}, we have 
   $$\partial_x u(t,x)>0\text{ for }x\in(0,m)\setminus[x_-(t),x_+(t)].$$  
   \item \label{it:propM} For each $t>0$ the strictly increasing and right-continuous function $M(t,\cdot)$ satisfies
   $$M(t,0-)=x_-(t)\text{ and }M(t,0)=x_+(t).$$
   Moreover, $M$ is $C^\infty$ in the set $\{(t,r):t>0,|r|\in(0,R)\}$.
   \item \label{it:propf} 
   Letting $x_p(t):= \mathcal{L}^1(\{u(t,\cdot)=0\})$, for each $t>0$ there exists a unique, 
   positive function $f(t,\cdot)\in L^1(-R,R)$ such that the measure $\mu(t)\in\mathcal{M}^+_b([-R,R])$ has the decomposition
   \begin{align}\label{eq:decmu}
\mu(t)=x_p(t)\delta_0 + f(t,\cdot)\mathcal{L}^1,\quad t\in(0,\infty).
\end{align}
Furthermore, $f(t,\cdot)\in C^\infty((-R,R)\setminus\{0\})$,
\begin{align}\label{eq:relfu}
  \begin{cases}
 f(t,u(t,x))=1/\partial_x u(t,x)&\text{ for }x\in(0,m)\setminus[x_-(t),x_+(t)],
 \\ f(t,r)=1/\partial_x u(t,M(t,r))&\text{ for }|r|\in(0,R),
  \end{cases}
\end{align}
and the function $f$ satisfies
\begin{align}\label{eq:dens}
 \partial_tf-\partial_r(\partial_rf+rh_\gamma(f))=0,\quad t>0, |r|\in(0,R),
\end{align}
in the classical sense.
 \end{enumerate}
\end{proposition}

\begin{figure}[ht]\centering
   \resizebox{.8\linewidth}{!}{     
  \raisebox{-0.5\height}{
  \begin{tikzpicture}[line cap=round,line join=round,>=triangle 45]
\begin{axis}[
    anchor=origin,  
    x=0.5cm,y=2cm,
    hide axis
]
\pgfplotstableread{tikzData/simU1_15.txt}
\asymIV
\addplot[smooth,line width=0.45mm] table[y = u2(x), x = x] from \asymIV node[left]{\Large$u(t,x)\hspace{1cm}$};
\end{axis}
\draw[color=black,loosely dotted, line width=0.25mm] 
 (2.15,-2.3)node[below,black]{$x_-(t)\quad$}  -- (2.15,0)
(2.57,-2.3)node[below]{$\qquad x_+(t)$} -- (2.57,0)
(0,0)node[left]{$0\;\;\;$} -- (2.7,0);
\end{tikzpicture}
  }
  \hspace*{2cm} 
  \raisebox{-0.5\height}{
    \begin{tikzpicture}[line cap=round,line join=round,>=triangle 45,>=stealth']
\begin{axis}[
    anchor=origin,  
    x=2cm,y=0.5cm,
    hide axis,
    ymax = 15, xmax = 2
]
\pgfplotstableread{tikzData/M1.txt}
\Mi
\pgfplotstableread{tikzData/M2.txt}
\Mii
\addplot[smooth,line width=0.45mm] table[y = x, x = u2(x)] from \Mi;
\addplot[smooth,line width=0.45mm] table[y = x, x = u2(x)] from \Mii node[above]{\Large$M(t,r)$};
\draw[line width=5mm] (100,52) node{$\bullet$};
\end{axis}
\begin{axis}[
    anchor=origin,  
    x=2cm,y=0.5cm,
    hide axis,
    ymax = 15,xmax = 2
]
\pgfplotstableread{tikzData/simU1_15.txt}
\asymIV
\addplot[smooth,line width=0.45mm,gray!50] table[y = f2, x = u2(x)] from \asymIV node[gray!50!black,right]{\Large$f(t,r)\qquad$};
\end{axis}
\draw[color=black,loosely dotted, line width=0.25mm] (0,0)node[below]{$r=0$}  -- (0,2.);
\end{tikzpicture}
  }
}
\vspace{.5cm}
\caption[Relation between $u(t,\cdot), M(t,\cdot)$ and $f(t,\cdot)$]{Relation between $u(t,\cdot)$, its generalised inverse $M(t,\cdot)$ 
and the density $f(t,\cdot)$ of the absolutely continuous component of the measure $\mu(t)$ associated with $M(t,\cdot)$ as introduced in Prop.~\ref{prop:uMmu}.
}
\label{fig:uMf}
\end{figure}
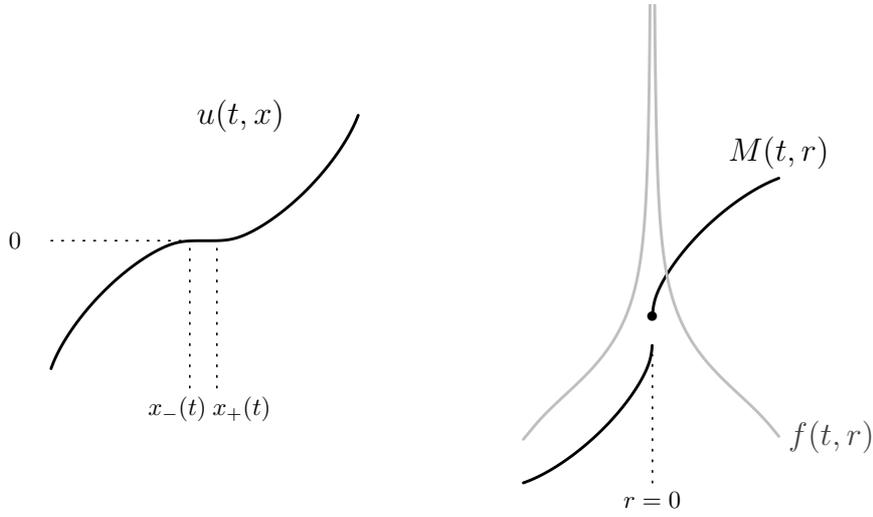

\noindent The proof of Prop.~\ref{prop:uMmu} is elementary and will be omitted.

Let us also note that we have regularity up to the boundary in the following sense.
\begin{lemma}[Regularity up to the boundary]\label{l:bdryreg}
 Under the assumptions of Proposition~\ref{prop:uMmu}, there exists $\sigma>0$ only depending on the initial datum $u_0$ such that for all $t>0$
 \begin{align}\label{eq:uxsig}
   \partial_xu(t,y)\ge \sigma\qquad\text{for }y\in\{0,m\}.
 \end{align}
Suppose now that, in addition, 
\begin{enumerate}[label=\normalfont{(I\arabic{*})}]  
   \item\label{it:I1}  there exists $\alpha>0$ such that 
 $u_0\in C^{2,\alpha}(\bar J)$. 
\item\label{it:I2} $u_0$ 
satisfies the compatibility condition $\mathcal{F}(u_0)_{|x}=0$ for $x\in\{0,m\}$.
\end{enumerate}
Then for any $T<\infty$ and $\Om:=(0,T)\times (0,m)$ there exists a neighbourhood $V$ of $\partial_p\Om$ in $\bar \Om$ such that $u$ has parabolic Schauder regularity in $V$, i.e.
\begin{align*}
  u_{|V}\in H_{2+\alpha}(\bar V) \subset C^{1,2}_{t,x}(V).
\end{align*}
As a consequence, in this case $\partial_rf\in C([0,\infty)\times ([-R,R]\setminus\{0\}))$ and 
 \begin{align}\label{eq:nofluxbcf}
   \partial_rf+rh_\gamma(f) = 0 \quad \text{ in }[0,\infty)\times\{-R,R\}.
 \end{align}
\end{lemma}
\begin{proof}
 Regarding the first part, we show that assertion~\eqref{eq:uxsig} is satisfied on the left lateral boundary, i.e.\;that there exists $\sigma>0$ such that $\partial_xu(t,0)\ge\sigma>0$ for all~$t$. 
  The uniform bound $\inf_t\partial_xu(t,m)\ge\sigma'>0$ can be deduced analogously (or by symmetry).
  For any $a>0$ and $b\in(0,a]$ the time-independent function 
  \begin{align*}
    u_1(x) =  u^{(R,m+a)}_\infty(x+b) - u^{(R,m+a)}_\infty(b) -R,\qquad x\in[0,m] 
  \end{align*}
is a viscosity subsolution of $\mathcal{F}=0$ in $(0,\infty)\times(0,m)$ satisfying 
$u_1(0)=-R$, $u_1(m)\le R$.
It is easy to see that, by the admissibility of the initial datum $u_0$, $a>0$ and $b\in(0,a]$ can be chosen such that we have the bound $u_1\le u_0$ \textit{as well as} the non-degeneracy $\sigma:=\partial_xu_1(0)>0$.
 Hence $ u_1\le u(t,\cdot)$ for all $t\ge0$ and therefore $\partial_xu(t,0)\ge\sigma$.

 The regularity of $u$, asserted under the extra assumptions~\ref{it:I1},~\ref{it:I2}, is a consequence of~\cite[Theorems~8.2 \& 8.3]{lieberman_parabolic} and the fact that, by continuity, a lower bound of the form~\eqref{eq:uxsig} (with $\sigma$ replaced by some $\sigma'\in(0,\sigma)$) holds true in a neighbourhood $V$ of $\partial_p\Om\subset\bar\Om$.
The zero-flux boundary condition~\eqref{eq:nofluxbcf} is now deduced as follows: first notice that, by the non-degeneracy near the boundary, close to the boundary the equation $\mathcal{F}(u)=0$ can be rewritten as
\begin{align*}
  \partial_tu-(\partial_xu)^{-2}\partial_x^2u +u((\partial_xu)^{-\gamma}+1)=0.
\end{align*} 
On the other hand, the constant-in-time lateral boundary conditions $u(\cdot,0)=-R$, $u(\cdot,m)=R$ combined with the continuity of $\partial_tu, \partial_x^2u$ in $V$ yield the identity $\partial_tu=0$ on $S:=(0,\infty)\times\{0,m\}$. 
Hence,
\begin{align*}
  -(\partial_xu)^{-2}\partial_x^2u +u((\partial_xu)^{-\gamma}+1)=0\quad \text{on }S.
\end{align*}
Reformulating the last identity in terms of $f$ leads to equation~\eqref{eq:nofluxbcf}. 
\end{proof}

\subsubsection{Spatial blow-up profile}\label{ssec:spatialProfile}

Here, we will establish 
\begin{proposition}[Blow-up profile]\label{prop:profile}
 Assume the hypotheses and use the notations of Prop.~\ref{prop:uMmu}. 
Then, if $\gamma>2$, for any $t>0$ the following holds true:
 \begin{enumerate}[label=\normalfont{(\roman*)}]
  \item\label{it:spatialBUprofile} 
 \textit{Time-uniform spatial bound:}
there exists a constant $C=C(R,\gamma,\|u\|_{C^{0,1}(\Om)})$ such that for all $t>0$ and $|r|\in(0,R)$
\begin{align}\label{eq:fupperbd}
  f(t,r)\le C|r|^{-\frac{2}{\gamma}}.
\end{align}
  Spatial behaviour near singularity: if $f(t,\cdot)$ is unbounded near the origin (or equivalently $\partial_xu(t,x_\pm(t))=0$), then 
  \begin{align}\label{eq:fprecise}
  f(t,r)=\left(\frac{\gamma}{q(t,r)}\int_0^rsq(t,s)\,\d s\right)^{-\frac{1}{\gamma}},
\end{align}
where for $|r|\in(0,R)$
\begin{align}
 q(t,r)&=\exp\left(\int_{0}^r a(t,s)\,\d s\right), \label{eq:def_q}
 \\a(t,r)&=-\gamma(\tau(t,r)+r), \label{eq:def_a}
 \\\tau(t,r)&=\partial_tu(t,M(t,r)). \label{eq:def_tau}
\end{align}
  In particular, the expansion
  \begin{align}\label{eq:profileThesis1}
  f(t,r)=\left(\frac{2}{\gamma}\right)^\frac{1}{\gamma}|r|^{-\frac{2}{\gamma}}\left(1+O(|r|)\right)\quad\text{as $r\to0$}
\end{align}
holds true uniformly in such $t$.

Furthermore,
\begin{align}\label{eq:uxx<0}
 \partial_x^2u(t,\cdot)=\left(\frac{\gamma}{q(u)}\int_0^usq(s)\,\d s\right)^{\frac{1}{\gamma}-1}
 \partial_x u\left(u-\frac{q'(u)}{(q(u))^2}\int_0^usq(s)\,\d s\right),
\end{align}
where, for simplicity, we dropped the time argument on the right-hand side of eq.~\eqref{eq:uxx<0}.
In particular, there exists a constant $c=c(\|u\|_{C^{0,1}(\Om)})\in(0,R)$ such that 
\begin{align}\label{eq:uxxPos}
  \partial_x^2u(t,\cdot)\cdot\mathrm{sign}(u(t,\cdot)) > 0 \qquad \text{ in }\{0<|u(t,\cdot)|<c\}.
\end{align}
\item\label{it:ptMassContinuous} The function $t\mapsto x_p(t)$, denoting the size of the condensate, is continuous.
\end{enumerate}
 In the $L^1$-critical case, $\gamma=2$, solutions are globally regular and condensates cannot form:
\begin{enumerate}[resume,label=\normalfont{(\roman*)}]
 \item\label{it:gamma2bdd}  If $\gamma=2$, the density $f(t,\cdot)$ is bounded and smooth in $(-R,R)$ for all $t\in(0,\infty)$.  In particular, in this case $\min_{[0,m]}\partial_xu(t,\cdot)>0$ for all $t>0$, and $f$ satisfies the problem~\eqref{eq:befpBdd} in the classical sense.
\end{enumerate} 
\end{proposition}

\begin{proof}[Proof of Prop.~\ref{prop:profile}]
We fix an arbitrary time $t>0$. For $x>x_+(t)$ we let $r=u(t,x)$, $\tau=\partial_tu(t,M(t,r))$,
$p=\partial_x u(t,M(t,r))$ and $q=\partial_x^2u(t,M(t,r))$. Notice that $r,p>0$ and that $\tau=\tau(r)$
defines a bounded function on $(0,R)$.
We have
  $p^\gamma\tau-p^{\gamma-2}q+r(1+p^\gamma)=0$
and thus 
\begin{align}\label{eq:taupq}
  \tau -p^{-2}q+r(p^{-\gamma}+1)=0.
\end{align} 
In the following the fixed time argument $t$ will be dropped.
From the identity $f(u)=\frac{1}{\partial_x u}$, we deduce
\begin{align*}
  \frac{f'(u)}{f(u)}=-\frac{\partial_x^2u}{(\partial_x u)^2},
\end{align*}
so that equation~\eqref{eq:taupq} can be rewritten as
\begin{align}\label{eq:odef}
\frac{f'(r)}{f(r)} + rf^\gamma(r) = -\tau(r)-r.
\end{align}
For later reference, we recall that in eq.~\eqref{eq:odef} we have dropped the time argument and abbreviated $f':=\partial_rf$. We further note that $|\tau(t,r)|\le \|u\|_{C^{0,1}(\bar\Om)}\le C(u_0)<\infty$.

Letting $k(r):=f^{-\gamma}(r)$, which, by the regularity of $u$, is well-defined, bounded and strictly positive for $r\in(0,R)$, the last equation becomes
\begin{align*}
  -\frac{1}{\gamma}\frac{k'(r)}{k(r)} + rk^{-1}(r) = -\tau(r)-r,
\end{align*}
or, equivalently,
\begin{align}\label{eq:odek}
  k'(r) +a(r)k(r)= \gamma r,
\end{align}
where we abbreviated $a(r):=-\gamma(\tau(r)+r)$.
Introducing $q(r):=q(t,r)$, where
\begin{align*}
 q(t,r)=\exp\left(\int_{0}^r a(t,s)\,\d s\right),
\end{align*}
the left-hand side of eq.~\eqref{eq:odek} equals $\frac{1}{q}\left(q\cdot k\right)'$.
Hence, upon integration over the interval $(\ve,r)$, where $0<\varepsilon<r$, 
\begin{align*}
 (qk)(r)=(qk)(\varepsilon)+\gamma \int_{\varepsilon}^rsq(s)\,\d s.
\end{align*}
Thus,
\begin{align}\label{eq:epsas}
  f(r)=\left(\frac{q(\varepsilon)k(\varepsilon)}{q(r)}+\frac{\gamma}{q(r)}\int_\varepsilon^rsq(s)\,\d s\right)^{-\frac{1}{\gamma}}.
\end{align}
Since $\partial_xu(t,\cdot)\in C([0,m])$, the limit $f^{-\gamma}(t,0):=\lim_{\ve\to0}k(t,\ve)$ exists in $[0,\infty)$. Thus, eq.~\eqref{eq:epsas} yields the identity
\begin{align}\label{eq:expgenII}
  f(t,r)=\left(\frac{f^{-\gamma}(t,0)}{q(t,r)}+\frac{\gamma}{q(t,r)}\int_0^rsq(t,s)\,\d s\right)^{-\frac{1}{\gamma}}, 
\end{align}
which implies inequality~\eqref{eq:fupperbd}. As a side note, we observe that formula~\eqref{eq:expgenII} provides an alternative means to deduce the non-degeneracy~\eqref{eq:uxsig} and to quantify the lower bound~$\sigma$.

Let us now suppose that
$\limsup_{r\searrow0}f(t,r)=\infty$. By the continuity of $\partial_xu(t,\cdot)$, we infer that $\lim_{\varepsilon\to0}f^{-\gamma}(t,\varepsilon)=0$ and thus, by~\eqref{eq:epsas}, 
\begin{align}\label{eq:fpreciseV2}
  f(t,r)=\left(\frac{\gamma}{q(t,r)}\int_0^rsq(t,s)\,\d s\right)^{-\frac{1}{\gamma}},
\end{align}
where $q$ is given by~\eqref{eq:def_q}. Since $q(t,r)=1-\gamma \tau(t,r)r+O(r^2)$ as $r\to0$ with uniform control in $t$,
we infer the behaviour
\begin{align}\label{eq:SingProfile}
  f(t,r)=\left(\frac{2}{\gamma}\right)^\frac{1}{\gamma}r^{-\frac{2}{\gamma}} \left(1+O(r)\right)\hspace{.55cm}\text{ as $r\to0^+$},
\end{align}
which again holds true uniformly in $t$ (provided $f(t,\cdot)$ is unbounded at $r=0$).
 
By identity~\eqref{eq:fpreciseV2}, 
\begin{align}\label{eq:uxasymp}
 \partial_x u=\left(\frac{\gamma}{q(u)}\int_0^usq(s)\,\d s\right)^{\frac{1}{\gamma}},
\end{align}
and thus $\partial_xu=\left(\frac{\gamma}{2}\right)^\frac{1}{\gamma}u^\frac{2}{\gamma}(1+O(u))$ as $u\searrow0$. 
Differentiation of~\eqref{eq:uxasymp} further yields
\begin{align}\label{eq:uxx<0V2}
 \partial_x^2 u=\left(\frac{\gamma}{q(u)}\int_0^usq(s)\,\d s\right)^{\frac{1}{\gamma}-1}
 \partial_x u\left(u-\frac{q'(u)}{(q(u))^2}\int_0^usq(s)\,\d s\right),
\end{align}
from which we observe that $\partial_x^2 u>0$ for sufficiently small $0<u\le c(\|\tau\|_{L^\infty})$.
 The asymptotics in the region $0\ll x<x_-(t)$ are derived analogously.  
This completes the proof of assertion~\ref{it:spatialBUprofile}.

Assertion~\ref{it:ptMassContinuous}, the continuity of $t\mapsto x_p(t)$, is a simple consequence of the bound~\eqref{eq:fupperbd}.

Let us finally suppose that $\gamma=2$ and prove assertion~\ref{it:gamma2bdd}.
Assuming, by contradiction, that there exists a time $T\in(0,\infty)$ such that 
  $f(T,\cdot)$ is unbounded near the origin, identity~\eqref{eq:SingProfile} implies that 
  $f(T,r)\ge r^{-1}/2$ for small enough $r>0$. This contradicts the fact that $\|f(T,\cdot)\|_{L^1(-R,R)}\le m$.
\end{proof}

\subsection{Entropy dissipation identity}\label{ssec:entropy}

In this subsection we aim to study the time evolution of
$\mathcal{\widetilde H}^{(h_\gamma,R)}(\mu(t))$. Here $\mathcal{\widetilde H}^{(h,R)}$ denotes the natural extension of $\mathcal{H}^{(h,R)}$ to $\mathcal{M}_b^+([-R,R])$ as described in Remark~\ref{rem:emin}, where $\mathcal{H}^{(h,R)}$ is defined by formula~\eqref{eq:entr}.
Observe that, by~\eqref{eq:decmu}, the entropy does not explicitly depend on the singular component of $\mu(t)$ and thus coincides with $\mathcal{H}^{(h_\gamma,R)}(f(t,\cdot))$:
\begin{align*}
 \mathcal{\widetilde H}^{(h_\gamma,R)}(\mu(t))=\int_{-R}^R\left(\frac{r^2}{2}f(t,r)+\Phi(f(t,r))\right)\d r.
\end{align*} 

\begin{proposition}[Entropy dissipation identity]\label{prop:ediBallFull}
  Suppose the hypotheses and use the notations of Proposition~\ref{prop:uMmu}. Further assume that $u_0$ satisfies hypothesis~\ref{it:I1} and~\ref{it:I2} of Lemma~\ref{l:bdryreg}.
  Then the function $t\mapsto \mathcal{\widetilde H}^{(h_\gamma,R)}(\mu(t))=\mathcal{H}^{(h_\gamma,R)}(f(t,\cdot))$ is absolutely continuous, and the identity
\begin{align}\label{eq:entrIdv2}
  \mathcal{H}^{(h_\gamma,R)}(f(t,\cdot)) -
\mathcal{H}^{(h_\gamma,R)}(f(s,\cdot)) =
  -\int_s^t\int_{-R}^R\frac{1}{h_\gamma(f)}|\partial_rf+rh_\gamma(f)|^2\,\dd r\,\dd \sigma
\end{align} 
  holds true for all $0\le s\le t<\infty$.
\end{proposition}

\begin{proof}
  We will derive eq.~\eqref{eq:entrIdv2} via approximation by a regularised problem. For convenience, our regularisations are based on the setting in Section~\ref{ssec:applications}, where the superlinearity $h_\gamma$ in the drift is attenuated in such a way that is has critical growth at infinity (i.e.\;$h(s)\approx s^3$ as $s\to\infty$).
  The smoothness of the approximate solutions then follows from the theory established in Sections~\ref{sec:impreg},~\ref{ssec:spatialProfile}.
  In order to deduce equality, we will introduce two entropy-type functionals approximating from above resp.\;from below the original problem.
  The approximation from above leads to an entropy dissipation \textit{inequality} which is crucial for the long-time asymptotic behaviour. Here, the passage to the limit relies on the lower semicontinuity properties of the original entropy. 
  
Let $\beta:=\gamma-2$ and take a smooth, non-decreasing function $\eta\in C^\infty(0,\infty)$ satisfying the identities
\begin{align*}
  \eta(\sigma)=\begin{cases}
    \sigma^\beta & \text{ if }\sigma\le1,
    \\\left(\frac{3}{2}\right)^\beta & \text{ if }\sigma\ge2
              \end{cases}
\end{align*}
as well as the bound
\begin{align*}
  \eta(\sigma)\le \sigma^\beta\qquad\text{ for all } \sigma\ge0.
\end{align*}
Then define $\eta_\ve(s)=\ve^{-\beta}\eta(\ve s)$ and set $\varphi_{\ve}(s)=s(1+s^2\eta_\ve(s))$. 
Notice that, by definition, $\varphi_\ve(s)=h_\gamma(s)$ for $s\le\frac{1}{\ve}$ and $\varphi_\ve\le h_\gamma$ on $[0,\infty)$.
The function $h=\varphi_\ve, 0<\ve\ll1,$ satisfies the hypotheses of Theorem~\ref{thm:BEFPtype}. 
Since, by assumption, our initial datum $u_0$ satisfies $\min u_0'>0$, it trivially fulfils hypothesis~\eqref{eq:defC} for any $\ve$.
Hence, 
Theorem~\ref{thm:BEFPtype} provides us with a family $\{v_\ve\}$ of approximate solutions emanating from $u_0$, where $v_\ve$ satisfies the equation~\eqref{eq:ibcgd} with $h:=\varphi_\ve$.
By the construction of the barriers $u_{\theta,\pm,h}^{(R,m)}$ (see page~\pageref{eq:barrGen}), it is obvious that 
for small $\ve>0$ the problem based on $h:=\varphi_\ve$ has barriers $u^\pm$ which are uniformly-in-$\ve$ Lipschitz continuous in space-time.
Thus,  Theorem~\ref{thm:BEFPtype} yields the bound
\begin{align*}
  \sup_\ve\|v_\ve\|_{C^{0,1}(\Om)}<\infty,
\end{align*}
which implies that, in the limit $\ve\to0$, $\{v_\ve\}$ converges locally uniformly to our viscosity solution~$u$. Here we used the stability and uniqueness of the BFP problem at the level of $u$ as well as the observation that $G_\ve(z,\alpha,p,q)\to (1+|p|^\gamma)^{-1}F(z,\alpha,p,q)$ locally uniformly in $(z,\alpha,p,q)\in\mathbb{R}^4$, where $F$ is defined by eq.~\eqref{eq:defF} and
$$G_\ve(z,\alpha,p,q) = \left(|p|^3\varphi_\ve(1/|p|)\right)^{-1}\left(|p|^2\alpha-q\right)+z,$$ cf.~\eqref{eq:defG}.
Since $\varphi_\ve(s)\approx_\ve s^3$ for $s\ge\frac{2}{\ve}$, 
Sections~\ref{sec:impreg},~\ref{ssec:spatialProfile} and in particular 
the argument in Proposition~\ref{prop:profile}~\ref{it:gamma2bdd} show that $v_\ve$ is non-degenerate and thus regular globally in time. 
Furthermore, by parabolic regularity, the $\ve$-uniform bound~\eqref{eq:fupperbd} implies convergence of $f_\ve$ to $f$ locally uniformly in $\{r\neq0\}$, where $f_\ve(t,\cdot)$ denotes the density of the inverse of $v_\ve(t,\cdot)$.
Combined with the analogue of the equation~\eqref{eq:odef} (with $h_\gamma$ replaced by $\varphi_\ve$), this allows us to pass to a limit in the dissipated quantity, viz.
\begin{align*}
  \lim_{\ve\to0}\int_s^t\int_{(-R,R)}\frac{1}{\varphi_\ve(f_\ve)}\left|\partial_rf_\ve+r\varphi_\ve(f_\ve)\right|^2\,\dd r\dd \sigma = \int_s^tD_R(\tau)\,\dd \tau,
\end{align*}
where $D_R$ is given by
\begin{align*}
  D_R(\tau)=\int_{-R}^R\frac{1}{h_\gamma(f)}|\partial_rf+rh_\gamma(f)|^2\,\dd r.
\end{align*}

We will now define two different entropies 
\begin{align}\label{eq:entrAbove}
  \mathcal{H}_\ve(f) = \int_{(-R,R)}\left(\frac{|r|^2}{2}f(r)+\Phi_\ve(f(r))\right)\,\dd r
\end{align}
and
\begin{align}\label{eq:entrBelow}
  \mathcal{H}^{(\varphi_\ve,R)}(f) = \int_{(-R,R)}\left(\frac{|r|^2}{2}f(r)+\Phi^{(\varphi_\ve)}(f(r))\right)\,\dd r
\end{align}
in such a way that both for
 $\mathcal{H}=\mathcal{H}_\ve$ and for $\mathcal{H}=\mathcal{H}^{(\varphi_\ve,R)}$
the density $f_\ve(t,\cdot)$ of the inverse of $v_\ve(t,\cdot)$ satisfies the entropy dissipation identity
\begin{align}\label{eq:regEdissG}
    \mathcal{H}(f_\ve(t,\cdot)) \;-\; &\mathcal{H}(f_\ve(s,\cdot)) = \nonumber 
    \\&=- \int_s^t\int_{(-R,R)}\frac{1}{\varphi_\ve(f_\ve)}\left|\partial_rf_\ve+r\varphi_\ve(f_\ve)\right|^2\,\dd r\dd \sigma
  \end{align}
for all $0\le s<t<\infty$.
In order to ensure~\eqref{eq:regEdissG}, the functions $\Phi=\Phi_\ve$, $\Phi=\Phi^{(\varphi_\ve)}$ will be constructed in such a way
that they satisfy $\Phi''=\frac{1}{\varphi_\ve}$ on $(0,\infty)$ and $\Phi(0)=0$, i.e.\;they will only differ by a linear function.

The first entropy, $\mathcal{H}_\ve$, will approximate the original problem from above:
\begin{align}\label{eq:approxAbove}
  \liminf_{\ve\to0}\mathcal{H}_\ve(f_\ve(t,\cdot)) \ge \mathcal{H}^{(h_\gamma,R)}(f(t,\cdot)) \qquad\text{ for all }t\ge0.
\end{align}
The second entropy, $\mathcal{H}^{(\varphi_\ve,R)}$, is defined as in Section~\ref{ssec:applications} (see eq.~\eqref{eq:entr}) and will approximate the original problem from below:
\begin{align}\label{eq:approxBelow}
  \limsup_{\ve\to0}\mathcal{H}^{(\varphi_\ve,R)}(f_\ve(t,\cdot)) \le \mathcal{H}^{(h_\gamma,R)}(f(t,\cdot)) \qquad\text{ for all }t\ge0.
\end{align}
Since at initial time $t=0$ we have equality in~\eqref{eq:approxAbove} and in~\eqref{eq:approxBelow} (with $\liminf$ resp.\;$\limsup$ replaced by $\lim$), 
we then infer that for all $t\ge0$ 
\begin{align*}
  \mathcal{H}^{(h_\gamma,R)}(f(t,\cdot)) &= \mathcal{H}^{(h_\gamma,R)}(f(0,\cdot)) 
  - \int_0^t\int_{-R}^R\frac{1}{h_\gamma(f)}|\partial_rf+rh_\gamma(f)|^2\,\dd r\,\dd \sigma,
\end{align*} 
 which implies the assertion~\eqref{eq:entrIdv2}.

\paragraph{\normalfont\underline{Approximation from above:}}

by construction $\varphi_\ve\le h_\gamma$ and thus
\begin{align*}
  -\int_s^\infty\frac{1}{\varphi_\ve(\sigma)}\,\dd\sigma \le -\int_s^\infty\frac{1}{h_\gamma(\sigma)}\,\dd\sigma \qquad \text{ for all }s\in(0,\infty).
\end{align*}
We can therefore choose $A_\ve\ge0$ such that 
\begin{align*}
 A_\ve -\int_\frac{1}{\ve}^\infty\frac{1}{\varphi_\ve(\sigma)}\,\dd\sigma = -\int_\frac{1}{\ve}^\infty\frac{1}{h_\gamma(\sigma)}\,\dd\sigma.
\end{align*}
We now define $\Phi_\ve$ via $\Phi_\ve(s)=\int_0^s\Phi_\ve'(\sigma)\,\dd\sigma$, where we let
\begin{align*}
 \Phi_\ve'(\sigma) = A_\ve-\int_\sigma^\infty\frac{1}{\varphi_\ve}.
\end{align*} 
This ensures that
\begin{align}
  \Phi_\ve(s) =  \Phi^{(h_\gamma)}(s) \qquad\text{ for }s\in [0,\ve^{-1}]
\end{align}
 and that $\Phi_\ve\ge \Phi^{(h_\gamma)}$ in $[0,\infty)$.
Since $\Phi_\ve''=\frac{1}{\varphi_\ve}$ in $(0,\infty)$, the functional $\mathcal{H}$ defined via~\eqref{eq:entrAbove} satisfies formula~\eqref{eq:regEdissG}. 
The inequality~\eqref{eq:approxAbove} follows from the bound $\Phi_\ve\ge \Phi^{(h_\gamma)}$ together with the lower semicontinuity of the extended functional $\mathcal{\widetilde H}^{(h_\gamma,R)}$ with respect to weak-$\ast$ convergence in measure~\cite{demengel_convex_1984}. We note that this inequality is sufficient to infer the long-time asymptotic behaviour in Section~\ref{ssec:convBdd}.

\paragraph{\normalfont\underline{Approximation from below:}}

the function $\Phi^{(\varphi_\ve)}$ has been defined in Section~\ref{ssec:applications}. Observe that, since $\varphi_\ve\le h_\gamma$ on $(0,\infty)$, we have 
\begin{align}\label{eq:PhiBelow}
  \Phi^{(\varphi_\ve,R)}\le \Phi^{(h_\gamma,R)}\le 0 \qquad\text{ on }[0,\infty).
\end{align}
To see the inequality~\eqref{eq:approxBelow}, we fix $\ve_1>0$ small and estimate, using the non-positivity of $\Phi^{(\varphi_\ve, R)}$ (and $\Phi^{(\varphi_\ve, R)}(0)=0$), mass conservation, and inequality~\eqref{eq:PhiBelow},
\begin{align*}
  \mathcal{H}^{(\varphi_\ve,R)}(f_\ve(t,\cdot)) &\le \mathcal{H}^{(\varphi_\ve,R)}(\chi_{\{|r|\ge\ve_1\}}f_\ve(t,\cdot)) +\frac{\ve_1^2}{2}m
  \\&\le \mathcal{H}^{(h_\gamma,R)}(\chi_{\{|r|\ge\ve_1\}}f_\ve(t,\cdot)) +\frac{\ve_1^2}{2}m.
\end{align*}
Hence, by the locally in $\{r\neq0\}$ uniform convergence of $f_\ve$ to $f$, we infer
\begin{align*}
  \limsup_{\ve\to0}\mathcal{H}^{(\varphi_\ve,R)}(f_\ve(t,\cdot)) \le \mathcal{H}^{(h_\gamma,R)}(\chi_{\{|r|\ge\ve_1\}}f(t,\cdot)) + \frac{\ve_1^2}{2}m \overset{\ve_1\to0}{\to}\mathcal{H}^{(h_\gamma,R)}(f(t,\cdot)),
\end{align*}
where the $\ve_1$-limit follows from dominated convergence.
\end{proof}

\subsection{Finite-time condensation and asymptotic behaviour}\label{ssec:convBdd}

Thanks to Proposition~\ref{prop:ediBallFull}, we can now show the convergence in entropy to the minimiser $\mu_\infty^{(R,m)}$ of $\mathcal{\widetilde H}^{(h_\gamma,R)}$ among measures of mass~$m$. We refer to Notations~\ref{not:Rmtheta} for the definition of $\theta^{(R,m)}, u_\infty^{(R,m)}$ and remind the reader of our notation $m_c(R)=\int_{-R}^Rf_c$, where $f_c=f_{\infty,0}$.
\begin{theorem}[Relaxation to the entropy minimiser of the given mass]\label{thm:FTCond} 
  Let $\gamma\ge2$, $m,R>0$ and assume the hypotheses and use the notations of Proposition~\ref{prop:ediBallFull}. 
Then, in the long-time limit $t\to\infty$, convergence to the minimiser of the entropy holds true in the following sense:
\begin{enumerate}[label=\normalfont{(C\arabic{*})}]
  \item\label{it:con_H} Convergence in entropy:
  \begin{align}\label{eq:convInEntropy}
    \lim_{t\to\infty}\mathcal{\widetilde H}^{(h_\gamma,R)}(\mu(t)) =\mathcal{\widetilde H}^{(h_\gamma,R)}\left(\mu_\infty^{(R,m)}\right),
\end{align}
where $\mu_\infty^{(R,m)}$ is given by eq.~\eqref{eq:muinfRm}, i.e.
\begin{align*}
  \mu_\infty^{(R,m)} = \begin{cases}
     f_{\infty,\theta}\cdot\mathcal{L}^1 & \text{if }m\le m_c(R), 
    \text{ where }\theta = \theta^{(R,m)},
   \\ f_c\cdot\mathcal{L}^1 + (m-m_c(R))\delta_0 & \text{if }m> m_c(R).
                   \end{cases}
\end{align*}

 \item\label{it:con_u} Uniform convergence at the level of $u$:
 \begin{align}\label{eq:ulim}
   \lim_{t\to\infty}\|u(t,\cdot)-u_{\infty}^{(R,m)}\|_{C([0,m])}=0.
  \end{align}
  \item\label{it:con_Dirac} Convergence of the Dirac mass at the origin:
  \begin{align*}
    \lim_{t\to\infty} x_p(t) = (m-m_c(R))_+,
    \quad\text{ where }  (m-m_c(R))_+ =  \max\{0,m-m_c(R)\}.
  \end{align*}
\end{enumerate}
\end{theorem}

\begin{proof}[Proof of Theorem~\ref{thm:FTCond}] 
We first show assertion~\ref{it:con_H}.
  Define
\begin{align*}
  D_R(t)=\int_{-R}^R\frac{1}{h_\gamma(f(t,r))}|\partial_rf(t,r)+rh_\gamma(f(t,r))|^2\,\d r
\end{align*}
and note that identity~\eqref{eq:entrIdv2} and Theorem~\ref{thm:entrMin}, together with Remark~\ref{rem:emin}, imply $D_R\in L^1(0,\infty)$. Hence, there exists a sequence $t_k\to\infty$ such that $D_R(t_k)\to0$. 
By estimate~\eqref{eq:C1bU}, 
there exists $u_\infty\in C^{1,\frac{1}{\gamma-1}}([0,m])$ such that, after transition to a subsequence, 
\begin{align*}
  u(t_k,\cdot)\to u_\infty \text{ in }C^{1,\beta}([0,m]) 
\end{align*}
for $\beta\in(0,\frac{1}{\gamma-1})$, and 
\begin{align*}
  f(t_k,\cdot)\to f_\infty \text{ locally uniformly in }A_{0,R}\cup\{-R,R\},
\end{align*}
where $A_{0,R}=(-R,R)\setminus\{0\}$ and where $f_\infty$ is defined via $f_\infty(u_\infty)=\frac{1}{u'_\infty}$.

We now adapt an argument appearing in Step~2 of the proof of~\cite[Theorem~4.3]{canizo_fokkerplanck_2016}.
Letting $f_k:=f(t_k,\cdot)$ and $g_k:=\frac{1}{f_k^{-\gamma}+1}$, we deduce
\begin{align}\label{eq:hlim}
  g_k \to g_\infty:=\frac{1}{f_\infty^{-\gamma}+1}
\end{align}
locally uniformly in $A_{0,R}\cup\{-R,R\}$.
We then estimate, using the Cauchy--Schwarz inequality, 
\begin{align*}
  \left(\int_{-R}^R|\gamma rg_k+\partial_r g_k|\d r\right)^2&=\gamma^2\left(\int_{-R}^R\left|g_k\left[r+\frac{\partial_rf_k}{f_k(1+f_k^\gamma)}\right]\right|\d r\right)^2
 \\&\le\gamma^2\|g_k\|_{L^1}\int_{-R}^R g_k\left|r+\frac{\partial_rf_k}{f_k(1+f_k^\gamma)}\right|^2 \d r 
 \\&\le C D_R(t_k)\to0\;\;\;\text{ as }k\to\infty.
 \end{align*}
Thus, we deduce that
\begin{align*}
 \gamma rg_k+\partial_r g_k\to0\;\;\text{ in $L^1(-R,R)\;\;\;$ as }k\to\infty,
\end{align*}
which, thanks to~\eqref{eq:hlim}, implies
$\gamma rg_\infty+\partial_rg_\infty=0$ in $\mathcal{D}'(A_{0,R})$
and hence $\gamma rg_\infty+\partial_rg_\infty=0$ almost everywhere in $A_{0,R}$.
 This implies that for certain $\theta_{\pm}\ge0$:
\begin{align*}
  f_\infty=f_{\infty,\theta_-}\chi_{\{-R<r<0\}}+f_{\infty,\theta_+}\chi_{\{0<r<R\}}.
\end{align*}
Since the assumption $\theta_+\not=\theta_-$ contradicts the regularity $u_\infty'\in C((0,m))$, we infer $\theta_+=\theta_-$. For the same reason, we conclude  $\theta_+=\theta_-=\theta^{(R,m)}$
and thus
\begin{align*}
  f_\infty=f_{\infty,\theta^{(R,m)}},\;\;\;u_\infty=u_\infty^{(R,m)}.
\end{align*}
By the dominated convergence theorem, we now have 
$$\mathcal{H}^{(h_\gamma,R)}(f(t_k,\cdot))\to \mathcal{H}^{(h_\gamma,R)}(f_{\infty})=\mathcal{\widetilde H}^{(h_\gamma,R)}(\mu_\infty^{(R,m)}),$$
which, combined with the monotonicity of the function $t\mapsto \mathcal{H}^{(h_\gamma,R)}(f(t,\cdot))$,
implies assertion~\ref{it:con_H}.

We next prove~\ref{it:con_u}.
For an arbitrary time sequence $s_n\to\infty$ we want to show that $\lim_{n\to\infty}\|u(s_n,\cdot)-u_\infty^{(R,m)}\|_{C(\bar J)}=0$.
 By the global Lipschitz continuity of $u$ (in time), we can assume without loss of generality that $|s_n-s_{n+1}|\ge\frac{2}{n}$. We now let $I_n=\{|t-s_n|\le\frac{1}{n}\}$.
 Then, since $D_R\in L^1(0,\infty)$, there exist $n_k$ and $t_{k}\in I_{n_k}$ such that $D_R(t_{k})\to0$. Now the proof of~\ref{it:con_H} shows that after passing to a subsequence, 
 \begin{align*}
  u(t_k,\cdot)\to u_\infty^{(R,m)} \text{ uniformly in }\bar J.
 \end{align*}
 Finally notice that for $K:=\|\partial_tu\|_{L^\infty(\Om)}$ we have
\begin{align*}
  |u(s_{n_k},x)-u_\infty^{(R,m)}(x)|\le K\underbrace{|s_{n_k}-t_k|}_{\le \frac{1}{n_k}}+
   |u(t_k,x)-u_\infty^{(R,m)}(x)|.
\end{align*}
Thus the (arbitrary) sequence $(s_n)$ has a subsequence $(s_{n_k})$ such that $u(s_{n_k},\cdot)\to u_\infty^{(R,m)}$ uniformly in $\bar J$. This implies~\eqref{eq:ulim}.

Assertion~\ref{it:con_Dirac} is a consequence of~\ref{it:con_u} and the fact that the bound~\eqref{eq:fupperbd} holds true uniformly in time. 
  \end{proof}
\begin{remark}\label{rem:convC1b}
  In view of estimate~\eqref{eq:C1bU} the convergence~\ref{it:con_u} of $u(t,\cdot)$ to the entropy minimiser holds true in the stronger topology $C^{1,\beta}([0,m])$ for $\beta\in(0,(\gamma-1)^{-1})$.
\end{remark}

\begin{corollary}[(No) Condensate after finite time]\label{cor:cond} Under the hypotheses of 
  Proposition~\ref{prop:ediBallFull}:
  \begin{itemize}[label=\raisebox{0.25ex}{\tiny$\bullet$}]\itemsep.1em
   \item  If $m>m_c(R)$, there exists $T<\infty$ such that $x_p(t)>0$ for all $t>T$.
   \item If  $m<m_c(R)$, there exists $T<\infty$ such that $\min_{[0,m]} \partial_xu(t,\cdot)>0$ for all $t>T$. In particular, the condensate component is compactly supported in $(0,\infty)$, i.e.\;$\supp x_p\subset\subset(0,\infty)$, and the density $f(t,\cdot)$ is smooth for all $t>T$.
  \end{itemize}
\end{corollary}
\begin{proof}[Proof of Corollary~\ref{cor:cond}]
  The assertion concerning the case $m>m_c(R)$ is an immediate consequence of Theorem~\ref{thm:FTCond}~\ref{it:con_Dirac}. Let us now assume that $m<m_c(R)$. By identity~\eqref{eq:profileThesis1} there exists a constant 
  $c(m,R,u_0)>0$ such that 
  \begin{align*}
    \|u(t,\cdot)-u_\infty^{(R,m)}\|_{C([0,m])}\ge c(m,R,u_0)
  \end{align*}
 whenever $\min_{[0,m]}\partial_xu(t,\cdot)=0$.
 The assertion now follows from Theorem~\ref{thm:FTCond}~\ref{it:con_u}.
\end{proof}

\noindent Corollary~\ref{cor:cond} raises the question of whether and under which conditions finite-time blow-up and condensation may occur in the mass-subcritical case $m<m_c(R)$.

 \begin{proposition}\label{prop:tc}
   In addition to the hypotheses of Proposition~\ref{prop:uMmu},
 suppose that $\gamma>2$.  
  There exists a constant $B_\gamma>0$ only depending on $\gamma$ such that if for some $\delta>0$ the inequality
  \begin{align}\label{eq:cond}
    m - B_\gamma\frac{m^\frac{3\gamma}{2}}{\left(\int_{-R}^R|v|^2f_0(v)\,\d v\right)^\frac{\gamma-2}{2}}\le-\delta
\end{align}
holds true, then the function $t\mapsto x_p(t)$ cannot be identically zero.
\end{proposition}\noindent 
   Note that, for any fixed mass $m>0$, inequality~\eqref{eq:cond} is satisfied for initial data sufficiently concentrated near the origin. 
   
 Theorem~\ref{thm:FTCond}, Corollary~\ref{cor:cond} and Proposition~\ref{prop:tc} show that in general the condensate does interact with the regular part of the solution and may partially or fully dissolve.
 
\begin{corollary}[Existence of transient condensates]\label{cor:tc}
  In addition to the hypotheses of Proposition~\ref{prop:ediBallFull}, suppose that 
   inequality~\eqref{eq:cond} is satisfied for some $\delta>0$. Then, if $m<m_c$,  the point mass at velocity origin satisfies 
\begin{align*}
 \mathrm{supp}\,x_p\subset\subset(0,\infty) \quad\text{ and } \quad x_p\not\equiv0.
\end{align*}
\end{corollary}
\noindent   The proof of Proposition~\ref{prop:tc} is an adaptation of the finite-time blow-up argument in~\cite{toscani_finite_2012} combined with the bounds~\eqref{eq:fupperbd},~\eqref{eq:odef}.
It makes use of the following inequality~\cite{toscani_finite_2012}:
\begin{proposition}[Ref.~\cite{toscani_finite_2012},~Lemma~2]\label{prop:toscani}
 Let $d=1$. For any $\gamma>2$ there exists a constant $B_\gamma\in(0,\infty)$ such that (for all sufficiently regular functions $f\not\equiv0$) 
\begin{align}\label{eq:tosc}
  \int|v|^2 f^{\gamma+1}(v)\,\d v\ge B_\gamma\frac{\left(\int f(v)\,\d v\right)^{\frac{3\gamma}{2}}}{\left(\int|v|^2f(v)\,\d v\right)^{\frac{\gamma}{2}-1}}.
\end{align} 
\end{proposition}\noindent
\begin{proof}[Proof of Proposition~\ref{prop:tc}] 
  Heuristically, the idea is to keep track of, or estimate from below, the flux of mass into the origin. For this purpose we use a virial type argument and consider the evolution of the kinetic energy
  \begin{align*}
    E(t):=\frac{1}{2}\int_{(-R,R)} |v|^2f(t,v)\,\d v=\frac{1}{2}\int_{(0,m)}|u(t,x)|^2\,\dd x.
\end{align*}
The following computations, performed at the level $u$, can be justified in a similar way as in the proof of Lemma~\ref{l:L2bound} (see Appendix~\ref{ssec:highermoments}). We have
\begin{align*}
  \frac{\dd}{\dd t} E(t) &= -\int_{\{|u|>0\}}|u|^2u_x^{-\gamma}\,\dd x
  -\int_{\{|u|>0\}}u\frac{\dd}{\dd x}\left(u_x^{-1}\right)\,\dd x
  - 2E(t)
  \\&\le m -\int_{\{|u|>0\}}|u|^2u_x^{-\gamma}\,\dd x.
\end{align*}
Observe that the last integral equals the left-hand side of ineq.~\eqref{eq:tosc}.
Hence, Proposition~\ref{prop:toscani} yields
\begin{align*}
  \frac{\dd}{\dd t} E(t)&\le m - B_\gamma \frac{(m-x_p(t))^\frac{3\gamma}{2}}{(2E(t))^\frac{\gamma-2}{2}}.
\end{align*}
Thus, if $x_p(t)\equiv0$, we find that whenever the bound~\eqref{eq:cond} holds true for some $\delta>0$,  
 $E(t)$ would have to become negative after some time $T\le \frac{E(0)}{\delta}$,
which is impossible. 
\end{proof}

On the other hand, there is a large class of globally bounded mass-subcritical solutions. We confine ourselves to providing a rather simple criterion.
  Since blow-up cannot occur in the case $\gamma=2$ (see Prop.~\ref{prop:profile}~\ref{it:gamma2bdd}), it suffices to consider the case $\gamma>2$.
  
\begin{proposition}[A criterion for global regularity]
  \label{prop:globalReg}
  Assume that $R>0$, $\gamma>2$ and let $f_0\in C^1([-R,R])$. Suppose that there exists $\theta>0$ such that the function $\tilde f_0(r)=\max_{\sigma\in\{\pm1\}}f_0(\sigma r)$ satisfies
\begin{align}\label{eq:f0globReg}
  \big|\int_0^r \tilde f_{0}(\rho)\,\dd \rho\big|\le \big|\int_0^rf_{\infty,\theta}(\rho)\,\dd \rho\big|
    \quad\text{ for }r\in[-R,R].
\end{align}
Let $m=\|f_0\|_{L^1}$ and denote by $u_0:[0,m]\to[-R,R]$ the inverse of the cumulative distribution function of $f_0$.
  Then the corresponding viscosity solution $u$ of~\eqref{eq:bosonicFPmassCHP4} satisfies $\min_{[0,m]}\partial_xu(t,\cdot)>0$ for all $t\ge0$  and we have
  \begin{align}\label{eq:MtMthetaBd}
   \big|\int_0^r f(t,\rho)\,\dd \rho\big|\le \big|\int_0^rf_{\infty,\theta}(\rho)\,\dd \rho\big|\quad\text{ for }r\in[-R,R],
  \end{align}
where $f(t,\cdot)$ denotes the density associated with the inverse of $u(t,\cdot)$.
\end{proposition}

\begin{remark}
Notice that condition~\eqref{eq:f0globReg} implies that $\int_{-R}^Rf_0< m_c(R)$. 
Conversely, for any $m\in(0,m_c(R))$ and $f_0\in (C^1\cap L^1)(\mathbb{R})$ even and of mass $m$ there exists $\lambda^*=\lambda^*(f_0)\in(0,\infty)$ such that $f_{0,\lambda}(\rho):=\lambda^{-1}f_0(\lambda^{-1}\rho)$ satisfies condition~\eqref{eq:f0globReg} for $r\in\mathbb{R}$ whenever $\lambda\ge\lambda^*$. 
\end{remark}
The proof of Prop.~\ref{prop:globalReg} is based on comparison arguments in the spirit of those used before. It will therefore be omitted.

\subsection{Higher-order comparison}\label{sec:cpho}

In this section, we aim to upgrade the comparison results at the level of $u$ in Section~\ref{ssec:visccp}. In fact, we will see that the intersection comparison result at the level of $u$ easily yields comparison between densities, i.e.\;comparison at the level of $f$.
The result may be of general interest, but will also be used explicitly in the next section.

\begin{definition}[Translations in $x$]\label{def:translates}
  Assume that $n>0$ and let $v$ be a function defined on 
  $[0,n]$.
  For $y\in \mathbb{R}$ let
  \begin{align*}
    ^{(y)}v: [y,n+y]\to [-R,R], \qquad  ^{(y)}v(x)=v(x-y).
  \end{align*}
  If $v=v(t,x)$ is time-dependent, $^{(y)}v$ is defined by ${^{(y)}v}(t,x)=v(t,x-y)$ for all $t$.
   Finally, given $\lambda>0$ let 
  \begin{align}
    \mathcal{T}_\lambda[v]=\{^{(y)}v: y\in(0,\lambda)\}.
  \end{align}
\end{definition}

\begin{proposition}[Comparison for densities]\label{prop:cpForDensitiesI}
  Let $\gamma\ge2$ and $R\in(0,\infty)$. Let $g_0,\tilde g_0\in C^1([-R,R])$, $g_0\not\equiv \tilde g_0$,  be positive functions satisfying 
 \begin{align}\label{eq:orderinitden}
   g_0\le \tilde g_0\qquad \text{ in }[-R,R].
 \end{align}
 Abbreviate $n=\|g_0\|_{L^1}$, $\tilde n=\|\tilde g_0\|_{L^1}$ and let $v_0:[0,n]\to[-R,R]$ (resp.\;$\tilde v_0:[0,\tilde n]\to[-R,R]$) be the inverse cdf of $g_0$ (resp.\;$\tilde g_0$). Denote by $v$ (resp.\;$\tilde v$) the global viscosity solution of problem~\eqref{eq:bosonicFPmassCHP4} (with mass $n$ resp.\;$\tilde n$ and initial datum $v_0$ resp.\;$\tilde v_0$),
 and let $g$ (resp.\;$\tilde g$) denote the density of the absolutely continuous part of the measure associated with the generalised inverse of $v$ (resp.\;$\tilde v$), as obtained in Proposition~\ref{prop:uMmu}.
 Then 
 \begin{align*}
   g \le \tilde g \text{ in }(0,\infty)\times(-R,R).
 \end{align*}
\end{proposition}

\begin{proof}
  The assumption $g_0\le \tilde g_0, g_0\not\equiv\tilde g_0$ implies that $n<\tilde n$. Moreover, for any $w\in \mathcal{T}_{\tilde n-n}[v]$ the number of sign changes (see Def.~\ref{def:Z}) satisfies
  \begin{align*}
   Z[\tilde v(0,\cdot)-w(0,\cdot)]=1.
  \end{align*}
  (Otherwise the fundamental theorem of calculus would lead to a contradiction with ineq.~\eqref{eq:orderinitden}.)
  Since $\tilde v$ is non-degenerate near the lateral boundary, for any $y\in(0,\tilde n-n)$ and $w:={^{(y)}v}$, we have
\begin{align}\label{eq:cpdlc}
  w(t,y)-\tilde v(t,y)<0, \qquad   w(t,n+y)-\tilde v(t,n+y)>0
\end{align}
for all $t\ge0$. Here, we used the fact that $w(t,y)=-R,w(t,n+y)=R$.
Hence, by Corollary~\ref{cor:intsec}, for all $y\in(0,\tilde n-n)$, $w:={^{(y)}v}$,
  \begin{align}\label{eq:cpdicp}
    Z[[\tilde v(t,\cdot)-w(t,\cdot)]_{|(y,n+y)}]=1 \text{ for all }t\ge0.
  \end{align}
  Let now $(t,r)\in (0,\infty)\times\left((-R,R)\setminus\{0\}\right)$ be arbitrary. The intermediate value theorem implies the existence of $x'\in(0,\tilde n)$ and $x''\in(0,n)$ such that 
  $\tilde v(t,x')=r$, $v(t,x'')=r$. Letting $y'=x'-x''$ and
  $w:={^{(y')}v}$, we infer that
  \begin{align*}
   w(t,x')=\tilde v(t,x')=r,
  \end{align*}
 which, owing to properties~\eqref{eq:cpdlc} and~\eqref{eq:cpdicp}, implies that
  \begin{align}\label{eq:tangentcp}
   \partial_xw(t,x')\ge \partial_x\tilde v(t,x').
  \end{align} 
  Now, the conclusion follows by observing that, in view of eq.~\eqref{eq:relfu},
  \begin{align*}
    g(t,r)=\frac{1}{\partial_xv(t,x'-y')} = \frac{1}{\partial_xw(t,x')}
  \end{align*}
  and 
  \begin{align*}
   \tilde g(t,r) =  \frac{1}{\partial_x\tilde v(t,x')},
  \end{align*}
where we used the convention $\frac{1}{0}=\infty$.

 As a side note let us remark that if  $\partial_xw(t,x')>0$, it is possible using classical arguments for uniformly parabolic equations~(see e.g.~\cite{lieberman_parabolic}) and the fact that $t>0$ to deduce that the inequality in~\eqref{eq:tangentcp} is strict.
\end{proof}

\section{The problem on the whole line~\texorpdfstring{$\mathbb{R}$}{}}\label{sec:wholeLine}

In this section we are concerned with the BFP eq.~\eqref{eq:befpOrig} posed on the real line, i.e.\;with
\begin{alignat}{2}\label{eq:befpLine} 
  \begin{cases}
 \partial_tf=\partial_r^2f+\partial_r(r\,h_\gamma(f)), \qquad &t>0, \;r\in\mathbb{R}, 
 \\ f(0,r)=f_0(r)>0,  & r\in\mathbb{R}, 
  \end{cases}
\end{alignat}
where we suppose again that $\gamma\ge2$. We always assume that the integrable initial density $f_0$ decays sufficiently fast at infinity (to be specified below) and denote by $m$ its mass $\|f_0\|_{L^1(\mathbb{R})}$.

As a motivation, let us first assume that $f=f(t,r)$ is a sufficiently regular, strictly positive classical solution of eq.~\eqref{eq:befpLine} with finite conserved mass $m:=\int f(t,\cdot)$.
Defining for $t>0$ the cumulative distribution function 
\begin{align}
  M(t,r) = \int_{-\infty}^r f(t,r')\,\dd r'
\end{align}
and letting $u(t,\cdot):(0,m)\to\mathbb{R}$ denote the inverse of $M(t,\cdot):\mathbb{R}\to(0,m)$, 
we find that $u$ satisfies the problem
\begin{align}\label{eq:BFPmassLine}
    \begin{cases}
       \mathcal{F}(u)=0 \quad &\text{ in }\Om := (0,\infty)\times(0,m),
       \\  \lim_{x\searrow0} u(t,x)=-\infty,\quad \lim_{x\nearrow m}u(t,x)=\infty \qquad & \text{ for } t>0,
       \\  u(0,x) = u_0(x)\quad & \text{ for } x\in(0,m),
    \end{cases}
\end{align}
where, as before, 
 $\mathcal{F}(u) := F(u,\partial_tu,\partial_xu,\partial_x^2u)$
with
\begin{align}\label{eq:Fagain}
  F(z,\alpha,p,q):=p^\gamma\alpha-p^{\gamma-2}q+z(1+p^\gamma)
\end{align}
   for $p\ge0$ and $z,\alpha,q\in\mathbb{R}$. 
      We are primarily interested in solutions for which the limits in the second line of problem~\eqref{eq:BFPmassLine} hold true locally uniformly in time (in the sense of eq.~\eqref{eq:bclineu}).
      
   With respect to the Cauchy--Dirichlet problem~\eqref{eq:bosonicFPmassCHP4} 
   and the general framework established in Section~\ref{chp:framework}, problem~\eqref{eq:BFPmassLine} has the added difficulty of the function $u$ being unbounded near the lateral boundary.
   This is, however, mainly a technical issue, and existence, uniqueness and regularity
   for problem~\eqref{eq:BFPmassLine}
 in the spirit of Corollary~\ref{cor:bosonicFP} will be established below for a large class of initial data.
 The adaptation of the theory in Section~\ref{chp:1dftc}  
 to eq.~\eqref{eq:BFPmassLine} will then be a fairly straightforward task
and will therefore not be discussed in detail.

\begin{definition}[Admissible initial datum for problem~\eqref{eq:BFPmassLine}]\label{p:hpInit}
  We say that an initial value $u_0\in C^2((0,m))$ is \textit{admissible} for problem~\eqref{eq:BFPmassLine} if it has the 
  following properties:
  \begin{enumerate}[label=\normalfont{(IV\arabic{*})}]
   \item\label{it:iv1} $\inf_{(0,m)}u_0'>0$.
    \item\label{it:iv2} The density $f_0$ associated with the inverse of $u_0$, given by $f(u_0)=\frac{1}{{u_0}'}$, satisfies 
    \begin{align}\label{hp:lb}
      f_0\ge f_{\infty,\theta}    \quad \text{in }\mathbb{R}\qquad \text{for some  $\theta\in(0,\infty)$}.
    \end{align}
    \item\label{it:ivL2} $\|u_0\|_{L^2(0,m)}<\infty$.
      \item\label{it:iv3} There exists $\varepsilon_0>0$ such that the function $r\mapsto |r|^{1+\varepsilon_0}f_0(r)$ lies in $L^\infty(\mathbb{R})$.
  \end{enumerate}
\end{definition}

\begin{remark}
 As we will see below, hypothesis~\ref{it:iv2} is a simple means to ensure $t$-uniform Lipschitz regularity, locally in $x\in(0,m)$, of the solution to be constructed. 
   Besides, notice that hypothesis~\ref{it:iv2} implies the boundary behaviour
   $\lim_{ x\to0^+}u_0( x)=-\infty$, $\lim_{ x\to m^-}u_0( x)=\infty$. It will, in fact, ensure that, for the solution to be constructed, the limits in the second line of eq.~\eqref{eq:BFPmassLine}  hold true uniformly in time.
   Hypothesis~\ref{it:ivL2} is equivalent to requiring that the second moment of the density $f_0$ is finite.
  The assumed boundedness of the function $r\mapsto |r|^{1+\varepsilon_0}f_0(r)$ is a technical hypothesis used to ensure that the
  constant $c(u_0)$ in estimate~\eqref{eq:dulb2} is independent of $R$.
\end{remark} 

\begin{definition}\label{rem:defuR}
 Let $u_0$ be admissible in the sense of Definition~\ref{p:hpInit}. Then for any $R\ge 1$ there exist points $a_R$ and $b_R$ satisfying $u_0(a_R)=-R$ and $u_0(b_R)=R$. Abbreviating $J_R:=(a_R,b_R)$ and $\Om_R:=(0,\infty)\times J_R$, we denote by $u^{(R)}$ the unique viscosity solution of $\mathcal{F}=0$ in $\Om_R$ subject to the conditions $u^{(R)}(0,\cdot)={u_0}_{|J_R}$, $u^{(R)}(t,a_R)=-R$,  $u^{(R)}(t,b_R)=R$. (See Corollary~\ref{cor:bosonicFP}.)
The measure $\mu^{(R)}(t)\in \mathcal{M}_b^+([-R,R])$ associated with the generalised inverse of $u^{(R)}(t,\cdot)$  has the form $\mu^{(R)}(t) = f^{(R)}(t,\cdot)\cdot\mathcal{L}^1 +x^{(R)}_p(t)\delta_0$, where $f^{(R)}, x_p^{(R)}(t)$ are as in Proposition~\ref{prop:uMmu}.
\end{definition}
\smallskip

Under the hypotheses on $u_0$ in Definition~\ref{p:hpInit}, we are able to construct a viscosity solution $u$ of problem~\eqref{eq:BFPmassLine} as the limit of a sequence of solutions $\{u^{(R)}\}$ as in Definition~\ref{rem:defuR}.

 We are now in a position to state the main results of this section. Recall that $u=u(t,x)$ is called $x$-monotonic if $u(t,\cdot)$ is non-decreasing in $x$ for all $t$ (see Definition~\ref{def:xm}).
 
 \begin{theorem}[Wellposedness]\label{thm:line}
Let $\gamma\ge2, m\in(0,\infty)$ and suppose that $u_0\in C^2((0,m))$ 
is admissible for eq.~\eqref{eq:BFPmassLine} in the sense of Definition~\ref{p:hpInit}.
 Then there exists a unique $x$-monotonic
 viscosity solution $u\in C([0,\infty)\times(0,m))$ of problem~\eqref{eq:BFPmassLine}
 with the property that
 \begin{align}\label{eq:ulbc}
  \lim_{x\to0}\sup_{t}u(t,x)=-\infty,\quad \lim_{x\to m}\inf_{t}u(t,x)=\infty.
 \end{align}
The function $u$ satisfies the bound
\begin{align}\label{eq:xlocLip}
  \|u\|_{C^{0,1}([0,\infty)\times J')}\le C_{J'}
\end{align}
for any $J'\subset\subset (0,m)$ and $u\in L^\infty([0,\infty),L^2(0,m))$.
 \end{theorem}
 \begin{definition}\label{def:MmuLine}
   \begin{enumerate}[label=(\roman{*})] 
    \item\label{it:defMv}
   Given a non-decreasing, continuous function $v:(0,m)\to\mathbb{R}$
   satisfying $\lim_{x\to0+}v(x)=-\infty, \lim_{x\to m-}v(x)=\infty$, we define its \textit{generalised inverse} $M_v:\mathbb{R}\to(0,m)$ via 
   \begin{align}\label{eq:MvLine}
     M_v(r) = \sup\big\{x\in(0,m):v(x)\le r\big\},\quad r\in\mathbb{R}.
   \end{align}
   \item\label{it:defmuv} It is elementary to see that $M_v$ in~\eqref{eq:MvLine} is increasing, right-continuous
   and satisfies 
   \begin{align*}
     \lim_{r\to-\infty} M_v(r)=0,\qquad \lim_{r\to\infty} M_v(r)=m.
   \end{align*}
   Hence, $M_v$ is the cumulative distribution function (cdf) of a measure $\mu_v\in \mathcal{M}^+_b(\mathbb{R})$ whose mass equals $m$ (see e.g.~\cite[Chapter~20.3]{royden2010real}). The measure $\mu_v$ is uniquely determined by 
   \begin{align*}
    \mu_v((-\infty,r])=M_v(r),\quad r\in \mathbb{R}.
   \end{align*} 
   \item 
   Given $u$ as in Theorem~\ref{thm:line} and $t\ge0$ we denote by $M(t,\cdot):\mathbb{R}\to(0,m)$
   the generalised inverse of $u(t,\cdot)$, i.e.
    $$ M(t,\cdot):=M_{u(t,\cdot)},$$
    where we used the notation~\eqref{eq:MvLine}. We further let $\mu(t)\in \mathcal{M}_b^+(\mathbb{R})$ denote the measure associated with the cdf $M(t,\cdot)$ as introduced in Definition~\ref{def:MmuLine}~\ref{it:defmuv}, i.e. $\mu(t)=\mu_{u(t,\cdot)}$.
   \end{enumerate}
 \end{definition}

 \begin{theorem}\label{thm:propLine} Under the hypotheses of Theorem~\ref{thm:line} and with the notations in Definition~\ref{def:MmuLine}, 
the viscosity solution $u$ obtained in Theorem~\ref{thm:line} has the following properties:
\begin{enumerate}[label=\normalfont{(L\arabic*)}]
  \item\label{it:l1} For all $t>0$ there exist unique points $x_-(t),x_+(t)\in(0,m)$ such that $$u(t,\cdot)^{-1}(0)=[x_-(t),x_+(t)].$$
   Also, 
   $\partial_x u(t,x)>0$ for $x\in(0,m)\setminus[x_-(t),x_+(t)],$
    and away from $\{\partial_xu=0\}$ the function $u$ is smooth and satisfies $\mathcal{F}(u)=0$ in the classical sense.
   \item  For each $t>0$ the strictly increasing and right-continuous function $M(t,\cdot)$ satisfies
   $$M(t,0-)=x_-(t)\text{ and }M(t,0)=x_+(t).$$
   Moreover, $M$ is $C^\infty$ in the open set $\{(t,r):t>0,|r|\in(0,\infty)\}$.
   \item\label{it:Ldecomp} 
   Let $x_p(t):= \mathcal{L}^1(\{u(t,\cdot)=0\})$, $t>0$.  There exists a unique,  
   positive function $f(t,\cdot)\in L^1(\mathbb{R})$ such that the measure $\mu(t)\in\mathcal{M}^+_b(\mathbb{R})$ associated with $M(t,\cdot)$ has the decomposition
   \begin{align*}
\mu(t)=f(t,\cdot)\mathcal{L}^1+x_p(t)\delta_0  ,\quad t\in(0,\infty),
\end{align*}
where away from $r=0$ the function $f$ is a classical solution of eq.~\eqref{eq:befpLine}.
\item\label{it:spatialBUprofileL} Blow-up behaviour:  
 if the function $f(t,\cdot)$ introduced in~\ref{it:Ldecomp} is unbounded near the origin (or equivalently $\partial_xu(t,x_\pm(t))=0$), then 
  \begin{align*}
  f(t,r)=\left(\frac{\gamma}{q(t,r)}\int_0^rsq(t,s)\,\d s\right)^{-\frac{1}{\gamma}},
\end{align*}
where $q$ is defined as in formula~\eqref{eq:def_q}.
  In particular, the expansion~\eqref{eq:profileThesis1} holds true for small $|r|$.
Hence, if $\gamma=2$, $f$ is globally regular and satisfies eq.~\eqref{eq:befpLine} in the classical sense.
\end{enumerate}
\end{theorem}

On the whole space, an entropy dissipation identity analogous to Proposition~\ref{prop:ediBallFull} requires some extra control on the tails of the density. This issue has been well studied, for instance in~\cite{carrillo_fermi_2009},
and here we omit the precise statements regarding the long-time asymptotics in the problem on the line. 

The rest of this section is devoted to the proof of Theorems~\ref{thm:line}.
The assertions in Theorem~\ref{thm:propLine} can then be deduced analogously to the case of a bounded interval (see Section~\ref{chp:1dftc}). We start by deriving uniqueness.

\subsection{Uniqueness for unbounded monotonic viscosity solutions}

In order to establish uniqueness for problem~\eqref{eq:BFPmassLine},~\eqref{eq:ulbc}, we first observe that the proof of the comparison principle, Proposition~\ref{prop:visccomp}, shows that the assumed boundary regularity of the functions involved can be relaxed as follows:
\begin{corollary}[Comparison, relaxed version]\label{cor:visccompR}
  Let $0<T\le\infty$ and assume that the continuous function $G$ satisfies~\nref{hp:q} \&~\nref{hp:strict}.
  Suppose that $u\in \mathrm{USC}([0,T)\times(0,m))$ is a subsolution, $v\in \mathrm{LSC}([0,T)\times(0,m))$ a supersolution of $\mathcal{G}=0$ in $\Om=(0,T)\times(0,m)$ with the boundary behaviour
  \begin{align}\label{eq:bcinvL}
    \limsup_{\om\to\partial_p\Om}\left(u(\om)-v(\om)\right)\le 0.
  \end{align} 
  Then $u\le v$ in $\Om$.
\end{corollary}

Corollary~\ref{cor:visccompR} 
implies uniqueness for BFP on the line (at the level of $u$) in the following sense:
\begin{corollary}[Uniqueness for problem~\eqref{eq:BFPmassLine}]\label{cor:uniqueBFPline}
 Let $T\in(0,\infty)$. Given a non-decreasing function $u_0\in C((0,m))$,
     there exists at most one $x$-monotonic viscosity solution $u\in C([0,T)\times(0,m))$ of problem~\eqref{eq:BFPmassLine} 
    with the property that
    \begin{align}\label{eq:bclineu}
     \lim_{x\to0}\sup_{t\in(0,T)}u(t,x)=-\infty, \qquad \lim_{x\to m}\inf_{t\in(0,T)}u(t,x)=\infty.
    \end{align}
\end{corollary}
\begin{proof}
 Suppose that $u$ and $v$ are $x$-monotonic viscosity solutions of problem~\eqref{eq:BFPmassLine} with the properties assumed in the statement of Cor.~\ref{cor:uniqueBFPline}.
 For functions $w=w(t,x)$ and $0<\delta\ll1$ we denote by ${^{(\mp\delta)}w}(t,x)$ the spatially shifted function $w(t,x\pm\delta)$. The same notation will be used for time-independent functions (see Definition~\ref{def:translates}). We further abbreviate $_\delta\Om:=(0,T)\times(\delta,m-\delta)$. 
  Then ${^{(\delta)}u}$ (resp.\;${^{(-\delta)}v}$) is a viscosity subsolution (resp.\;supersolution) of $\mathcal{G}=0$ in $_\delta\Om$. Conditions~\eqref{eq:bclineu} 
  and the $x$-monotonicity 
  ensure that 
  \begin{align*}
    \limsup_{\om\to\partial_p\left({_\delta\Om}\right)}\left({^{(\delta)}u}(\om)-{^{(-\delta)}v}(\om)\right)\le 0.
  \end{align*}
  Hence, by Corollary~\ref{cor:visccompR}, ${^{(\delta)}u}\le{^{(-\delta)}v}$ in ${_\delta\Om}$. As $\delta>0$ can be chosen arbitrarily small,
  this implies, thanks to the continuity of $u$ and $v$, that $u\le v$ in $\Om$. Since $u$ and $v$ are interchangeable, we infer that $u=v$.
\end{proof}

\subsection[Existence and regularity]{Proof of \texorpdfstring{Theorem~\ref{thm:line}:}{} 
    Existence and Regularity}\label{ssec:exLine}

  The uniqueness part of Theorem~\ref{thm:line} has been established in Corollary~\ref{cor:uniqueBFPline}. 
Now, our main task lies in establishing the existence part of Theorem~\ref{thm:line} and the bound~\eqref{eq:xlocLip}.
  The key is a local Lipschitz bound in space-time for $u^{(R)}$ which holds true uniformly in $R\gg1$.
  \begin{proposition}\label{prop:uRLip}
    Let $u^{(R)}$ and $\Om_R$ be as in Definition~\ref{rem:defuR}. Then, for any 
    $R\ge1$
\begin{align}\label{eq:LipRL}
  K_R:= \sup_{\tilde R\ge R}\|u^{(\tilde R)}\|_{C^{0,1}(\Om_R)}<\infty.
\end{align}
  \end{proposition}
\noindent  Estimate~\eqref{eq:LipRL} yields local compactness of our family  $\{u^{(R)}\}$  of approximate solutions.

 Proposition~\ref{prop:uRLip} will be proved in three steps:

 In \textit{Step~1} we establish an upper bound on the spatial Lipschitz constants of the approximate sequence $\{u^{(R)}\}$ taking the form
\begin{align}\label{eq:upperbduRx}
  \|\partial_xu^{(\tilde R)}\|_{L^\infty(\Om_R)}\le C(\theta,R),\quad \tilde R\ge R\ge1,
\end{align}
where $\theta$ is the parameter in ineq.~\eqref{hp:lb}. This step relies on  hypothesis~\ref{it:iv2} and the following 
\begin{lemma}\label{l:1}For any $R\ge1$ there exists $c_R<\infty$ such that for all $\tilde R\ge R$ 
 \begin{align*}
   \sup_{t>0}\|u^{(\tilde R)}(t,\cdot)\|_{L^\infty(J_R)}\le c_R,
 \end{align*}
 where $J_R=(a_R,b_R)$ are as in Definition~\ref{rem:defuR}.
\end{lemma}
Lemma~\ref{l:1} is an immediate consequence of
\begin{lemma}\label{l:L2bound} For all $R\ge1$
 \begin{align*}
   \sup_{t>0}\|u^{(R)}(t,\cdot)\|_{L^2(J_R)}^2\le \max\{m,\|u_0\|_{L^2}^2\}.
 \end{align*}
\end{lemma}
Lemma~\ref{l:L2bound} is proved in Appendix~\ref{ssec:highermoments}.
We note that the uniform bound in Lemma~\ref{l:L2bound} can easily be generalised to $L^p$ spaces for $p\ge2$.
Observe that the $L^p$ norm at the level of $u$ equals the $p^\mathrm{th}$ moment of the density $f$. 
In the original variables, the propagation of higher-order moments for several other (nonlinear) Fokker--Planck-type equations on $\mathbb{R}^d, d\in\mathbb{N}$, is rather well-established. 
See~\cite{carrillo_fermi_2009} for a proof in the case of the Kaniadakis--Quarati model for fermions.

In \textit{Step~2} of the proof of Proposition~\ref{prop:uRLip} we derive a lower bound on $\partial_xu^{(R)}$: $\exists\, c(u_0)>0$ such that
\begin{align}\label{eq:dulb2}
  \partial_xu^{(R)}\ge c(u_0)|u^{(R)}|,
\end{align}
The constant $c(u_0)$ only depends on the mass of a symmetric, radially decreasing smooth
function $\tilde f_0$ lying above $f_0$ (see~\eqref{eq:deConst}).

Steps~1 \&~2 both use the comparison principle for densities, Proposition~\ref{prop:cpForDensitiesI}, applied to the functions $f^{(R)}$
introduced in Definition~\ref{rem:defuR} and a suitable reference solution.

In \textit{Step~3} we show that, thanks to parabolic estimates, Steps~1 \&~2 imply a uniform control of $|\partial_tu^{(R)}|$ on sets of the form $\{\delta<|u^{(R)}|<\delta^{-1}\},\delta>0$.
Reasoning as in the proof of Proposition~\ref{prop:lipTime}, we will then infer that an $R$-uniform control of the quantity $|\partial_tu^{(R)}|$ is even true on sets of the form  $\{|u^{(R)}|<\delta^{-1}\},\delta>0$. 

\begin{proof}[Proof of Proposition~\ref{prop:uRLip}] We proceed by showing the three steps outlined above. Throughout the proof we assume that $\tilde R\ge R\ge1$.\\
  \underline{Step~1:}
  Since $f^{(\tilde R)}(0,\cdot)=f_0\ge f_{\infty,\theta}$ on $[-\tilde R,\tilde R]$, Proposition~\ref{prop:cpForDensitiesI}
  yields  
  $$f^{(\tilde R)}(t,\cdot)\ge f_{\infty,\theta}\quad\text{ on }[-\tilde R,\tilde R]\;\text{ for any }t\ge0.$$
  Owing to relation~\eqref{eq:relfu} 
  and Lemma~\ref{l:1} we infer that for any $\tilde R\ge R$   
\begin{align}\label{eq:estStp1}
  \|\partial_xu^{(\tilde R)}\|_{L^\infty(\Om_R)}
  \le (f_{\infty,\theta}(c_R))^{-1}.
\end{align}
Here we used the monotonicity of $f_{\infty,\theta}(r)$ in $|r|$. The constant $c_R<\infty$ in estimate~\eqref{eq:estStp1} equals the one in Lemma~\ref{l:1}.
This proves estimate~\eqref{eq:upperbduRx} and completes Step~1.

\medskip

\noindent \underline{Step~2:}
Let $\hat f_0(r)=\max_{\sigma\in\{\pm 1\}}f_0(\sigma r)$.
Then, by~\ref{it:iv3}, 
there exists $C<\infty$ such that
\begin{align}\label{eq:deConst}
  f_0(r)\le\hat f_0(r)\le C (1+|r|^2)^{-\frac{(1+\varepsilon_0)}{2}}=:\tilde f_0(r),\;\;r\in\mathbb{R}.
\end{align}
Notice that $\tilde f_0$ is even, non-increasing in $|r|$, and, moreover, 
$\tilde f_0\in L^1(\mathbb{R})\cap C^\infty(\mathbb{R})$. 
For $R\ge1$ consider the solutions  $\tilde u^{(R)}$ and $u^{(R)}$ emanating from the 
inverse cdf of $\tilde f_{0|[-R,R]}$ and $f_{0|[-R,R]}$
and denote the corresponding densities, defined on $(0,\infty)\times(-R,R)$, by $\tilde f^{(R)}$ and $f^{(R)}$.
Then, by Proposition~\ref{prop:cpForDensitiesI}, for all $t\ge0$
\begin{align*}
  f^{(R)}(t,r)\le\tilde f^{(R)}(t,r),\;\;r\in[-R,R].
\end{align*}
\noindent
By uniqueness and the equation's symmetry, $\tilde u^{(R)}(t.\cdot)$ is symmetric for any $t\ge0$.
Moreover,  letting $\tilde m_R=\|\tilde f_0\|_{L^1(-R,R)}$, the function
$\tilde u^{(R)}(t.\cdot)_{|(\frac{\tilde m_R}{2},\tilde m_R)}$ is convex as a consequence of a classical minimum argument combined with
inequality~\eqref{eq:uxxPos}, which controls the delicate region near the origin. (Strictly speaking, this argument requires an additional regularity hypothesis on the initial datum near the lateral boundary, which, after construction of the solution $u$, can easily be removed by an approximation argument.)
Hence, $\tilde f^{(R)}(t,\cdot)$ is non-increasing in $|r|$, implying that 
 $\tilde f^{(R)}(t,r)\le \frac{\tilde m}{2|r|}$ for $t\ge0$, $r\in(-R,R)\setminus\{0\}$, where $\tilde m:=\|\tilde f_0\|_{L^1(\mathbb{R})}$.
This yields 
\begin{align}\label{eq:dulb}
  \partial_xu^{(R)}\ge \frac{2|u^{(R)}|}{\tilde m},
\end{align}
which concludes Step~2. 

\medskip

\noindent \underline{Step~3:}
Thanks to hypothesis~\eqref{hp:lb} there exist time-independent $x$-monotonic functions 
$$u_+(t,\cdot)\equiv u_+:(0,m)\to(\infty,\infty], \quad u_-(t,\cdot)\equiv u_-:(0,m)\to[-\infty,\infty)$$
with the following properties:
\begin{enumerate}
  \item $u_+\in C(\Om\cap\{u_+<\infty\})$ is a supersolution, $u_-\in C(\Om\cap\{u_->-\infty\})$ a subsolution of $\mathcal{F}=0$ 
  in $\Om\cap\{u_+<\infty\}$ resp.\;in $\Om\cap\{u_->-\infty\}$
  \item $u_-(x)\le u_0(x)\le u_+(x)$ for all $x\in(0,m)$
  \item $\lim_{x\to0}u_+(x)=-\infty$, $\lim_{x\to m}u_-(x)=\infty$.
\end{enumerate}
Thus, by comparison, for any $\tilde R\in[1,\infty)$
\begin{align}\label{eq:-u+}
  u_-(x)\le u^{(\tilde R)}(t,x)\le u_+(x)\quad \text{ for all }{x\in J_{\tilde R},}\;t\ge0.
\end{align} 
Hence, owing to bound~\eqref{eq:dulb}, we infer the existence of $\underline{R}\in[1,\infty)$  and $c_1=c_1(u_0)>0$ such that for any $\tilde R\ge\underline{R}$ the inequality $\partial_xu^{(\tilde R)}(t,\cdot)\ge c_1>0$ holds true in $(a_{\tilde R},a_{\underline{R}})\cup(b_{\underline{R}},b_{\tilde R})$. 
Now, for $R\ge \underline{R}$ 
we can apply classical parabolic estimates (see~\cite[Theorem~V.5.1]{ladyzhenskaya_1968_linear})
 to the equation for $u^{(\tilde R)}$, $\tilde R\ge R+1$, in $(0,\infty)\times I_{\eta,R}$,
where for $0<\eta\ll1$ we denote $I_{\eta,R}:=(a_R,a_R+\eta)\cup(b_R-\eta,b_R)$ and, for small $\varepsilon>0$, $I_{\eta,R,\varepsilon}:=\{x\in (0,m):\mathrm{dist}(x,I_{\eta,R})<\varepsilon\}$. 
In particular, one has the bound
\begin{align*}
  \|\partial_tu^{(\tilde R)}\|_{L^\infty((0,\infty)\times I_{\eta,R})}\le C\left(\varepsilon,R, 
  \|u^{(\tilde R)}\|_{L^\infty((0,\infty),C^1(\bar I_{\eta,R,\varepsilon}))},
  \|u_0\|_{C^{2}(\bar I_{\eta,R,\varepsilon})}, c_1, \theta\right)
\end{align*}
for any $\tilde R>R+1$.
Arguing as in Proposition~\ref{prop:lipTime} we deduce, also owing to Lemma~\ref{l:1}, 
\begin{align}\label{eq:estutStp3}
\|\partial_tu^{(\tilde R)}\|_{L^\infty(\Om_R)}\le C(R,u_0).
\end{align}
Combining estimates~\eqref{eq:estStp1} and~\eqref{eq:estutStp3} 
we obtain the bound~\eqref{eq:LipRL}.
\end{proof}

We are now in a position to prove Theorem~\ref{thm:line}.
\begin{proof}[Proof of Theorem~\ref{thm:line}] 
  We argue similarly to Section~\ref{ssec:approx}.
  The bound~\eqref{eq:LipRL} and the equation satisfied by $u^{(\tilde R)}$ yield
\begin{align*}
 \sup_{\tilde R>R}\|\partial_x((\partial_xu^{(\tilde R)})^{\gamma-1})\|_{L^\infty(\Om_R)}\le C(R).
\end{align*}
Thus, we find
$\beta_0>0$,  
$u\in C([0,\infty);C^{1,\beta_0}_\mathrm{loc}((0,m)))\cap C_\mathrm{loc}^{0,1}([0,\infty)\times(0,m))$ and a sequence $\tilde R\to\infty$ such that
for any $T>0$ and any $R>0$:
\begin{align*}
  u^{(\tilde R)}\overset{\tilde R\to\infty}{\longrightarrow} u \;\;\;\text{ in }C([0,T];C^{1,\beta_0}(\bar J_{R})).
\end{align*}
By Remark~\ref{rem:stab}~\ref{it:stabC} the limit $u$ is itself a viscosity solution of eq.~\eqref{eq:genInv}, and, by construction, $u(0,\cdot)=u_0$. 
Owing to inequalities~\eqref{eq:-u+}, we have 
\begin{align*}
 \lim_{ x\to0+}\sup_t u(t,x)\le\lim_{ x\to0+}u_+(x)= -\infty, \qquad \lim_{x\to m-}\inf_tu(t,x)\ge \lim_{x\to m-}u_-(x)=\infty.
\end{align*} 
Estimate~\eqref{eq:xlocLip} is an immediate consequence of~\eqref{eq:LipRL}  and the locally uniform convergence of the subsequence $\{u^{(\tilde R)}\}$.
\end{proof}

\begin{subappendices}
  \section{Appendix}

\subsection{Semi-convexity}

\begin{definition}[Semi-convexity and -concavity]\label{def:semiconvex}
  Let $U\subset\mathbb{R}^d$ be convex. A function $v:U\to\mathbb{R}$ is called 
  \textit{semi-convex} (resp.\;\textit{semi-concave})  if there exists a constant $C\in\mathbb{R}$ such that the function 
  $x\mapsto v(x)+\frac{C}{2}|x|^2$ is convex (resp.\;such that $v(x)-\frac{C}{2}|x|^2$ is concave). 
\end{definition}

\begin{proposition}\label{prop:semi-conv}
Let $u:\Om\to\mathbb{R}$ be continuous.  Suppose that there exists a constant $C<\infty$ such that for all $\om\in\Om$ and all $(\tau,p,q)\in{\mathcal{P}}^+ u(\om)$ (resp.\;all $(\tau,p,q)\in{\mathcal{P}}^- u(\om)$) the bound $q\ge-C$ (resp.\;$q\le C$) holds true. Then, for all $t>0$ the function $u(t,\cdot)$ is semi-convex (semi-concave) in $J$ with  constant bounded above by~$C$.
\end{proposition}
\begin{proof} By symmetry, it suffices to prove the statement asserting semi-convexity.
 Thanks to~\cite[Lemma~1]{alvarez_convex_1997}, it is enough to show that 
for all $t\in(0,\infty)$ and all $x\in J$
\begin{align}\label{eq:q-lbd-elliptic}
(p,q)\in \mathcal{J}^{2,+}(u(t,\cdot))(x)\;\;\Rightarrow\;\;  q\ge -C.
\end{align}
 The implication~\eqref{eq:q-lbd-elliptic} is a consequence of 
 the following general argument. A similar reasoning can be found in\;\cite{silvestre_fully_2013}. 
 
 In order to see implication~\eqref{eq:q-lbd-elliptic}, we fix $t\in(0,\infty)$ and $x\in J$ and assume that $(p,q)\in \mathcal{J}^{2,+}(u(t,\cdot))(x)$. By definition (and the local boundedness of $u$), there exists $\phi\in C^2(J)$ such that $0\ge u(t,y)-\phi(y)$, 
$0= u(t,x)-\phi(x)$ and $p=\phi'(x),\; q=\phi''(x)$. In particular, $u(t,\cdot)-\phi$ reaches a maximum at $x$. After possibly replacing $\phi$ with $\phi(y)+|x-y|^4$, we may assume that the maximum is strict.
Now consider for suitably small $0<\delta\ll1$ the function 
\begin{align*}
  w(s,y):=u(s,y)-\left(\phi(y)+\frac{1}{2\varepsilon}|s-t|^2\right) \quad\text{ in }Q_\delta:=[t-\delta,t+\delta]\times[x-\delta,x+\delta].
\end{align*}
By continuity, $w$ reaches its (non-negative) maximum at some point $(s_\varepsilon,y_\varepsilon)\in Q_\delta$ and as $\varepsilon\to0$ we must have $s_\varepsilon\to t$. Moreover, $y_\varepsilon\to x$ since if this was not the case, then along a subsequence $(s_\varepsilon,y_\varepsilon)\to (t,\tilde x)$ for some $\tilde x\not= x$
and therefore $0\le w(s_\varepsilon,y_\varepsilon)\le u(s_\varepsilon,y_\varepsilon)-\phi(y_\varepsilon)\to u(t,\tilde x)-\phi(\tilde x)<0$ by the strictness of the maximum, a contradiction.

Hence for small enough $\varepsilon>0$ we have $(0,0,0)\in\mathcal{P}^+w(s_\varepsilon,y_\varepsilon)$
or, equivalently,
\begin{align*}
  \left(\frac{s_\varepsilon-t}{\varepsilon},\phi'(y_\varepsilon),\phi''(y_\varepsilon)\right)\in\mathcal{P}^+u(s_\varepsilon,y_\varepsilon).
\end{align*}
Thus $\phi''(y_\varepsilon)\ge -C$ and,
letting $\varepsilon\to0$, we conclude $q=\phi''(x)\ge -C.$
\end{proof}

\begin{lemma}\label{l:c11}
  Suppose the function $v:J\to\mathbb{R}$ is semi-convex and semi-concave with constant $C<\infty$. Then $v\in C^{1,1}(\bar J)$ and $[v']_{C^{0,1}(\bar J)}\le C$.
\end{lemma}
The simple proof of L.~\ref{l:c11} is omitted.

\subsection{\texorpdfstring{$\mathcal{L}^2$}{}-measurability}\label{app:measurable}

\begin{lemma}
Using the notation in Sec.~\ref{ssec:approx}, the $2^\mathrm{nd}$ order pointwise derivative $^{(p)}\partial_x^2v_\sigma$ of $v_\sigma$ with respect to $x$ exists $\mathcal{L}^2$-a.e.\;in $\Om$ and  $\partial_xv_\sigma$ has a weak derivative in $x$-direction satisfying 
\begin{align*}
  \partial_x^2v_\sigma= {^{(p)}\partial_x^2}v_\sigma \text{ in }L^\infty(\Om).
\end{align*}
\end{lemma}

\begin{proof}

Throughout the proof we abbreviate $u:=v_\sigma$. Recall that for fixed time this function is semi-convex, semi-concave (uniformly in $t$) and, thus, by Lemma~\ref{l:c11}, of the class $C^{1,1}(\bar J)$ (uniformly in $t$). 
For any $t>0$ we denote by $N_t$ the subset of points in $J$ where the second pointwise derivative of $u(t,\cdot)$ does not exist. 
Then the set $N_t$ is an  $\mathcal{L}^1$-null set, and our goal is to show that the set $\cup_t\{t\}\times N_t\subset\Om$ is $\mathcal{L}^2$-measurable. 

We choose $C$ large enough such that the function $\tilde u(t,x)=u(t,x)+\frac{C}{2}|x|^2$ is convex for all~$t$ and define $v(t,x):=\partial_x\tilde u(t,x)$. Then $v(t,\cdot)$ is non-decreasing and $v(t,\cdot)\in C^{0,1}(\bar J)$. Moreover, $v$ lies in $L^\infty(\Om)$ and is thus $\mathcal{L}^2$-measurable. Now define 
\begin{align*}
  \overline{\partial}v:=\limsup_{h\to0}\partial^hv,\qquad \underline{\partial}v:=\liminf_{h\to0}\partial^hv,
\end{align*}
where the function $\partial^hv(t,x):=\frac{v(t,x+h)-v(t,x)}{h}$ is bounded.
In view of the monotonicity and the continuity of $v(t,\cdot)$, it is  clear that in taking the $\limsup$ resp.\;the $\liminf$ one can restrict to $h=\frac{1}{n},\; n\in\mathbb{Z}$.
Since $w_n:=\partial^\frac{1}{n}v$ is $\mathcal{L}^2$-measurable, the pointwise $\limsup$ resp.\;$\liminf$ of this countable family $\{w_n\}$ must itself be $\mathcal{L}^2$-measurable.
Therefore the set 
\begin{align*}
  G:=\{\om\in\Om: \;\; \overline{\partial}v(\om)- \underline{\partial}v(\om)=0\},
\end{align*}
which is exactly the set where $^{(p)}\partial_x^2u$ exists, is $\mathcal{L}^2$-measurable. Hence its complement $\Om\setminus G = \cup_{t}\left(\{t\}\times N_t\right)$ is $\mathcal{L}^2$-measurable and thus, by Fubini's theorem, an $\mathcal{L}^2$-null set. Extending the function $^{(p)}\partial_x^2u$ defined on $G$ to $\Om$, e.g., by setting $^{(p)}\partial_x^2u(\om)=0$ for all $\om\in\Om\setminus G$, the fact that $^{(p)}\partial_x^2u(\om)=\overline{\partial}v(\om)$ for any $\om\in G$ implies that  $^{(p)}\partial_x^2u$ is $\mathcal{L}^2$-measurable, so that, thanks to the boundedness of $\overline{\partial}v$, $^{(p)}\partial_x^2u\in L^\infty(\Om)$.
 Fubini's theorem finally yields that the identity $^{(p)}\partial_x^2u=\partial_x^2u$ holds true $\mathcal{L}^2$-almost everywhere in~$\Om$.
\end{proof}

\subsection{Propagation of moments}\label{ssec:highermoments}
\begin{proof}[Proof of Lemma~\ref{l:L2bound}]
  For the proof we abbreviate $u:=u^{(R)}$, $J:=J_R=(a_R,b_R)$ and $a:=a_R,b:=b_R$.
  We first gather several observations on the regularity of the functions involved, which will justify our computations.
  The fact that the function $t\mapsto u(t,x)$ is Lipschitz continuous uniformly in $x$ combined with the results in Proposition~\ref{prop:profile} 
  implies that for each $x$ the map $t\mapsto u^2(t,x)$ is differentiable with bounded derivative.
  Furthermore, in $\{|u|>0\}$ we have
  \begin{align*}
    \frac{1}{2}\frac{\d}{\dd t}u^2 = u\partial_tu &= u(\partial_xu)^{-2}\partial_x^2u-u^2(\partial_xu)^{-\gamma}-u^2
    \\&\le -u\frac{\d}{\dd x}\left((\partial_xu)^{-1}\right)     -u^2
    \\& =-\frac{\d}{\dd x}\left(u(\partial_xu)^{-1}\right)+1-u^2.
  \end{align*}
Finally notice that, again thanks to Proposition~\ref{prop:profile}, for every $t>0$ the function 
$$-\frac{\d}{\dd x}\left(u(\partial_xu)^{-1}\right)=u(\partial_xu)^{-2}\partial_x^2u$$ is integrable in $\{|u(t,\cdot)|>0\}$ and its integral satisfies
\begin{align*}
  -\int_{(a,b)\cap\{|u(t,\cdot)|>0\}}\frac{\d}{\dd x}\left(u(\partial_xu)^{-1}\right)\,\dd x 
  &= -\lim_{\ve\to0}  \int_{(a+\ve,b-\ve)\cap\{|u(t,\cdot)|>0\}}\frac{\d}{\dd x}\left(u(\partial_xu)^{-1}\right)\,\dd x
 \\&= -\lim_{\ve\to0} \left[u(\partial_xu)^{-1}\right]_{a+\ve}^{b-\ve}
 \\&= -\frac{R}{\partial_xu(t,b)}-\frac{R}{\partial_xu(t,a)},
\end{align*}
where in the second step we used again Proposition~\ref{prop:profile} to deduce that 
$$\lim_{y\to (x_\pm(t))^\pm}\left(\frac{u(t,y)}{\partial_xu(t,y)}\right)=0.$$
Hence, the function $t\mapsto\|u(t,\cdot)\|_{L^2(a,b)}^2$ is absolutely continuous and its derivative satisfies
\begin{align*}
  \frac{1}{2}\frac{\d}{\dd t}\|u(t,\cdot)\|_{L^2(a,b)}^2 
  & =\int_{\{|u(t,\cdot)|> 0\}}u(t,x)\partial_tu(t,x)\,\dd x
  \\&\le \mathcal{L}^1(\{|u(t,\cdot)|>0\}) - \|u(t,\cdot)\|_{L^2(a,b)}^2 .
\end{align*}
Recalling the fact that, by construction, $(a,b)=(a_R,b_R)\subset (0,m)$ and $u=u^{(R)}$ with $u^{(R)}(0,\cdot)={u_0}$ in $(a_R,b_R)$, we infer the bound
\begin{align*}
  \|u^{(R)}(t,\cdot)\|_{L^2(a,b)}^2 &\le \max\{m,\|u^{(R)}(0,\cdot)\|_{L^2(a,b)}^2 \}
  \\ &\le \max\{m,\|u_0\|_{L^2(0,m)}^2 \} 
\end{align*}
for all $t\ge0$. 
\end{proof}

\end{subappendices}

\section*{Acknowledgements}
JAC was partially supported by the EPSRC grant number EP/P031587/1. KH was supported by MASDOC DTC at the University of Warwick, which is funded by the EPSRC grant EP/HO23364/1. JLR was partially supported  by the European Research Council (grant agreement no. 616797). We thank John King for several enlightening discussions about this problem.

\footnotesize
\bibliographystyle{abbrv}
\bibliography{bib_all}

\end{document}